\documentstyle[11pt,graphicx,amsfonts,epsfig]{book}

 \newtheorem{theorem}{Theorem}[section]
 \newtheorem{lemma}[theorem]{Lemma}
 \newtheorem{corollary}[theorem]{Corollary}
 \newtheorem{remark}[theorem]{Remark}
 
\newtheorem{proposition}[theorem]{Proposition}
\newtheorem{problem}[theorem]{Problem}
\newtheorem{example}[theorem]{Example}
\newtheorem{exercise}[theorem]{Exercise}
\newtheorem{definition}[theorem]{Definition}
 \newenvironment{proof}{{\it Proof:\/}}{$\Box$\vskip 0.08in}
\newcommand{\mod}{{\mbox{ mod }}}

\newtheorem{conjecture}[theorem]{Conjecture}
\newtheorem{property}[theorem]{Property}
\newtheorem{formulla}[theorem]{}

\newcommand{\sgn}{{\mbox{ sgn\ }}}

\newcommand{\lk}{{\mbox{ lk }}}
\newcommand{\skein}{{\mbox{$\sim_c$}}}
\newcommand{\crs}[1]{{\mbox{ cr(#1) }}}
\newcommand{\row}[2]{{\mbox{$#1_1,#1_2,\ldots,#1_{#2}$}}} 
\newcommand{\kwad}{\#}

\newcommand{\pct}[1]{}
\newcommand{\Psfig}[1]{{\mbox{$\ \ $}}}

\newcommand\Z{{\mathbb Z}}

\def\Label#1{\label{#1} {\tiny $\backslash${\rm label\{#1\}}}}

\def\Label#1{\label{#1} {\tiny $\backslash${\rm label\{#1\}}}}

\begin{document}
\renewcommand{\thechapter}{\Roman{chapter}}

\thispagestyle{empty}
\
\vspace{0.4in}
 \begin{center}
 {\LARGE\bf KNOTS}\\
{\bf From combinatorics of knot diagrams to combinatorial topology
based on knots}
\end{center}

\vspace*{0.2in}

\centerline{Warszawa, November 30, 1984 -- Bethesda, September 7, 2012}
\vspace*{0.2in}
\begin{center}
                      {\LARGE \bf J\'ozef H.~Przytycki}
\end{center}

\vspace*{0.2in}

\ \\
{\LARGE  LIST OF CHAPTERS}:\ \\
\ \\
{\LARGE \bf Chapter I: \ Preliminaries }\\
\ \\
{\LARGE \bf Chapter II:\ History of Knot Theory}\\
{\bf e-print: http://arxiv.org/abs/math/0703096}\\
\ \\
{\LARGE \bf Chapter III:\ Conway type invariants  of links and Kauffman's method}\\
{\bf This e-print. Chapter III starts at page 3}\\
\ \\
{\LARGE \bf Chapter IV:\  Goeritz and Seifert matrices}\\
See {\bf e-print:\  http://front.math.ucdavis.edu/0909.1118}  \\
\ \\
{\LARGE \bf Chapter V:\ Graphs and links}\\
{\bf e-print: http://arxiv.org/pdf/math.GT/0601227}\\
\ \\
{\LARGE \bf Chapter VI:\ Fox $n$-colorings, Rational moves, Lagrangian tangles
and Burnside groups}\ \\
See {\bf e-print:\  http://front.math.ucdavis.edu/1105.2238}  \\
\ \\
{\LARGE \bf Chapter VII:\ Symmetries of links}\ \\
\ \\
{\LARGE \bf Chapter VIII:\ Different links with the same
Jones type polynomials}\ \\
\ \\
{\LARGE \bf Chapter IX:\ Skein modules} \\
{\bf e-print: http://arxiv.org/pdf/math.GT/0602264}\\
\ \\
{\LARGE \bf Chapter X:\ Khovanov Homology: categorification of the Kauffman
bracket relation}\\
{\bf e-print: http://arxiv.org/pdf/math.GT/0512630 }\\ \ \\
{\LARGE \bf Appendix I.\ }\ \\
\ \\
{\LARGE \bf Appendix II.\ }\\ \ \\
{\LARGE \bf Appendix III.\ }\\
\

\ \\
{\LARGE \bf Introduction}\\
\ \\
This book is
about classical Knot Theory, that is, about
the position of a circle (a knot) or of a number of disjoint circles
(a link) in the space $R^3$ or in the sphere $S^3$.
We also venture into Knot Theory in general 3-dimensional
manifolds.

The book has its predecessor in Lecture Notes on Knot Theory,
which were published in Polish\footnote{The
Polish edition was prepared for the ``Knot Theory" mini-semester
at the Stefan Banach Center, Warsaw, Poland, July-August, 1995.}
in 1995  \cite{P-18}.
A rough translation of the Notes (by Jarek Wi\'sniewski) was
ready by the summer of 1995. It differed from the Polish edition
with the addition of
the full proof of Reidemeister's theorem. While I couldn't find
time to refine the translation and prepare the final manuscript,
I was adding new material and rewriting existing
chapters. In this way I created a new book based on the Polish
Lecture Notes
but expanded 3-fold.
Only the first part of Chapter III (formerly Chapter II),
on Conway's algebras is essentially unchanged from the Polish book 
(except new Subsection \ref{Subsection III.1.1} on Monoid of Conway algebras), 
and is based on preprints \cite{P-1}.
\ \\   ... \ \\ SEE INTRODUCTION AND CHAPTER I OF THE BOOK.

\setcounter{chapter}{2}
\newpage
\chapter{Conway type invariants of links and Kauffman's method}\label{Chapter III}
\markboth{\hfil{\sc Conway type invariants }\hfil}
{\hfil{\sc Conway type invariants}\hfil}

 \vspace{0.2in}


\section{Conway algebras}\label{III.1}

While considering quick methods of computing Alexander 
polynomial\footnote{One can find this statement ironic because the classical 
Alexander method, which uses certain determinant (see Chapter IV), 
to compute the 
 Alexander polynomial has polynomial time complexity, the method 
developed by Conway has an exponential time complexity.}
(a classical invariant of links, compare Chapter $IV$ for another approach to the Alexander polynomial), 
Conway \cite{Co-1}
suggested a normalized form of it (now called the Conway 
or Alexander-Conway polynomial)
and he showed that the polynomial, $\bigtriangleup_L(z)$,  
satisfies the following two conditions:

\begin{enumerate}
\item [(i)] (Initial condition) \ If $T_1$ is the trivial knot then
$\bigtriangleup_{T_1}(z) = 1$.

\item [(ii)] (Conway's skein relation)\ 
$\bigtriangleup_{L_+}(z) - \bigtriangleup_{L_-}(z) = z\bigtriangleup_{L_0},$ 
where $L_+, L_-$ and $L_0$ are diagrams of oriented links
which are identical except for the part presented in Fig.~1.1. 

\end{enumerate}
The conditions $(i)$ and $(ii)$ define the Conway polynomial (or, maybe
more properly, Alexander-Conway polynomial)
$\bigtriangleup_L(z)$ uniquely, see \cite{Co-1,K-1,Gi,B-M}.
Alexander used the variable $t$ in his polynomial. For $z=\sqrt{t}-
\frac{1}{\sqrt{t}}$ we obtain the normalized version of the 
Alexander polynomial. The skein relation has now the form:
$$ (ii')\ \ \ \ \ \ \ \ \ \ \ \bigtriangleup_{L_+} - \bigtriangleup_{L_-} = 
(\sqrt{t}- \frac{1}{\sqrt{t}})\bigtriangleup_{L_0}.$$
In fact the un-normalized version of the
 formula $(ii')$ was noted by J.W.~Alexander in 
his original paper
introducing the polynomial \cite{Al-3}, in 1928. Alexander 
polynomial was defined up to invertible elements, $\pm t^i$,
 in the ring of Laurent polynomials, $Z[t^{\pm 1}]$), 
so the formula (ii) was not easily
available for a computation of the polynomial.

In May of 1984 V.~Jones, \cite{Jo-1,Jo-2},
showed that there exists an invariant $V$ of links which is a Laurent 
polynomial with respect to the variable $\sqrt{t}$ which satisfies
the following conditions:

\begin{enumerate}
\item[(i)]
$V_{T_1}(t) = 1,$

\item[(ii)]
${\displaystyle\frac{1}{t}V_{L_+}(t) - tV_{L_-}(t) = (\sqrt{t} - 
\frac{1}{\sqrt{t}})V_{L_0}(t).}$
\end{enumerate}

These two examples of invariants were a base for an idea
that there exists an invariant (of ambient isotopy) of 
oriented links which is a Laurent polynomial 
$P_L(x,y)$ of two variables\footnote{Most of us started from a polynomial of 3-variables 
$P_L(x,y,z)\in {\Z}[x^{\pm 1},y^{\pm 1},z]$ and then assuming, with only partial 
justification that $z$ can be assumed to be invertible, e.g. $z=1$.} 
and which satisfies the following conditions:

\begin{enumerate}
\item[(i)]
$P_{T_1}(x,y) = 1$

\item[(ii)]
$xP_{L_+}(x,y) + yP_{L_-}(x,y) = P_{L_0}(x,y).$
\end{enumerate}

Indeed, such an invariant exists and it was discovered a few months after
the Jones polynomial, in July-September of 1984, by four groups of
mathematicians: 
R.~Lickorish and K.~Millett, 
J.~Hoste, A.~Ocneanu as well as by  P.~Freyd and D.~Yetter \cite{FYHLMO}.
Independently, it was discovered in November-December\footnote{In fact, I had 
to stop thinking for few days on the idea of the proof of existence of 
generalized Jones polynomial because I had to prepare, by the end of November 
of 1984, the syllabus for (an early version) of this book.} 
of 1984 by J.~Przytycki and P.~Traczyk \cite{P-T-1}). 
We call this polynomial the Homflypt or Jones-Conway
 polynomial\footnote{HOMFLYPT (or, as I prefer to write: Homflypt) is the acronym after the initials
of the inventors: Hoste,
Ocneanu, Millett, Freyd, Lickorish, Yetter, Przytycki, and Traczyk.
We note also some other names that
are used for this invariant:
FLYPMOTH, HOMFLY, the generalized Jones polynomial,
two variable Jones polynomial,
twisted Alexander polynomial and skein-polynomial.}.
\ \\
\ \\
\begin{center}
\begin{tabular}{c} 
\includegraphics[trim=0mm 0mm 0mm 0mm, width=.5\linewidth]
{L+L-L0.eps}\\
\end{tabular}
\\
Fig.~1.1
\end{center}

Instead of looking for polynomial invariants of links related to Fig.~1.1
we can approach the problem from a more general point of view.
Namely, we can look for universal invariants of links which 
have the following property: 
a given value of the invariant for $L_+$ and $L_0$ determines
the value of the invariant for $L_-$, and similarly: 
if we know the value of the invariant 
for $L_-$ and $L_0$ we can find its value for $L_+$.
The invariants with this 
property are called Conway type invariants.
We will develop these ideas in the present chapter of the book
which is based mainly on a joint work of Traczyk and the author 
\cite{P-T-1,P-1}.

Let us consider the following general situation involving an 
abstract algebra $\cal A$; in our setting a set (called universe)
$A$ together with countable number of 0-argument operations (fixed elements)
$a_1,a_2,\ldots,a_n,\ldots$ 
and two 2-argument operations $\ |\ $ and
$\ \star\ $. We would like to construct an invariant $w$ of oriented links
with values in $A$
which satisfies the following conditions:

\begin{eqnarray*}
w_{L_+}&=&w_{L_-} | w_{L_0} \mbox{ \ and}\\
w_{L_-}&=& w_{L_+} \star w_{L_0} \mbox{\ and}\\
w_{T_n}&=& a_n\\
\end{eqnarray*}
where $T_n$ is a trivial link of $n$ components.

The operation $\ |\ $ is meant to recover values of the invariant
$w$ for $L_+$ from its values for $L_-$ and $L_0$ 
while the operation $\ *\ $ is supposed to recover values of $w$
for $L_-$ from its values for $L_+$ and $L_0$. 

\begin{definition}
We say that ${\cal A} = (A;a_1,a_2,\ldots,|,\star)$ is 
a Conway algebra if the following conditions are satisfied:

$$
\left.
\begin{array}{llll}
\mbox{C1} & a_n| a_{n+1} &=& a_n\\
\mbox{C2} & a_n \star a_{n+1}& =& a_n\\
\end{array}
\right\} 
\mbox{ initial values properties}
$$

$$
\left.
\begin{array}{llll}
\mbox{C3} & (a| b)| (c| d)&=&(a|
c)|(b| d) \\
\mbox{C4} &(a| b)\star(c| d)&=&(a\star
c)|(b\star d) \\
\mbox{C5} &(a\star b)\star(c\star d)&=&(a\star c)\star(b\star d) \\
\end{array}
\right\} 
\mbox{ transposition or entropy properties}
$$

$$
\left.
\begin{array}{llll}
\mbox{C6} &(a| b)\star b &=& a\mbox{\ \ \ \ \ \ \ } \\
\mbox{C7} &(a\star b)| b &=& a \\
\end{array}  
\right\} 
\mbox{ inversion properties.}
$$
\end{definition}

We will prove the following theorem which is the main result
of this chapter.

\begin{theorem}\label{2:1.2}
For a given Conway algebra ${\cal A}$ there exists a uniquely
determined invariant of oriented links $w$ which to any class $L$ of 
ambient isotopy 
of links associates an element $w_L\in A$ and satisfies the following
conditions:
$$
\begin{array}{lll}
(1)& w_{T_n} = a_n&\mbox{ -- initial conditions}\\
\end{array}$$
$$\left.
\begin{array}{lll}
(2)&w_{L_+} = w_{L_-}| w_{L_0}\\
(3)&w_{L_-} = w_{L_+}\star w_{L_0}\\
\end{array}
\right\} \mbox{ -- Conway relations}
$$
\end{theorem}

The theorem will be proved in the next section.

Now we briefly discuss geometric interpretation of conditions
$C1-C7$  in the definition of Conway algebra.
Conditions $C1$ and $C2$ are  reflecting 
relations between trivial links of $n$ and $n+1$ components.
The diagrams of the links, which are in these relations, are
pictured in Fig.~1.2.

\begin{center}
\begin{tabular}{c} 
\includegraphics[trim=0mm 0mm 0mm 0mm, width=.5\linewidth]
{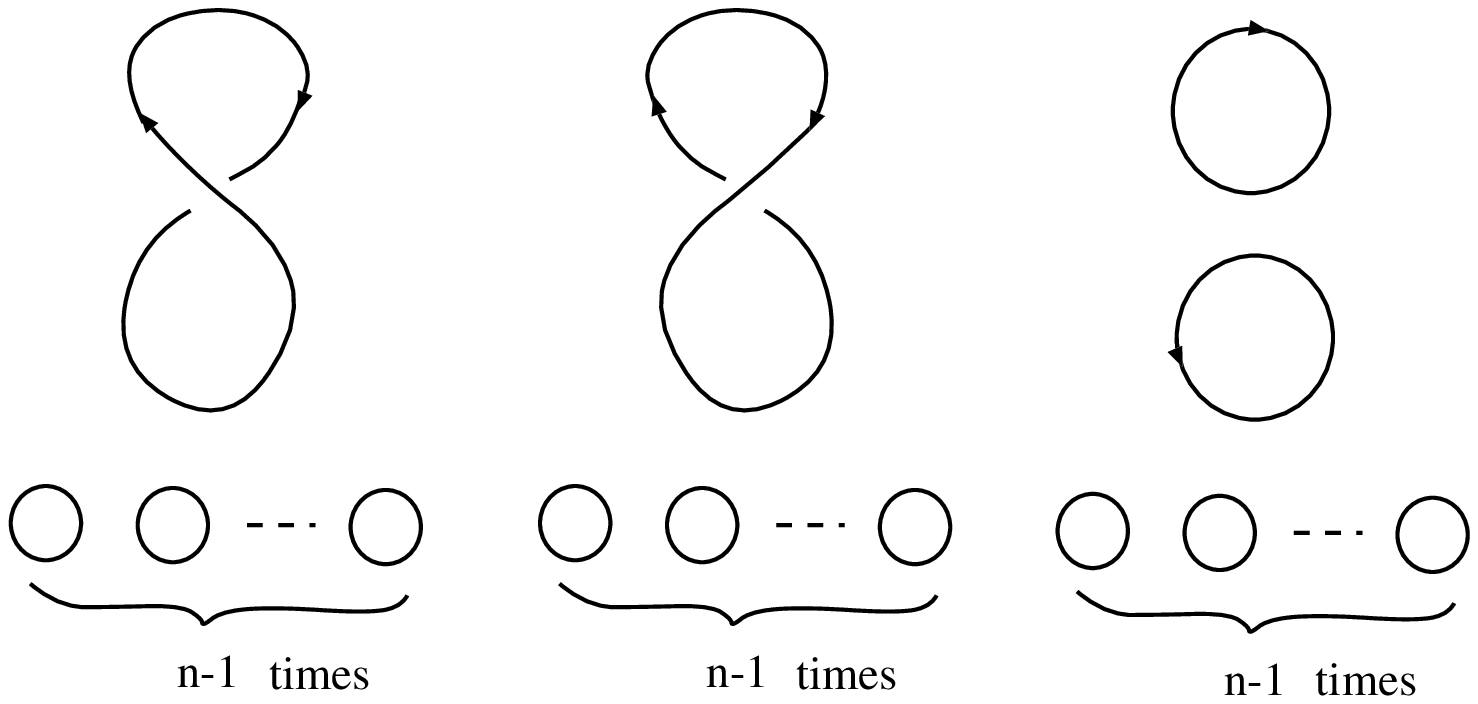}\\
\end{tabular}
\\
Fig.~1.2
\end{center}

Relations  $C3$, $C4$ and $C5$ are obtained when we perform a calculation of a link invariant at 
 two crossings of the diagram in  different order.
These relations will become apparent in Section 2. 

Relations $C6$ and $C7$ illustrate
the fact that we need the operations $|$ and $\star$
to be opposite in some sense (see Lemma \ref{Lemma III:1.4}(a) and 
Section \ref{Section III:2}).

Before giving examples (models) of Conway algebras and proving the main 
theorem (in Section 2), we show some elementary properties of Conway algebras.
In the definition of a Conway algebra we have introduced seven conditions.
It was mainly because of aesthetic and practical reasons 
(we wanted to display the symmetry between the two relations).
These conditions, however, are not independent one from another:

\begin{lemma}\label{Lemma III.1.3}
There are the following dependencies between conditions
$C1-C7$ in the definition of the Conway algebra. 
$$
\begin{array}{cccc}
(a)& C1\mbox{ and }C6 &\Rightarrow&C2 \\
(b)& C2\mbox{ and }C7 &\Rightarrow&C1 \\
(c)& C6\mbox{ and } C4&\Rightarrow&C7 \\
(d)& C7\mbox{ and }C4 &\Rightarrow&C6 \\
(e)& C6\mbox{ and }C4 &\Rightarrow&C5 \\
(f)& C7\mbox{ and }C4 &\Rightarrow&C3 \\
(g)& C5, C6\mbox{ and } C7&\Rightarrow&C4 \\
(h)& C3,C6\mbox{ and } C7 &\Rightarrow&C4 \\
\end{array}
$$
\end{lemma}
We will prove, as examples, the implications (a), (c), (e), and (g)\\
$\begin{array}{lll}
(a)&C1&\Leftrightarrow\\
&a_n| a_{n+1} = a_n&\Rightarrow\\
&(a_n| a_{n+1})\star a_{n+1} = a_n\star
a_{n+1}&\stackrel{C6}{\Rightarrow}\\
&a_n = a_n\star a_{n+1}&\Leftrightarrow\\
&C2.&\\
(c)&C6&\Rightarrow\\
&(a|(b| a))\star(b| a) =
a&\stackrel{C4}{\Leftrightarrow}\\
 &(a\star b)|((b| a)\star a) =
 a&\stackrel{C6}{\Rightarrow}\\
 &(a\star b)| b = a&\Leftrightarrow\\
 &C7.&\\

\end{array}
$\\ \ \\

(e) and (g):\\
$\begin{array}{ccc}
&C5\ \ \Leftrightarrow \ \ \ 
(a\star b)\star(c\star d) = (a\star c)\star(b\star d)&\\
& \Downarrow C7\mbox{ \ \ \ \  } \Uparrow C6 
&\\
&(a\star b) = ((a\star c)\star(b\star d))|(c\star d)&\\
&C7\Updownarrow\mbox{ or }\Updownarrow C6\mbox{ and } C4&\\
&\ \ [((a\star c)| c)\star((b\star d)| d)] = ((a\star
c)\star(b\star d))|(c\star d)&\\
&\mbox{ we substitute }
\begin{array}{cccc}
a = x| c & \Downarrow C6& x= a*c & \Uparrow  \\
b = y| d & \Downarrow C6& y= b*d & \Uparrow \\
\end{array}&\\
&\ \ \ \ \ \ \ \ \ \ \ \ \ \ \ \ \ \ \ \ \ \ \ \ \ \ \ \ \ \ \ \ \ \ 
(x| c)\star(y| d) = (x\star y)|(c\star d)\ \ \Leftrightarrow\ \ \
C4.\\
\end{array}$

An interesting example of identity in Conway algebras is the following 
equality suggested by P.~Traczyk:
$$((a|b)*c)|d=((a|d)*c)|b$$
The following short proof of the identity is by  C.~Bowszyc:\\
$$((a|b)*c)|d \stackrel{C6}{=} ((a|b)*c)|((d|e)*e)  \stackrel{C4}{=} ((a|b)|(d|e))*(c|e)  \stackrel{C3}{=} $$
$$((a|d)|(b|e))*(c|e)  \stackrel{C4}{=} ((a|d)*c)|((b|e)*e)  \stackrel{C6}{=}((a|d)*c)|b.$$

\begin{lemma}\label{Lemma III:1.4}
\begin{enumerate}
\item[(a)]
Let ${\cal A}$ be a Conway algebra. For any $b\in A$ let us consider a map
$|_b: A\rightarrow A$ (respectively, $\star_b: A\rightarrow A$) 
defined by $|_b(a)= a| b$ (respectively, $\star_b(a)= a\star b$).
Then $|_b$ and  $\star_b$ are bijections on $A$ and one is the inverse
of the other, that is: $|_b\cdot\star_b = \star_b\cdot|_b = \mbox{ Id}$.
If $| =*$, we say that the Conway algebra in involutive.
\item[(b)] One can give equivalent definition of a Conway algebra 
using only one 2-argument operation, say $|$.
The axioms are as follows:\\
(i) $a_n|a_{n+1} = a_n$,\\
(ii) The map $|_b: A \to A$ is a bijection,\\
(iii) $(a|b)|(c|d) = (a|c)|(b|d)$.
\end{enumerate}
\end{lemma}
\begin{proof}
Lemma 1.4(a)
follows from conditions $C6$ and $C7$.\\
Lemma 1.4(b) can be derived from Lemma 1.3. Let us show, for example 
that (iii) (i.e. $C3$) and (ii) (from which $C6$ and $C7$ follow 
immediately), imply $C4$. Namely,\\
By (iii) we have $((a*b)|b))|((c*d)|d) = ((a*b)|(c*d))|(b|d)$ 
and by $C7$ $((a*b)|b))|((c*d)|d) = a*c$, and by $C6$
$(((a*b)|(c*d))|(b|d))*(b|d) = (a*b)|(c*d)$. Combining it together we 
obtain $C4$: $(a|c)*(b|d) = (a*b)|(c*d)$.
In fact assuming $C6$ and $C7$ we have that $C3$, $C4$ and $C5$ 
are equivalent. For example $C4 \Rightarrow C3$:\\
By $C4$ we have $((a|b)*b)|((c|d)*d)= ((a|b)|(c|d))*(b|d)$ or 
equivalently\\
 $$(((a|b)*b)|((c|d)*d))|(b|d)= (((a|b)|(c|d))*(b|d))|(b|d).$$ 
This formula reduces, by $C6$ and $C7$ to $(a|c)|(b|d)= (a|b)|(c|d)$, as 
needed.
\end{proof}
\medskip

Now, let us discuss some examples of Conway algebras.

\begin{example}[Number of components]\label{Example III:1.5}
Let $A=N$ be the set of positive integers. We define $a_i=i$ and $i| j =
i\star j = i$.

Verification of conditions $C1-C7$ is immediate
(note that the first letter of each side of every relation is the same).

The invariant of a link defined by this algebra
(it exists according to Theorem III:1.2)
is equal to the number of components of the link.
\end{example}

\begin{example}\label{Example III:1.6}
Let us set $A={\Z}_3 = \{0,1,2\}$, $a_i\equiv i \mbox{ mod } 3$, 
and $a|b=a*b \equiv 1-a-b \mod 3$. In other words $\ | \ $ and $\ *\ $ are 
 both given by the following symmetric table:
\begin{center}
\begin{tabular}{c|ccc}
$|$&0&1&2\\ \hline
0&1&0&2\\
1&0&2&1\\
2&2&1&0\\
\end{tabular}
\end{center}
The invariant defined by this algebra distinguishes 
the trefoil knot from the trivial knot
(see Fig.~1.3.).

\begin{center}
\begin{tabular}{c} 
\includegraphics[trim=0mm 0mm 0mm 0mm, width=.5\linewidth]
{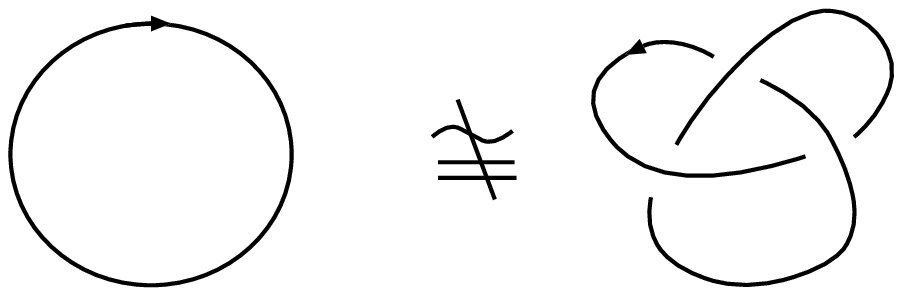}\\
\end{tabular}
\\
Fig.~1.3; For the trivial knot the value of the invariant is: $a_1=1$.
For the left-handed trefoil we have the value of the
invariant: $a_1*(a_2*a_1)=2$.
\end{center} 

\end{example}
The direct generalization of Example \ref{Example III:1.6} can be obtained by taking
$A={\Z}_n$, $a_i\equiv i \mbox{ mod } n$, and $a|b=a*b \equiv 2b-a-2 \mod n$.
\begin{example}\label{p:1.7}
Let us consider the universe  $A=\{ 0,1,2,s\}$ with the distinguished
elements $a_i\equiv i \mbox{ mod } 3$.
The operations  $\ |\ $ and $\ \star\ $ are given by the following tables
(notice that $\ |\ $ is, but $\ \star\ $ is not, a symmetric operation; 
for example $a_i*a_{i+1}=a_i$, but $a_{i+1}*a_i=s$).
$$
\mbox{
\begin{tabular}{c|cccc}
$|$&0&1&2&s\\ \hline
0&s&0&2&1\\
1&0&s&1&2\\
2&2&1&s&0\\
s&1&2&0&s\\
\end{tabular}
}\ \ \ \ \ \ \ 
\mbox{
\begin{tabular}{c|cccc}
$\star$&0&1&2&s\\ \hline
0&1&0&s&2\\
1&s&2&1&0\\
2&2&s&0&1\\
s&0&1&2&s\\
\end{tabular}
}
$$

The invariant defined by this algebra
distinguishes right-handed trefoil knot from the left-handed trefoil
(see Fig.~1.4).

\begin{center}
\begin{tabular}{c} 
\includegraphics[trim=0mm 0mm 0mm 0mm, width=.5\linewidth]
{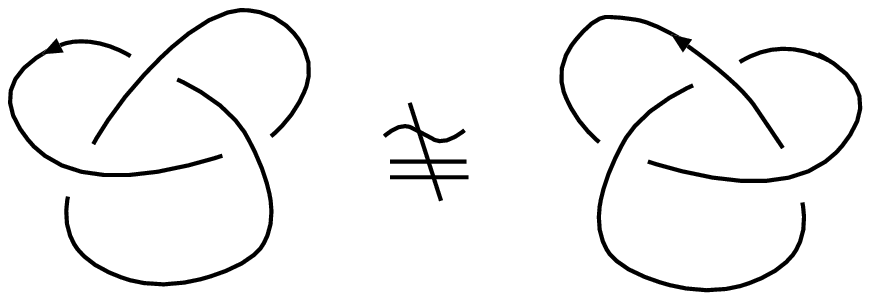}\\
\end{tabular}
\\
Fig.~1.4; For the left handed trefoil we have the value of the 
invariant: $a_1*(a_2*a_1)=0$. 
For the right-handed trefoil we have the value of the
invariant: $a_1|(a_2|a_1)=s$.
\end{center}
\end{example}

Using a computer, T.~Przytycka found all Conway algebras 
with up to five elements.  If we assume additionally that $a_1=1, a_2=2$, 
then (up to isomorphisms) we get\footnote{E. Biendzio, student of 
P.Traczyk, found in her master 
degree thesis all 6-element Conway algebras \cite{Bie}.}:

\begin{tabular}{lcccc}
The number of elements in a Conway algebra&2&3&4&5\\
The number of isomorphism classes of Conway algebras&2&9&51&204\\
\end{tabular}

\begin{example}[Homflypt (Jones-Conway) polynomial]\label{p:1.8}
We set:
$$A = \Z\left[ x^{\pm 1}, y^{\pm 1}\right],\ a_1 = 1,\
a_2=x+y,\ldots,a_i=(x+y)^{i-1},\ldots$$

The operations  $\ |\ $ and $\ \star\ $ are defined as follows:
$w_2| w_0 = w_1$ and $w_1\star w_0 = w_2$, where the polynomials
$w_1,\ w_2,\ w_0$ satisfy the equation:

\begin{formulla}\label{III:1.9}
$$xw_1 +yw_2 = w_0.$$
\end{formulla}

The invariant of links defined by this algebra is the Jones-Conway 
(Homflypt) polynomial
which we have mentioned at the beginning of this section.
In particular, if we substitute $x = 1/z$ and $y=-1/z$, then we get
the Conway polynomial, and after the substitution
$${\displaystyle x = \frac{-t}{\sqrt{t}-\frac{1}{\sqrt{t}}}, \ \ 
y = \frac{1}{t}\frac{1}{\sqrt{t}-\frac{1}{\sqrt{t}}},}$$
we obtain the Jones polynomial.

\end{example}
 
Now we shall prove that the algebra defined in Example~\ref{p:1.8} 
is a Conway algebra.

First we note that conditions $C1$ and $C2$ follow from 
the identity
$$x(x+y)^{n-1}+y(x+y)^{n-1} = (x+y)^n.$$
Next, conditions $C6$ and $C7$ follow from the fact that the operations $\ | $ 
and $\ \star\ $ were defined by the linear equation 1.9. 

We prove the condition $C3$ and the conditions $C4$ and $C5$ follow by 
Lemma 1.3.

We get:
\begin{formulla}\label{*}
\begin{eqnarray*}
&(a|b)|(c|d) &=\\
& \frac{1}{x}((c|d)-y(a|b))& =\\
&\frac{1}{x}(\frac{1}{x}(d-yc)-y\frac{1}{x}(b-ya))& =\\
&\frac{1}{x^2}d - \frac{y}{x^2}c -\frac{y}{x^2}b+\frac{y^2}{x^2}a,&\\
\end{eqnarray*}
\end{formulla}
and thus, because the coefficients of $b$ and $c$ are equal,
it follows that we can interchange $b$ and $c$ in the formula,
which proves the relation $C3$.

One may generalize the algebra from Example 1.8
by introducing a new variable $z$ and considering
instead of Equation 1.9
the equation
$$xw_1 + yw_2 = w_0 - z.$$
It turns out, however, that the invariant obtained this way is
not stronger than the Jones-Conway polynomial (see Exercise~\ref{Exercise III:3.43}).

\begin{example}[Global linking number]\label{p:1.10}
Let us set $A= N\times \Z$ and $a_i = (i,0)$, and moreover
$$
(a,b)|(c,d)=\left\{
\begin{array}{lcl}
(a,b+1)&\mbox{if}&a>c\\
(a,b)&\mbox{if}&a\leq c\\
\end{array}
\right.
$$

$$
(a,b)\star(c,d)=\left\{
\begin{array}{lcl}
(a,b-1)&\mbox{if}&a>c\\
(a,b)&\mbox{if}&a\leq c\\
\end{array}
\right.
$$

The invariant defined by this algebra is a pair, the first entry of which 
is the number of components of the link and the second entry is called
the global linking number (or index) (see Exercise 1.12 and Chapter IV).
\end{example}

Notice that we have $(a,b)|(a,b)=(a,b)=(a,b)*(a,b)$ that is idempotency condition holds, and that
idempotency condition and Conway (entropy) relations $C3-C5$ lead to distributivity (left and 
right distributivity).
For example, if we put $b=d$ in C4 we get:
$$(x|y)*(z|y) \stackrel{entr}{=} (x*z)|(y*y) \stackrel{idem}{=}  
(x*z)|y; \textrm{ right distributivity of $|$ with respect to $*$ }.$$
Such magmas $(A;|)$, satisfying idempotency condition, invertibility, and right distributivity are called Quandles
and the conditions reflect the Reidemeister moves \cite{Joy}. 
If idempotency conditions are not assumed, these magmas are called racks (introduced by 
Conway and Wraith in 1959 \cite{C-W}), and if only right self-distributivity 
is kept they are called right shelves, or just shelves 
(the word coined by Alissa Crans \cite{Cr}; compare \cite{P-40}).

Now we shall prove that the algebra from Example 1.11 
is a Conway algebra.

The proof of conditions $C1, C2, C6$ and $C7$ is not hard.
We will check condition $C3$ in more detail.
Because of our definition of the operation 
$\ |\ $ in ${\cal A}$ we get:

$$((a_1,a_2)|(b_1,b_2))|((c_1,c_2)|(d_1,d_2)) = \left\{
\begin{array}{cccc}
(a_1,a_2+2)&\mbox{if}&a_1>b_1&\mbox{and}\\
&&a_1>c_1&\\
(a_1,a_2+1)&\mbox{if}&a_1>b_1&\mbox{and}\\
&&a_1\leq c_1&\\
&\mbox{or}&a_1\leq b_1&\mbox{and}\\
&&a_1>c_1&\\
(a_1,a_2)&\mbox{if}&a_1\leq b_1&\mbox{and}\\
&&a_1\leq c_1&\\
\end{array}
\right.
$$

Now, if we change the positions of $b_1$ and $c_1$ then the result
will be the same. Therefore the relation $C3$ is satisfied.

The global linking number of an oriented link can be read directly from 
a diagram of a link. Namely, let us call a crossing of type
\includegraphics[trim=0mm 0mm 0mm 0mm, width=.02\linewidth]
{L+maly.eps} positive and a crossing of type
{\includegraphics[trim=0mm 0mm 0mm 0mm, width=.02\linewidth]
{L-maly.eps}} negative. 
We will write  $\mbox{sgn} p =+$ or $-$ depending on whether the
crossing $p$ is positive or negative.

\begin{exercise}\label{Exercise III:1.12}
Suppose that $D$ is a diagram of an oriented link.
Let us define 
$\mbox{lk}(D) = \frac{1}{2}\sum\mbox{sgn}(p)$, where the sum is taken over
all crossings $p$, between different components of the link.
Show that $\mbox{lk}(D)$ is equal to the global linking number
of the link.

Hint. Let us note that, if $L_+, L_-, L_0$ are diagrams
of links as on the Fig.~1.1, then 
$$\mbox{lk}(L_+) =
\left\{
\begin{array}{cl}
\mbox{lk}(L_-)+1&\mbox{if the crossing involves two 
different} \\ & \mbox{components of the link}\\
\mbox{lk}(L_+)&\mbox{if the crossing involves only one} \\ 
& \mbox{component of the link}
\end{array}
\right.
$$
Moreover, note that the number of components of $L_0$ is smaller by one than 
the number of components of $L_+$ if the crossing involves two different
components and it is bigger by one otherwise.
\end{exercise}

\begin{exercise}\label{Exercise III:1.13}
Consider the universe $A=N\times B$ for any set $B$. Let $f_{i,j}:B \to B$ be a bijection for any 
pair of natural numbers, and assume that for fixed $i$ the functions commute (i.e. $f_{i,j}f_{i,k}=f_{i,k}f_{i,j}$.
Check that $(A;|)$ with $(a_1,a_2)|(b_1,b_2)=(a_1,f_{a_1,b_1}(a_2))$ forms an entropic system 
with an invertible $|$. If we further assume that $f_{i,i+1}=Id$ then for a fixed $b_0\in B$, 
 $(A;a_1,a_2,...,|)$ with $a_i=(i,b_0)$ forms a Conway algebra.
\end{exercise}
We will write $L^p_+, L^p_-$ and $L^p_0$ instead of $L_+,L_-$ and $L_0$ 
if we want the crossing point $p$ to be explicitly specified.

\begin{definition}\label{Definition III.1.13}
Let $T$ be a binary tree of links, that is a binary tree 
each of whose vertices represents a link.
We draw $T$ with the  root at the top. The root represents 
the given link $L$.  The leaves in the bottom (vertices of degree one),
represent trivial links. Moreover we assume that
at each vertex which is not a leaf the situation is as follows:

\begin{center}
\begin{tabular}{c} 
\includegraphics[trim=0mm 0mm 0mm 0mm, width=.3\linewidth]
{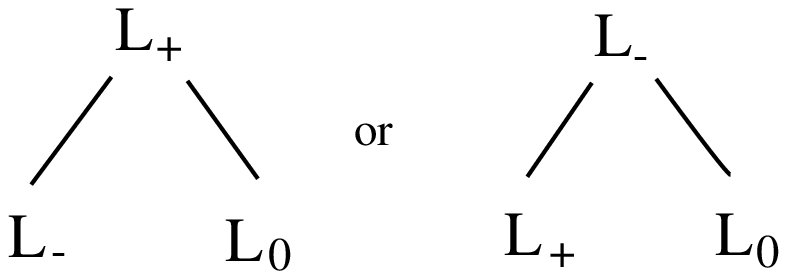}\\
\end{tabular}
\\
Fig.~1.5
\end{center}

The tree $T$ yields, in a natural way, an associated
tree with $a_i$ at each of leaves and signs $+$ or $-$ 
at any other vertex.
We call it a resolving tree of the link $L$ (see Fig. 1.7 for an example); compare \cite{F-M}.
\end{definition}

There exists a standard procedure for resolving a link.
It will be described in the subsequent section
and it will play an essential role in the proof of Theorem~\ref{2:1.2}.

\begin{example}\label{p:1.13}
Let $L$ be the figure-eight knot represented by the diagram on Fig.~1.6a.

\begin{center}
\begin{tabular}{c} 
\includegraphics[trim=0mm 0mm 0mm 0mm, width=.25\linewidth]
{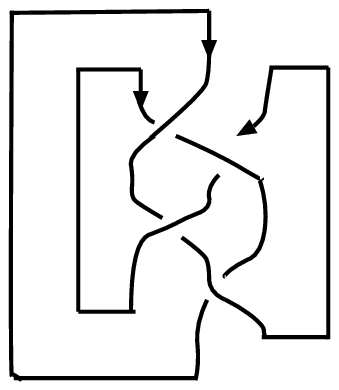}\\
\end{tabular}
\\
Fig.~1.6a
\end{center}

In order to determine $w_L$ (i.e.~the value of an invariant $w$
associated to a Conway algebra ${\cal A}$
on the link $L$) let us consider the following binary tree of 
links:
\begin{center}
\begin{tabular}{c} 
\includegraphics[trim=0mm 0mm 0mm 0mm, width=.5\linewidth]
{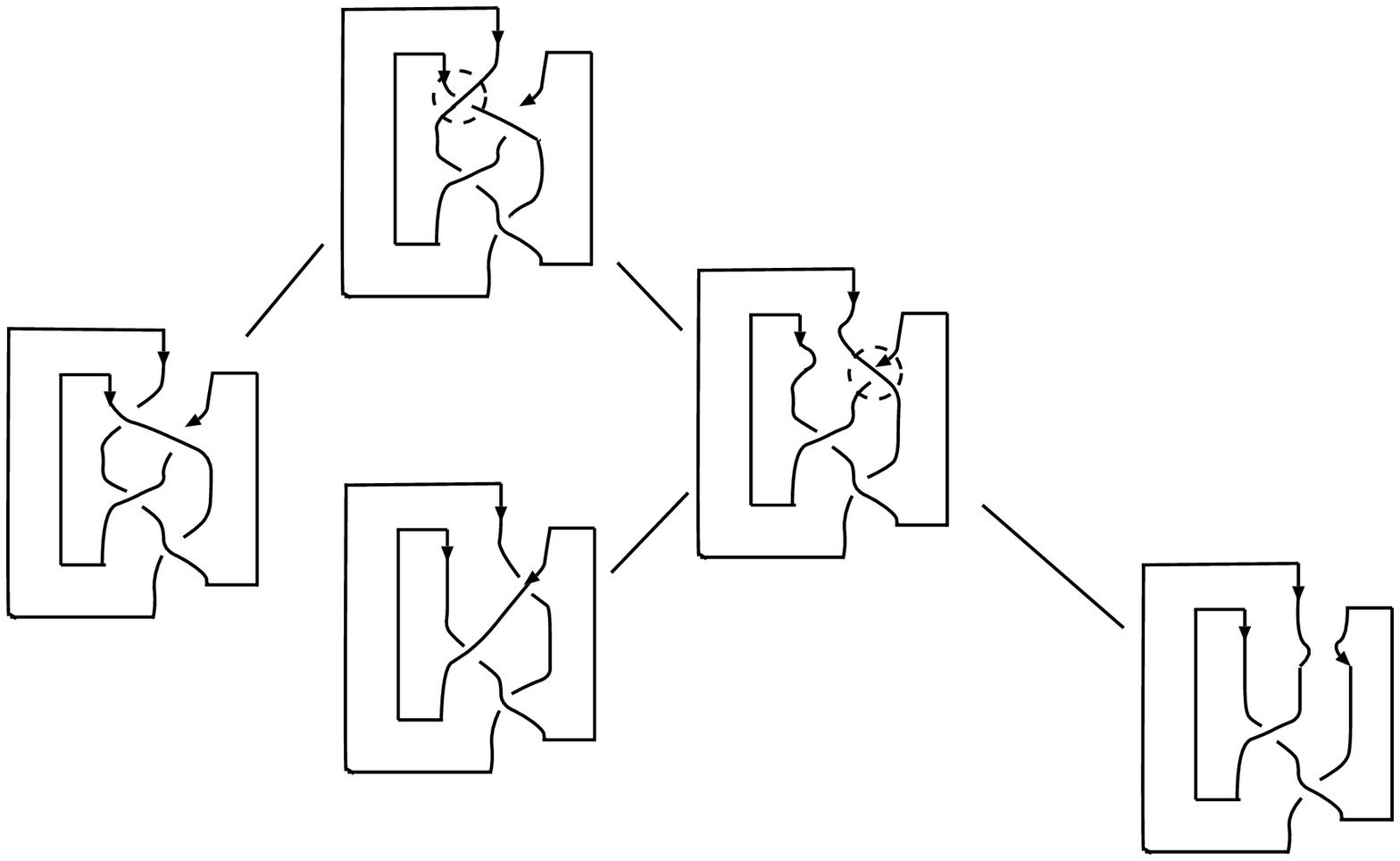}\\
\end{tabular}
\\
Fig.~1.6b
\end{center}
 
It can be easily seen that leaves of this tree
represent trivial links and the branching vertices
of the tree are related to the admissible operations on the diagram
of the tree at each of the marked crossing points.
The above tree defines the following resolving tree for the figure-eight
knot:

\begin{center}
\begin{tabular}{c} 
\includegraphics[trim=0mm 0mm 0mm 0mm, width=.15\linewidth]
{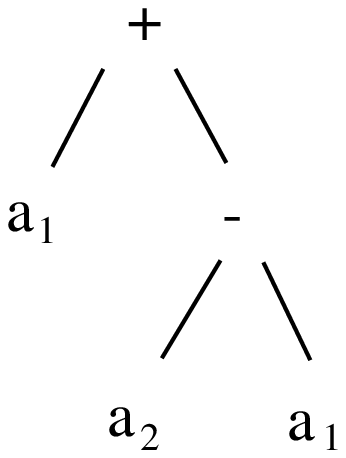}\\
\end{tabular}
\\
Fig.~1.7
\end{center}

Applying this tree we get $w_L = a_1|(a_2\star a_1)$.
\end{example}

\begin{exercise}\label{Exercise III.1.15}
\ \\
\begin{enumerate}
\item Show that the knot pictured on Fig.~1.5(b) is ambient 
isotopic to the one pictured on Fig.~I:0.1.

\item Show that the figure-eight knot is ambient isotopic
with its mirror image, that is

\begin{center}
\begin{tabular}{c} 
\includegraphics[trim=0mm 0mm 0mm 0mm, width=.5\linewidth]
{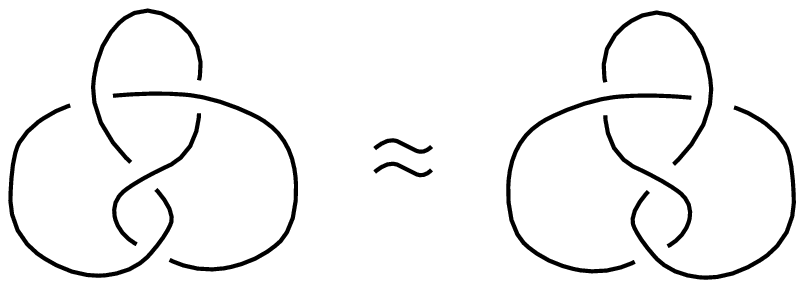}\\
\end{tabular}
\\
Fig.~1.8
\end{center}

\item Prove that the following identity holds in any Conway algebra
$$a_1\star(a_2|a_1)=a_1|(a_2\star a_1).$$ 
(This is an algebraic version of the property (2) of the figure-eight knot,
c.f.~Example \ref{p:1.13}.)

\item Prove that if $*$ is an invertible binary operation and entropy property holds 
(that is C3, C6, and C7) and there is $c$ such that $a*c=a$, then
$$(a*b)|a= a|(b*a).$$
\item Compute values of invariants from Examples
1.6, \ref{p:1.7}, \ref{p:1.8} and \ref{p:1.10} 
for the figure eight knot.
\end{enumerate}
\end{exercise}

\begin{exercise}
Let $K$ be a left handed trefoil knot and let
$\overline{K}$  be a right handed trefoil knot.
Prove that
$w_K = a_1\star (a_2\star a_1)$ and $w_{\overline{K}} =
a_1|(a_2|a_1)$ (c.f.~Fig.~1.4).
\end{exercise}

\begin{exercise}
Show that the figure-eight knot is isotopic neither with the
trefoil knot (left- and right-handed) nor with the trivial knot.
\end{exercise}

For a long boring evening we suggest the following exercise
(c.f.~Example \ref{p:1.7}). 

\begin{exercise}
Prove that there does not exist any three element Conway algebra
which distinguishes the right-handed trefoil knot from the left-handed
trefoil knot.
\end{exercise}

\subsection{Monoid of Conway algebras}\label{Subsection III.1.1} 

Recall that a {\it magma} $(X;*)$ is a set $X$ with a binary operation $*:X\times X \to X$.
For any $b\in X$ the adjoint map $*_b: X\to X$, is defined by $*_b(a)=a*b$.
Let $Bin(X)$ be the set of all binary operations on $X$.
\begin{proposition}\cite{P-40}\label{Proposition III:1.19} $Bin(X)$ is a monoid
(i.e. semigroup with identity) with  composition
$*_1*_2$ given by $a*_1*_2b= (a*_1b)*_2b$, and the identity $*_0$ being the right trivial operation, that is,
$a*_0b=a$ for any $a,b\in X$.
\end{proposition}
\begin{proof} Associativity follows from the fact that adjoint maps $*_b$
compose in an associative way, $(*_3)_b((*_2)_b(*_1)_b) = ((*_3)_b(*_2)_b)(*_1)_b$; we can write
directly: $a(*_1*_2)*_3b= ((a*_1b)*_2b)*_3b = (a*_1b)(*_2*_3)b= a*_1(*_2*_3)b$.
\end{proof}
The submonoid of $Bin(X)$ of all invertible elements in $Bin(X)$ is a group denoted by $Bin_{inv}(X)$.
If $* \in Bin_{inv}(X)$ then $*^{-1}$ is usually denoted by $\bar *$, however in the case of 
Conway algebras the inverse of $*$ was denoted by $|$. Consider the set $a\circ c = \{b\in X\ | \ a*b=c\}$.
If $*$ is invertible and $a\circ c$ has always exactly one element, we call the magma $(X;*)$ a quasigroup.
If $X$ is finite then the multiplication table for $*$ has a permutation in each row and column.
Such a table is called a latin square -- the objects studied for hundred of years (e.g. by Euler).


Probably Sushkevich\footnote{Anton Kazimirovich Sushkevich (1889-1961) a Ukrainian-Russian-Polish
 mathematician who spent most of his working life at Kharkov State University in the Ukraine. 
Ph.D. Voronezh State University 1922.
In the 1920s, he embarked upon the first systematic study of semigroups, placing him at the very 
beginning of algebraic semigroup theory and, arguably, earning him the title of the world's first semigroup theorist. 
Owing to the political circumstances under which he lived, however, his work failed to find a wide 
audience during his lifetime \cite{Hol}.} in 1937 \cite{Su} was the first to consider a binary operation satisfying 
entropic property $(a*b)*(c*d)=(a*b)*(c*d)$. He was motivated by the Burstin and Mayer paper of 1929, \cite{Bu-Ma}.
Soon after, Murdoch\footnote{David Carruthers Murdoch, PhD at University of Toronto 1937, Then at Yale and 
University of British Columbia.}
in \cite{Mur} and 
Toyoda\footnote{Koshichi Toyoda, worked till 1945 at Harbin Polytechnic University, Mituhisa Takasaki was his 
assistant there. Both perished after Soviet occupation of Harbin in 1945.}
 in a series of papers \cite{Toy-1,Toy-2,Toy-3,Toy-4},   
have established main properties of such magmas and, in the case $(X;*)$ is a quasigroup,  
they proved the result name after them, 
Murdoch-Toyoda theorem (Theorem \ref{Theorem III:1.20}).
This result is very important in our search for invariants of Conway type. The name used by Toyoda was, 
associativity. The word {\it entropic property} which I use was coined in 1949 by 
I.M.H.~Etherington\footnote{ Ivor Malcolm Haddon Etherington (1908--1994), English mathematician 
who started his work in general relativity under the supervision of Prof. E. T. Whittaker, and later on moved 
into the area of non-associative algebras, where he made very important contributions in connection with genetics.
He was actively involved in helping refugees to escape from Nazi Germany in the 1930s \cite{Kra}.}\cite{Et}. 
Other names for the property are medial, alternation, bi-commutative, bisymmetric, commutative, surcommutative 
and abelian. The word entropic refers to inner turning \cite{R-S-2}.

\begin{theorem}\cite{Mur,Toy-2}\label{Theorem III:1.20} 
If $(X;*)$ is an entropic quasigroup, then $X$ has an abelian group structure 
such that $a*b= f(a) + g(b) + c$ where $f,g: X\to X$ are commuting group automorphisms (equivalently 
$(X;*)$ is a $\Z[x_f^{\pm 1},x_g^{\pm 1}]$-module).
\end{theorem} 

From this follows partially A.~Sikora result \cite{Si-2} for invariants coming from quasigroup Conway algebras 
(Sikora result is general: no invariant coming from a Conway algebra can distinguish links with the same 
Homflypt polynomial).

\begin{definition}\label{Definition III:1.21}
\begin{enumerate}
\item[(1)]
We say that a subset $S \subset Bin(X)$ is an entropic set if all
pairs of elements $*_{\alpha},*_{\beta} \in S$ (we allow  $*_{\alpha}=*_{\beta}$) satisfy entropic property, that is:
$$ (a*_{\alpha}b)*_{\beta}(c*_{\alpha}d)= (a*_{\beta}c)*_{\alpha}(b*_{\beta}d).$$
\item[(2)] If $S \subset Bin(X)$ is an entropic set, and there are given elements $a_1,a_2,...$ in $X$ such 
that for any $*\in S$ we have $a_i*a_{i+1}=a_i$ then we call $S$ a pre-Conway set of operations.
\item[(3)] If $S \subset Bin(X)$ is a pre-Conway set of invertible operations then we call $S$ a Conway set of operations.
\end{enumerate}
\end{definition}

\begin{proposition}\label{Proposition III:1.22}
\begin{enumerate}
\item[(i)] If $S$ is an entropic set and $*\in S$ is invertible, then $S\cup \{\bar *\}$ is also an
entropic set.
\item[(ii)] If $S$  is an entropic set and $M(S)$ is the monoid generated by $S$ then $M(S)$ is an
entropic monoid.
\item[(iii)] If $S$  is an entropic set of invertible operations and $G(S)$ is the group generated by $S$, then
$G(S)$ is an entropic group.
\end{enumerate}
\end{proposition}
\begin{proof} (i) We have proven already that if $*$ is a self-entropic operation then $\{\bar *\}$ 
is also  self-entropic (Lemma III:1.3 (e)(h)). We will show now, more generally, that if 
 $*,*' \in Bin(X)$ and $*$ is invertible and entropic with respect to $*'$, then
$\bar *$ is entropic with respect to $*'$. We start from 
$$(a*b)*'(c*d)= (a*'c)*(b*'d)$$
Then we substitute $x=a*b$ and $y=c*d$, or, equivalently $a=x\bar *b$ and $c=y\bar * d$ to get
$$x*'y = ((x\bar *b)*'(y\bar * d))*(b*'d) \mbox{ \ \  and eventually }$$
$$(x*'y)\bar * (b*'d)=(x\bar *b)*'(y\bar * d)$$
as needed.\\
(ii) We want to prove that any element of $M(S)$ is entropic with respect to any other element 
(including self-entropy of every element). To prove this
it suffice to show that if $*$ is entropic with respect to $*_1$ and $*_2$ then it is also
entropic with respect to $*_1*_2$.  We have
$$(a*_1*_2b)*(c*_1*_2d)\stackrel{def}{=} ((a*_1b)*_2b)*(c*_1d)*_2d)\stackrel{entr}{=}$$
$$((a*_1b)*(c*_1d))*_2(b*d))\stackrel{entr}{=}$$
$$((a*c)*_1(b*d))*_2(b*d))\stackrel{def}{=}$$
$$(a*c)*_1*_2(b*d), \mbox{ \ \  as needed}.$$
(iii) It follows directly from (i) and (ii).
\end{proof}
\begin{remark} Composition of entropic operations was considered in \cite{R-S-1} were 
it was observed, in particular Theorem 246, that if $S$ is an entropic set of 
idempotent operations then elements of $S$ commute. Namely, we have 
$a*_1*_2b=(a*_1b)*_2b \stackrel{idem}{=}  (a*_1b)*_2(b*_1b)\stackrel{entr}{=}  (a*_2b)*_1(b*_2b) 
\stackrel{idem}{=} (a*_2b)*_1b=a*_2*_1b$ as needed.
\end{remark}

\begin{proposition}
\begin{enumerate}
\item[(1)]
If $(A;*)$ satisfies entropic condition and idempotency conditions then it is (right and left) distributive.
\item[(2)] A right self-distributive quasigroup has all elements satisfying idempotent property.
\end{enumerate}
\end{proposition}
\begin{proof} (1)
We have $(a*b)*(c*b)\stackrel{entr}{=} (a*c)*(b*b) \stackrel{idem}{=} (a*c)*b$  (right self-distributivity), and\\
 $(a*b)*(a*d)\stackrel{entr}{=} (a*a)*(b*d) \stackrel{idem}{=} a*(b*d$a (left self-distributivity).\\
(2) We have $(a*a)*a=(a*a)*(a*a) \mbox{ so } a=a*a.$
\end{proof}

A interesting example of an entropic algebra was proposes by D.~Robinson in 1962 \cite{Rob,R-S-2}).
Consider a group of nilpotent class two\footnote{The smallest
nilpotent groups of class 2 are the quaternion group $Q_8$, and the dihedral group $D_{2\cdot 4}$. 
Generally a finite group of a nilpotent class $2$
are products of  an abelian group and nonabelian $p$-groups of class two. In particular, any group of order $p^3$. 
There are, up to isomorphism, two nonabelian groups of order $p^3$:
for $p=2$ they are quaternion group $Q_8$ and dihedral group $D_{2\cdot 4}$; for $p>3$ they are 
Heisenberg $3\times 3$ matrices with $\Z_p$ entries (semidirect product of $\Z_{p^2}$ by $\Z_p$), and 
the group $\{a,b\ | \ a^{p^2}= b^p=1,\ ab=ba^{p+1}\}$ (semidirect product of $\Z_{p}\oplus \Z_p$  by 
$\Z_p$). For two generator p-groups of nilpotency class two see \cite{B-Ka,AMM}.
 An infinite example is the Heisenberg $3\times 3$ matrices group with entries in $\Z$; see 
\cite{Du-Fo,Ma-Bi}.}, 
that is $[[G,G],G]=0$ where 
$[G,G]$ is generated by commutators $[g,h]=g^{-1}h^{-1}gh$.
Then $(G;*)$ with $a*b=ba^{-1}b$ is an entropic magma (this operation was introduced by Brucks and
called by him a core \cite{Bruc-2}).
Similarly if we consider on $G$ a conjugacy operation $a*b=b^{-1}ab$; see Corollary \ref{Corollary III:1.25}.

We prove the above observations in more general situation:
\begin{proposition}\label{Proposition III:1.24} 
Let $(X;*)$ be a magma with invertible $*$ and $\tau$ an involution 
on $X$. Conditions (1)-(3) below are equivalent 
and (4) is their consequence.
\begin{enumerate}
\item[(1)] $(a*b)*c= (a*\tau(c))*\tau(b)$;
\item[(2)] $(a*b)\bar *c = (a\bar * \tau(c))*\tau(b)$
\item[(3)] $(a\bar *b)\bar *c = (a \bar * \tau(c))\bar *\tau(b)$
\item[(4)] $(a*b)*(c*d)= (a*c)*(b*d)$; that is entropic property holds.
\end{enumerate} 
\end{proposition}
\begin{proof} (1)-(3) We show, for example, that (2) follows from (1) (other cases are similar).
$(a*b)\bar *c = (a\bar * \tau(c))*\tau(b)$ it is equivalent to
$a*b=(((a\bar * \tau(c))*\tau(b))*c$. The right side by (1) is $((a\bar * \tau(c))*\tau(c))*b=(a*b)$ as 
needed.\\
(4) We have $$(a*b)*(c*d)= (((a*b)\bar *d)*c)*d \stackrel{(2)}{=}$$ 
$$(((a\bar *\tau(d))*\tau(b))*c)*d  \stackrel{(1)}{=} (((a\bar *\tau(d))*\tau(c))*b)*d \stackrel{(2)}{=}$$
$$(((a*c)\bar * d)*b)*d= (a*c)*(b*d) \mbox{ as needed.}$$

\end{proof}
\begin{corollary}\label{Corollary III:1.25} Let $G$ be group of a nilpotent class $2$, Then: 
\begin{enumerate} 
\item[(i)] The core magma $(X;*)$ with $a*b=ba^{-1}b$ and $\tau(a)=a^{-1}$ satisfies
conditions (1)-(4) of Proposition \ref{Proposition III:1.24}.
\item[(ii)] The conjugation magma $(X;*)$ with $a*b=b^{-1}ab$ and $\tau(a)=a$ 
satisfies
conditions (1)-(4) of Proposition \ref{Proposition III:1.24}.
\end{enumerate}
\end{corollary}
\begin{proof}
(i) In a core magma we have $(a*b)*c= (ba^{-1}b)*c= cb^{-1}ab^{-1}c$, and 
$(a*\tau(c))*\tau(b)= (a*c^{-1})*b^{-1}= b^{-1}cacb^{-1}= cb^{-1}[b^{-1},c]a[c^{-1},b]b^{-1}c$;
Thus $(a*b)*c=(a*\tau(c))*\tau(b)$ is equivalent to $a=[b^{-1},c]a[c^{-1},b]$ and further to 
$a^{-1}[b^{-1},c]a[c^{-1},b]=[a,[cb^{-1}]][[cb^{-1}],c^{-1}b]=1$ which holds in a group of class 2. 
Other conditions follow from Proposition \ref{Proposition III:1.24}.\\ 
(ii) In a conjugation magma we have 
 $(a*b)*c=(b^{-1}ab)*c=c^{-1}b^{-1}abc$ and $(a*\tau(c))*\tau(b)= (a*c)*b)=b^{-1}c^{-1}acb= 
= c^{-1}b^{-1}[b^{-1},c^{-1}]a[c^{-1},b^{-1}]bc $ Thus $(a*b)*c=(a*\tau(c))*\tau(b)$ is equivalent to 
$a= [b^{-1},c^{-1}]a[c^{-1},b^{-1}]$ and further to $a^{-1}[b^{-1},c^{-1}]a[b^{-1},c^{-1}]^{-1}= [a,[c^{-1},b^{-1}]]\in 
[[G,G],G]$ as needed. Other conditions follow from 
Proposition \ref{Proposition III:1.24}.
\end{proof}
The theory of homology of entropic magmas (analogous to that of homology of semigroups or shelves) is 
developed in \cite{Ni-P}.

\section{Proof of the main theorem}\label{Section III:2}

\begin{definition}\label{Definition III:2.1}(Descending diagram).\\
Assume that $L$ is an oriented link diagram of  
$n$ components together with $b= (b_1\ldots,b_n)$ 
which are base points
on $L$, each one chosen on a different link component of $L$
(the base points are outside of the crossings of the diagram).
Suppose that we move along the diagram $L$ according to the orientation
of $L$ so that we start from $b_1$, travel the first component ending in $b_1$
and then we travel the second component starting from $b_2$, etc... 
Every crossing is traveled in that way twice. 
We say that the diagram $L$ is descending with respect to $b$ 
if each crossing that we meet on our way first
is crossed  by an overcrossing (a bridge). We say that $L$ is 
ascending with respect to $b$ if its mirror image, $\bar L$ is 
descending with respect to $b$.

\end{definition}

It is not hard to see that for any diagram $L$ and any choice of 
basepoints $b= (b_1\ldots,b_n)$
there exists a resolving tree such that its leaves are descending
diagrams with respect to some appropriate choice of base points. 
Furthermore the diagram corresponding to the extreme left leaf has 
the same projection as $L$ and is descending with respect to $b$.

To show this
 we apply induction with respect to the number $k$ of crossings in
a diagram. The diagrams with no crossings are already descending.
Suppose that our claim is true for diagrams with fewer than $k>0$ crossings.
Now we can apply the following procedure for diagrams with $k$ crossings.
For any choice of  base points we start walking along the
diagram until we meet the first ``bad'' crossing $p$, 
i.e.~the first crossing on our way which is entered for the first time 
at an underpass (tunnel).
Then we begin to construct the tree changing the diagram at $p$.
If, for example, $\mbox{sgn} p = +$, then we get

\begin{center}
\begin{tabular}{c} 
\includegraphics[trim=0mm 0mm 0mm 0mm, width=.2\linewidth]
{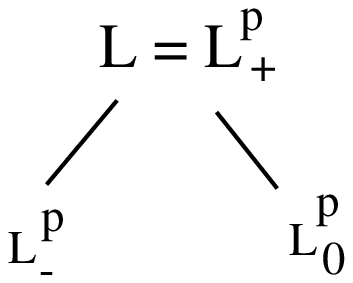}\\
\end{tabular}
\\
\end{center}

Next, we can apply our inductive assumption on the diagram $L_0^p$
and we continue the construction of the tree for the diagram
$L^p_-$ (i.e.~ we walk along the diagram looking for next ``bad''
crossings) until we eliminate all ``bad'' crossings. In $L^p_-$ we 
use the same base points as in $L$.
Later we will use similar reasoning several time so it is useful to 
formalize our double induction: having a diagram with a basepoints 
we induct on lexicographically ordered pair: 
(the number of crossings of $L$, the number of ``bad" crossings of $L$ 
with respect to the chosen basepoint). 

To prove Theorem~\ref{2:1.2} we will construct a value
function $w$ on diagrams of links. In order to prove
that $w$ is an invariant of ambient isotopy of oriented links
we will show that it is preserved by Reidemeister moves.

We will use the induction with respect to the number 
$\mbox{cr}(L)$ of crossing points in the diagram.
For any $k\geq 0$ we will define a function $w_k$ which assigns 
an element in $A$ to each oriented diagram $L$ with at most $k$ crossings.
Then we will define the function $w$ by setting $w(L)=w_k(L)$,
where $k\geq \mbox{cr}(L)$. 
Clearly, for this to work, the functions $w_k$ must satisfy certain coherence 
conditions.
Finally, the required properties of $w$ will be obtained from these
of $w_k$.

We begin by defining $w_0$.                                                                   

For a diagram $L$ with $n$ components and $\mbox{cr}(L)=0$ we put:
\begin{formulla}\label{e:2.2}
$$w_0(L)=a_n$$
\end{formulla}

To define $w_{k+1}$ and to prove its properties we will use induction
several times. To avoid misunderstandings the following will be called
the ``Main Inductive Hypothesis'' (abbreviated by MIH):
There is a function $w_k$ which associates $w_k(L)\in A$ 
to any diagram $L$ with $\mbox{cr}(L)\leq k$
and the function $w_k$ has the following properties:
\begin{formulla}
\begin{eqnarray*}\label{e:2.3}
w_k(U_n)&=&a_n\mbox{ where $U_n$ is a descending diagram
with $n$ components
}\\
&&\mbox{(with respect to some choice of base points).}\\
\end{eqnarray*}
\end{formulla}

\begin{formulla}\label{e:2.4}
$$w_k(L_+)=w_k(L_-)|w_k(L_0)$$
\end{formulla}

\begin{formulla}\label{e:2.5}
\ \ \ \ \ \ \ \ \ \ \ \ \ \ \ \ \ \ \ 
$w_k(L_-)=w_k(L_+)\star w_k(L_0)$ \\
\ \ \ \ ~~~~~~~~ \ \ \ for $L_+.L_-$ and $L_0$, being related
as usually.
\end{formulla}

\begin{formulla}\label{e:2.6}
\begin{eqnarray*}
w_k(L) &=&  w_k(R(L)),\mbox{ where $R$ is Reidemeister move
on $L$ such that} \\
&&\mbox{$\mbox{cr}(R(L))$ is at most $k$.}\\ 
\end{eqnarray*}
\end{formulla}

Then, as the reader may expect, we  want to make the Main Inductive Step,
abbreviated as MIS, to obtain the existence of a function
$w_{k+1}$ with analogous properties defined on diagrams with at most $k+1$
crossings. Before dealing with the task of making the MIS let us explain that
it will really end the proof of the theorem.
It is clear that the function $w_k$
satisfying MIS is uniquely determined by properties III.2.3, \ref{e:2.4},
 \ref{e:2.5}, and the fact that for every diagram there exists a resolving tree
with descending leaf diagrams. Thus the compatibility of the functions $w_k$ 
is obvious and they define a function $w$ defined on diagrams.

The function $w$ satisfies conditions (2) and (3) of Theorem~\ref{2:1.2},
because the functions $w_k$ satisfy such conditions.

If $R$ is a Reidemeister move on a diagram $L$, then $\mbox{cr}(R(L))$ equals
at most $k = \mbox{cr}(L)+2$, whence
$w_R(L)=w_k(R(L)), w_L = w_k(L)$ and by the properties of $w_k$
we have $w_k(L)=w_k(R(L))$, which implies $w_{R(L)} = w_L$. 

It follows that $w$ is an invariant of equivalence classes of oriented diagrams
and therefore also of the isotopy class of oriented links.

Now it is clear that $w$ has the required property (1) of Theorem~\ref{2:1.2},
since there exists a descending diagram $U_n$ in the same ambient 
isotopy class as $T_n$ (e.g. a link diagram without a crossing)
 and we have $w_{U_n} = a_n$.

The rest of the section \S 2 will be devoted to the proof of MIS.
For a given diagram $D$ with $\mbox{cr}(D)\leq k+1$ we will denote by ${\cal
D}$ the set of all diagrams which are obtained from $D$ by a finite number of 
operations of the kind 
{\includegraphics[trim=0mm 0mm 0mm 0mm, width=.02\linewidth]
{L+nmaly.eps}}$\rightarrow$ \includegraphics[trim=0mm 0mm 0mm 0mm, width=.02\linewidth]
{L-nmaly.eps}
or {\includegraphics[trim=0mm 0mm 0mm 0mm, width=.02\linewidth]
{L+nmaly.eps}}$\rightarrow${\includegraphics[trim=0mm 0mm 0mm 0mm, width=.02\linewidth]
{L0nmaly.eps}}.

Of course, once base points $b=(b_1,\ldots,b_n)$ are chosen on $D$, then the
same points can be chosen as base points for any $L\in{\cal D}$ provided that
$L$ is obtained from $D$ by operations of the first type only.

Let us define a function $w_b$ for a given $D$ and $b$ by assigning
an element of $A$ to each $L\in{\cal D}$.
If $\mbox{cr}(L)<k+1$, then we define 
\begin{formulla}\label{e:2.7}
$$w_b(L)=w_k(L).$$
\end{formulla} 

If $U_n$ is a descending diagram with respect to $b$ we put
\begin{formulla}\label{e:2.8}
$$w_b(U_n) = a_n\ \ \ \mbox{($n$ denotes the number of components).}$$
\end{formulla}

Now we proceed by induction with respect to the number $b(L)$ of 
bad crossings in $L$
(in the symbol $b(L)$ the letter $b$ 
works simultaneously for ``bad'' and for the choice 
$b=(b_1,\ldots,b_n)$). For a different choice of base points
$b'=(b_1',\ldots,b_n')$ we will write $b'(L)$). Assume that $w_b$
is defined for all $L\in{\cal D}$ such that $b(L)<t$ ($t>0$).
Then for the diagram $L$, $b(L)=t$, let $p$ be the first bad crossing
of $L$ (starting from $b_1$ and proceeding along the diagram).
Depending on the crossing $p$ being positive or negative 
we have $L=L^p_+$ or $L=L^p_-$. We define

\begin{formulla}\label{e:2.9}
$$
w_b(L) = \left\{
\begin{array}{ccc}
w_b(L^p_-)|w_b(L^p_0)&\mbox{if}&\mbox{sgn} (p) = +\\
w_b(L^p_+)\star w_b(LJ^p_0)&\mbox{if}&\mbox{sgn} (p) = -.\\
\end{array}
\right.
$$
\end{formulla}

We will show that $w_b$ is in fact independent of the choice of $b$ 
and that it has  the properties required of $w_{k+1}$.

\subsection{Conway Relations for $w_b$}

Let us begin with the proof that $w_b$ has Properties 2.4 
and 2.5.  

The considered crossing point will be denoted by $p$. We will restrict
our attention to the case $b(L^p_+)>b(L^p_-)$. The opposite case is quite
analogous.

We proceed by induction on $b(L^p_-)$. If $b(L^p_-) = 0$, then
$b(L^p_-) = 1$ and $p$ is the only one (hence the first) bad 
crossing of $L^p_+$. Hence by definition ~\ref{e:2.9} we get
$$w_b(L^p_+) = w_b(L^p_-)|w_b(L^p_0)$$
and further on using the properties of $C6$ we obtain
$$w_b(L^p_-)=w_b(L^p_+)\star w_b(L^p_0).$$

Let us assume now that every diagram $L$ such that $w_b(L^p_-)<t$, with $t\geq
1$, satisfies formulae ~\ref{e:2.4} and \ref{e:2.5} for $w_b$. Let us consider
the case $b(L^p_-)   =t$. 

By the assumption we have $w_b(L^p_+)\geq2$. Let $q$ be 
the first bad crossing of
the diagram $L^p_+$. If $q=p$, then by \ref{e:2.9} we have
$$ w_b(L^p_+) = w_b(L^p_-)   |w_b(L^p_0).$$

Let us now consider the case when $q\neq p$. 
Let  $\mbox{sgn}\ q = +$, for example. 
Then by \ref{e:2.9} we have
$$w_b(L^p_+)=w_b(L^{p\ q}_{+\ +}) =w_b(L^{p\ q}_{+\ -})|w_b(L^{p\ q}_{+\
0}).$$

Because $w_b(L^{p\ q}_{-\ -})<t$ and $\mbox{cr}L^{p\ q}_{+\ 0}\leq
k$, therefore by the inductive hypothesis and by MIH  we have
$$  w_b(L^{p\ q}_{+\ -}) = w_b(L^{p\ q}_{-\ -})|w_b(L^{p\ q}_{0\ -})$$
and $$  w_b(L^{p\ q}_{+\ 0}) = w_b(L^{p\ q}_{-\ 0})|w_b(L^{p\ q}_{0\ 0})$$
hence $$  w_b(L^p_+) = (w_b(L^{p\ q}_{-\ -}) |w_b(L^{p\ q}_{0\ -}))|(
w_b(L^{p\ q}_{-\ 0})|w_b(L^{p\ q}_{0\ 0})).$$

By the transposition property $C3$ we obtain
\begin{formulla}\label{e:2.10}
  $$w_b(L^p_+) = (w_b(L^{p\ q}_{-\ -}) |w_b(L^{p\ q}_{-\ 0}))|(
w_b(L^{p\ q}_{0\ -})|w_b(L^{p\ q}_{0\ 0})).$$
\end{formulla} 

On the other hand $b(L^{p\ q}_{-\ -})<t$ and $\mbox{cr}(L^p_0)\leq k$,  
so using once more the inductive hypothesis and MIH we obtain
\begin{formulla}\label{e:2.11}
$$w_b(L^p_-)=w_b(L^{p\ q}_{-\ +}) = w_b(L^{p\ q}_{-\ -})|w_b(L^{p\ q}_{-\
0})
\\
w_b(L^p_0)=w_b(L^{p\ q}_{0\ +}) = w_b(L^{p\ q}_{0\ -})|w_b(L^{p\ q}_{0\ 0})$$
\end{formulla}

Putting together III.2.11 and III.2.10 we obtain
$$w_b(L^p_+)=w_b(L^p_-)|w_b(L^p_0)$$
as required. If $\mbox{sgn}q =-$ we should use $C4$
instead of $C3$. This completes the proof of Conway Relations for $w_b$.
\subsection{Changing Base Points}

We will show now that $w_b$ does not depend on the choice of $b$,
provided the order of components of the diagram is not changed.
It amounts to the verification 
that we can replace $b_i$ by $b_i'$, taken from
the same component in such a way that $b_i$ lies after $b_i'$ and there is 
exactly one crossing point (say $p$) between them. Let
 $b'=(b_1,\ldots,b_i',\ldots,b_n)$. We want to show that
$w_b(L) = w_{b'}(L)$ for every diagram with $k+1$ crossings belonging to
${\cal D}$. Now we will consider the case $\mbox{sgn} p = +$; 
the case $\mbox{sgn}p = -$ is quite analogous.

We use induction with respect to $B(L):= \max(b(L),b'(L))$. 
We need to consider three cases.

\begin{description}
\item[CBP 1] Let us assume that $B(L)=0$. Then $L$ is a descending diagram
with respect to both choices of base points: $b$ just like $b'$. By~\ref{e:2.8}
we have
$$ w_b(L) = a_n = w_{b'}(L).$$
\item[CBP 2] Let us assume that $B(L)=1$ and $b(L)\neq b'(L)$.

This is possible only when $p$ is a self-crossing point of the $i$-th component 
of $L$. There are two subcases to  be considered.
\begin{description}
\item[CBP 2 (a)]: $b(L) = 1$ and $b'(L) = 0$. 

Then $L$ is a descending diagram with respect to $b'$  and by ~\ref{e:2.8} we have
$$w_{b'}(L) = a_n$$
thus by ~\ref{e:2.9} we obtain
$$w_b(L) = w_b(L^p_+) = w_b(L^p_-)|w_b(L^p_0)$$
(we have restricted our attention to the case $\mbox{sgn} p
= +$).

Now $w_b(L^p_-) = a_n$, because $b(L^p_-) = 0$, moreover $L^p_0$ is
a descending diagram with respect to a proper choice of base points.
Since $L^p_0$ has $n+1$ components, so
$w_b(L^p_0) = a_{n+1}$.

It follows that $w_b(L) = a_n|a_{n+1}$, and since by condition $C1$
we have $a_n|a_{n+1} = a_n$, so $w_b(L) = a_n = w_{b'}(L)$.

\item[CBP 2 (b)]: $B(L)=0$ and $b'(L)=1$.

We deal with this case just like with CBP 2(a).
\end{description}

\item [CBP 3] $B(L) = t>1$ or $B(L) = 1 = b(L) = b'(L)$.

We assume by induction that $w_b(K) = w_{b'}(K)$ for
$B(K)<B(L)$. Let $q$ be a bad crossing with respect to
$b$ and $b'$ as well. This time we will consider the case  $\mbox{sgn}q =
-$. The case $\mbox{sgn}q = +$ is analogous.

Using the already proven Conway relations for $w_b$ and $w_{b'}$
we obtain:
$$w_b(L) = w_b(L^q_-) = w_b(L^q_+)\star w_b(L^q_0)$$
$$w_{b'}(L) = w_{b'}(L^q_-) = w_{b'}(L^q_+)\star w_{b'}(L^q_0)$$

Since $B(L^q_+)<B(L)$ and $\mbox{cr}(L^q_0)\leq k$, so by the (local) inductive
hypothesis and by MIH we get
$$w_b(L^q_+) = w_{b'}(L^q_+)$$
$$w_b(L^q_0) = w_{b'}(L^q_0)$$
Which imply $w_b(L) = w_{b'}(L)$. This completes the CBP proof.

\end{description}  

Since $w_b$ has turned out to be independent of base point changes 
which preserve the order of components of the diagram, so now we can
consider a new function $w^0$ which associates an element of 
$A$ to any diagram $L$ with $\mbox{cr}(L)\leq k+1$ and with a fixed 
order of components.

\subsection{Independence of $w^0$ of Reidemeister Moves (IRM)}

When $L$ is a diagram with a fixed order of components and $R$
is a Reidemeister move on $L$, then $R(L)$ has a natural order of components
yielded by the order of components of $L$. Assuming that
$\mbox{cr}(R(L))\leq \mbox{cr}(L) \leq k+1$, we will show that 
$w^0(L) = w^0(R(L))$. 

We use induction with respect to  the number of bad crossings $b(L)$ 
for a  proper choice of base points $b=(b_1,\ldots,b_n)$. 
This choice must be compatible with the given 
order of components.
Let us choose the base points which lie 
outside of the part of the diagram involved in the considered 
Reidemeister move $R$, so that the same points could work for
the diagram $R(L)$ as well. 

We need to consider three standard types of Reidemeister moves (Fig.2.1).
\begin{center}
\begin{tabular}{c} 
\includegraphics[trim=0mm 0mm 0mm 0mm, width=.8\linewidth]
{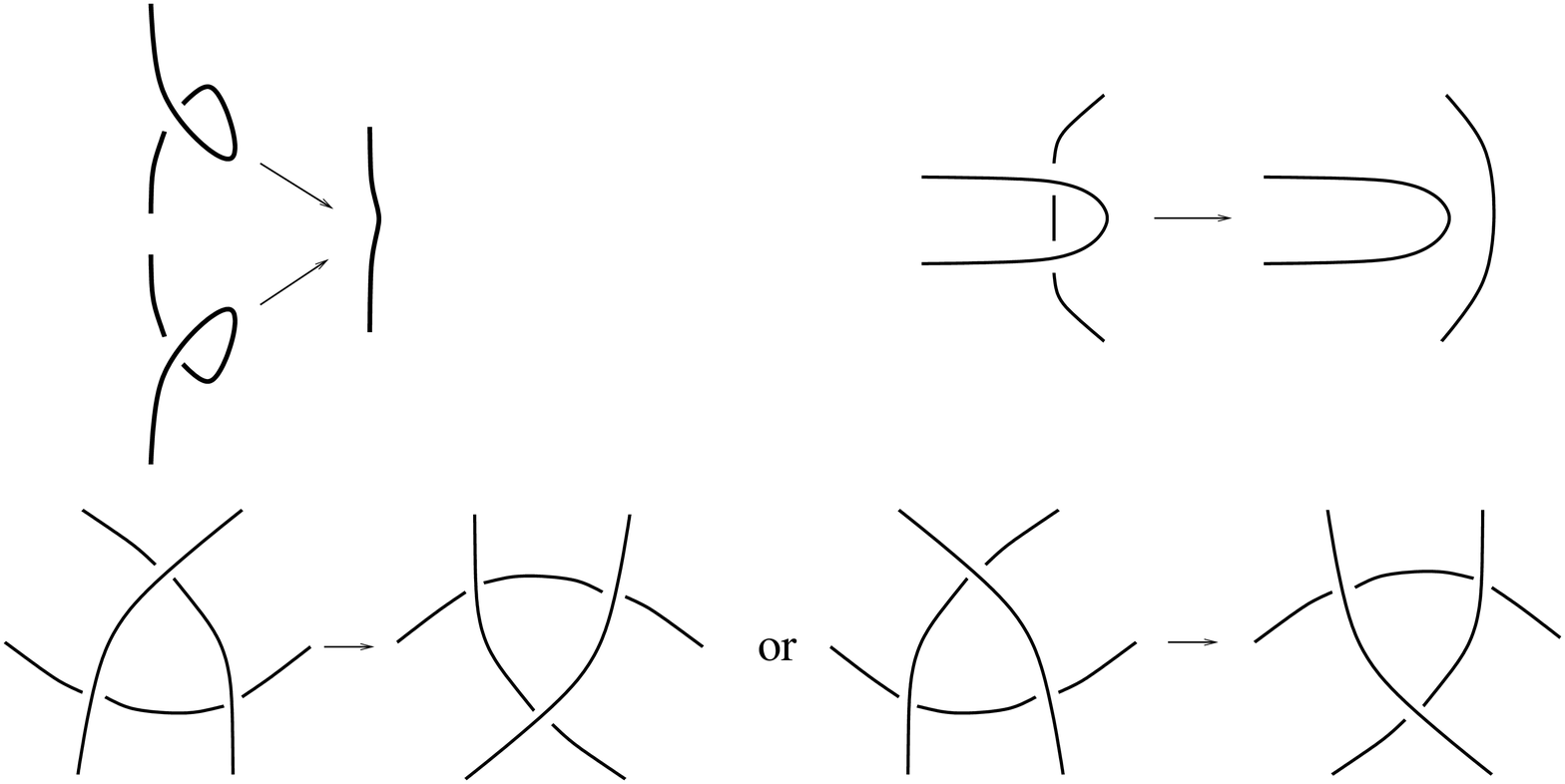}\\
\end{tabular}
\\
Fig.~2.1
\end{center}

Let us assume that $b(L) = 0$. It is clear then that also $b(R(L)) = 0$
and the number of components remains the same. Thus by ~\ref{e:2.8}
$$w^0(L) = w^0(R(L)).$$

Now we assume by induction that $w^0(L) = w^0(R(L))$ when $b(L)<t$.
Let us consider the case $b(L) = t$. Assume that there is a bad crossing
$p$ in $L$ which is different from all the crossings involved in
the considered Reidemeister move. Assume, for example, that 
$\mbox{sgn} p = +$. Then, by the inductive hypothesis,we have

\begin{formulla}\label{e:2.12}
$$w_0(L^p_-) = w^0(R(L^p_-)).$$
\end{formulla}
and by MIH we obtain
\begin{formulla}\label{e:2.13}
$$w^0(L^p_0) = w^0(R(L^p_0)).$$
\end{formulla}
Now by the Conway relation \ref{e:2.4}, which has been already verified for
$w^0$ we have 
$$w^0(L) = w^0(L^p_+) = w^0(L^p_-)|w^0(L^p_0)$$
$$w^0(R(L)) = w^0(R(L)^p_+) = w^0(R(L)^p_-)|w^0(R(L)^p_0).$$

Since $R(L^p_-) = R(L)^p_-$ and obviously $R(L^p_0)=R(L)^p_0$,
so by \ref{e:2.12} and \ref{e:2.13} we will get
$$w^0(L) = w^0(R(L)).$$

It remains to consider the case, when all the bad crossings of $L$,
lie in the part of diagram involved in the considered Reidemeister move.
Let us consider each of the three types of Reidemeister moves separately. 
The most complicated case is that of a Reidemeister move of the third type. 
Before we deal with it, let us formulate the following observation:

Whatever the choice of base points, 
the crossing point of the top arc and the bottom arc (e.g. $q$ in Fig. 2.2) 
cannot be the only bad point of the diagram.
\begin{center}
\begin{tabular}{c} 
\includegraphics[trim=0mm 0mm 0mm 0mm, width=.3\linewidth]
{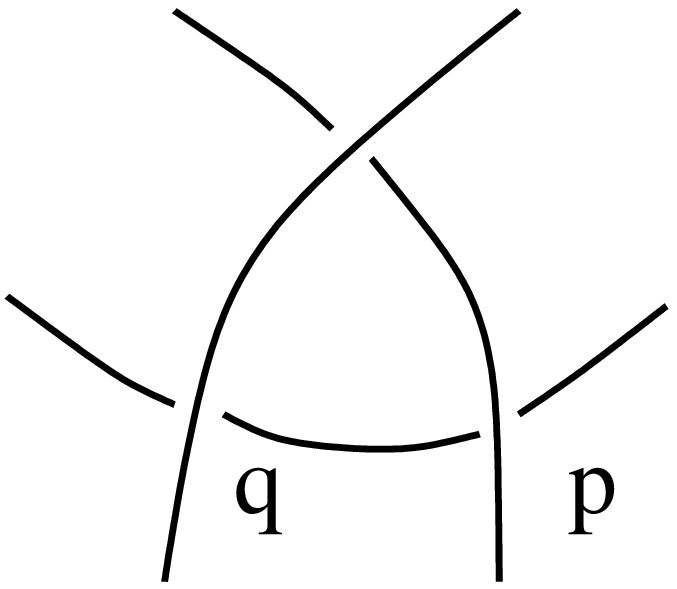}\\
\end{tabular}
\\
Fig.~2.2
\end{center}

The proof of the above observation amounts to an easy case by case checking.
This observation makes the following induction possible. 
We can assume that we have
a bad crossing between the middle arc and either the lower or the upper arc.
Let us consider, for example, the situation described by Fig. 2.2. 
Subsequently we need to consider two subcases, according to 
$\mbox{sgn} p = +$ or $-$. 

Let us assume that $\mbox{sgn}p = -$. Then by Conway relations:
$$w^0(L) =w^0(L^p_-) = w^0(L^p_+)\star w^0(L^p_0)$$
$$w^0(R(L)) = w^0(R(L)^p_-) = w^0(R(L)^p_+)\star w^0(R(L)^p_0).$$
By the inductive hypothesis and by the equation $R(L)^p_+ = R(L^p_+)$
we obtain:
$$w^0(L^p_+) = w^0(R(L)^p_+).$$
Since $R(L)^p_0$ is obtained from $L^p_0$ by two Reidemeister moves of the 
second type (Fig. 2.3), thus by MIH 
$w^0(R(L)^p_0) = w^0(L^p_0)$ and then follows the equality
$$w^0(L) = w^0(R(L)).$$

\begin{center}
\begin{tabular}{c} 
\includegraphics[trim=0mm 0mm 0mm 0mm, width=.5\linewidth]
{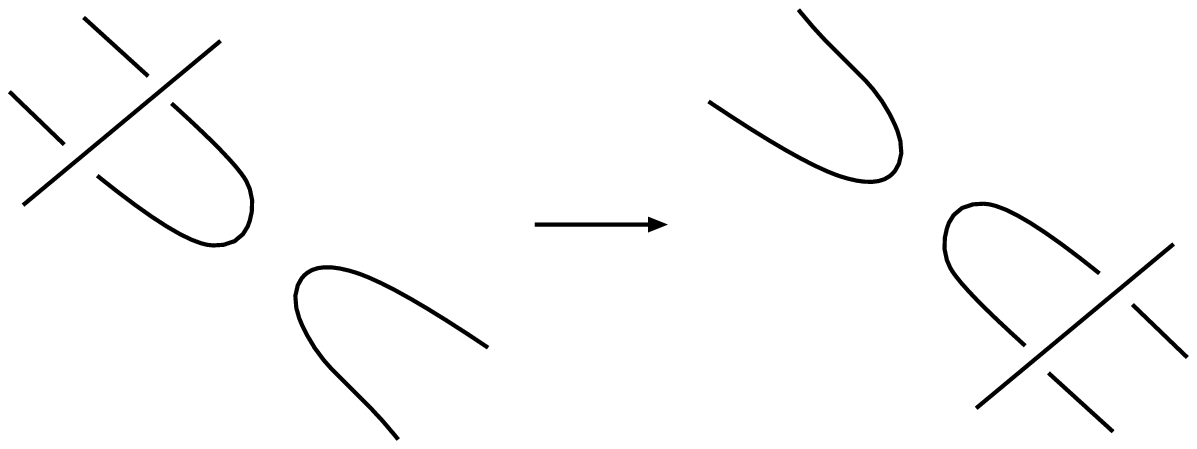}\\
\end{tabular}
\\
Fig.~2.3
\end{center}
Let us assume now that $\mbox{sgn}p=+$. Then by Conway relations
$$w^0(L) =w^0(L^p_+) = w^0(L^p_-)| w^0(L^p_0)$$
$$w^0(R(L)) = w^0(R(L)^p_+) = w^0(R(L)^p_-)| w^0(R(L)^p_0).$$
By the inductive hypothesis and equality $R(L)^p_- = R(L^p_-)$
we obtain:
$$w^0(L^p_-) = w^0(R(L)^p_-).$$

Now, $L^p_0$ and $R(L)^p_0$ are essentially the same diagrams
(Fig. 2.4), thus $w^0(L^p_0) = w^0(R(L)^p_0)$ and in the end we obtain
an equality.
$$w^0(L) = w^0(R(L)).$$
\begin{center}
\begin{tabular}{c} 
\includegraphics[trim=0mm 0mm 0mm 0mm, width=.5\linewidth]
{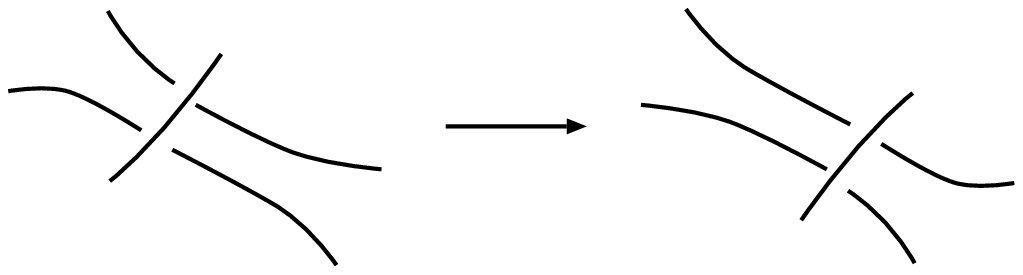}\\
\end{tabular}
\\
Fig. 2.4
\end{center}

{\em Reidemeister moves of the first type}.

The base points can always be chosen in such a way 
that the crossing point involved
in the move is good. Thus $b(L)=b(R(L))=0$ and $w^0(L)=w^0(R(L)$.

{\em Reidemeister moves of the second type}.

There is only one case when we cannot choose base points to secure that
the points involved in the move are good.
It happens when the arcs involved are parts of different components and
the lower arc is a part of the smaller component in the ordering.
In this case both the crossing points are bad and of the opposite signs.
Let us consider the case shown on Fig.2.5.
\vspace{0.5cm}
\begin{center}
\begin{tabular}{c} 
\includegraphics[trim=0mm 0mm 0mm 0mm, width=.7\linewidth]
{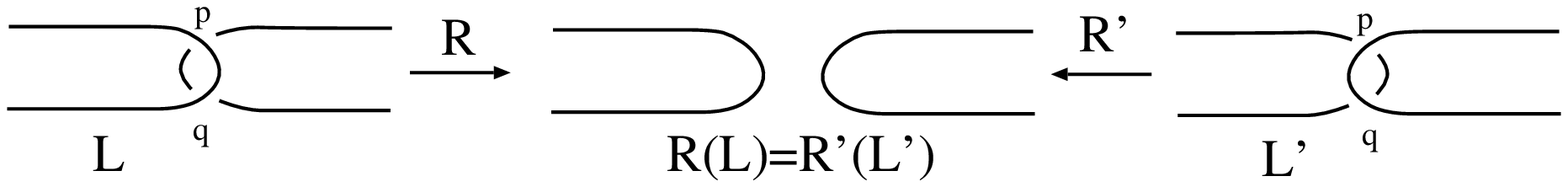}\\
\end{tabular}
\\
Fig. 2.5
\end{center}
\vspace{0.5cm}
By the inductive hypothesis we have
$$ w^0(L') = w^0(R'(L')) = w^0(R(L)).$$

Using the already proven Conway relations, conditions $C6$ 
and $C7$, and MIH, if necessary, it can be proved that
$w^0(L) = w^0(L')$. Let us discuss in detail the case involving MIH. 
It occurs when $\mbox{sgn} p = +$. Then we have
$$w^0(L) = w^0(L^q_-) = w^0(L^q_+)\star w^0(L^q_0) = (w^0(L^{q\ p}_{+\
-})|w^0(L^{q\ p}_{+\ 0}))\star w^0(L^q_0).$$
But $L^{q\ p}_{+\ -} = L'$ and by MIH it follows that $w^0(L^{q\ p}_{+\ 0}) =
w^0(L^q_0)$ (see Fig.2.6; here diagrams $L^{q\ p}_{+\ 0}$ and $L^q_0$
are both obtained from $K$ by a first Reidemeister move).
\begin{center}
\begin{tabular}{c} 
\includegraphics[trim=0mm 0mm 0mm 0mm, width=.5\linewidth]
{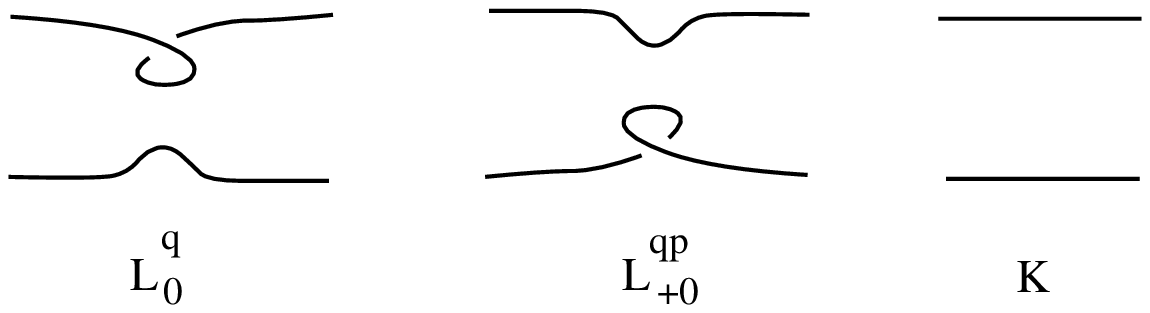}\\
\end{tabular}
\\
Fig. 2.6
\end{center}

By condition $C7$ we obtain
$$w^0(L) = w^0(L'),$$
and consequently $$w^0(L) = w^0(R(L)).$$
The case $\mbox{sgn}p=-$ is even simpler (see Fig.2.7) 
and we leave it for the reader.

\begin{center}
\begin{tabular}{c} 
\includegraphics[trim=0mm 0mm 0mm 0mm, width=.5\linewidth]
{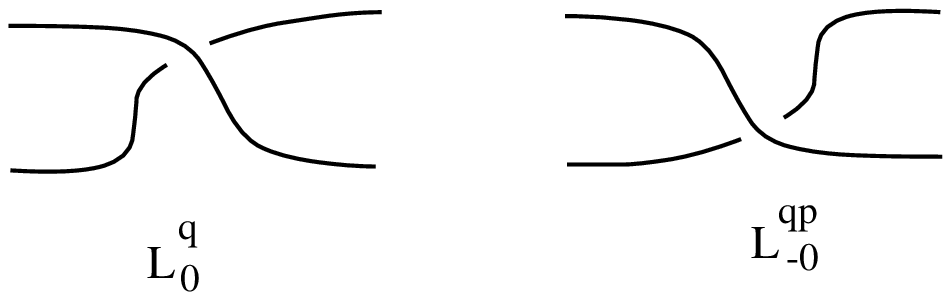}\\
\end{tabular}
\\
Fig. 2.7 (case $\sgn p = -1$).
\end{center}

This completes the proof of the independence of $w^0$ of Reidemeister moves.

\begin{remark}\label{Remark III.2.14}
If $L_i$ is a trivial component of a diagram $L$, i.e. $L_i$
has no crossing points, then the specific position of $L_i$
has no effect on $w^0(L)$. Equivalently, If $L$ and $L'$ differ by the 
position of $L_i$ as in Figure 2.8, then $w^0(L) = w_b(L)=w_b(L')=w^0(L')$.
This can be easily achieved by the induction with respect to
$b(L)$. Notice however that we cannot use IRM in this case because 
Reidemeister moves have to increase the number of crossings on 
the way from $L$ to $L'$.
\end{remark}

\begin{center}
\begin{tabular}{c}
\includegraphics[trim=0mm 0mm 0mm 0mm, width=.3\linewidth]
{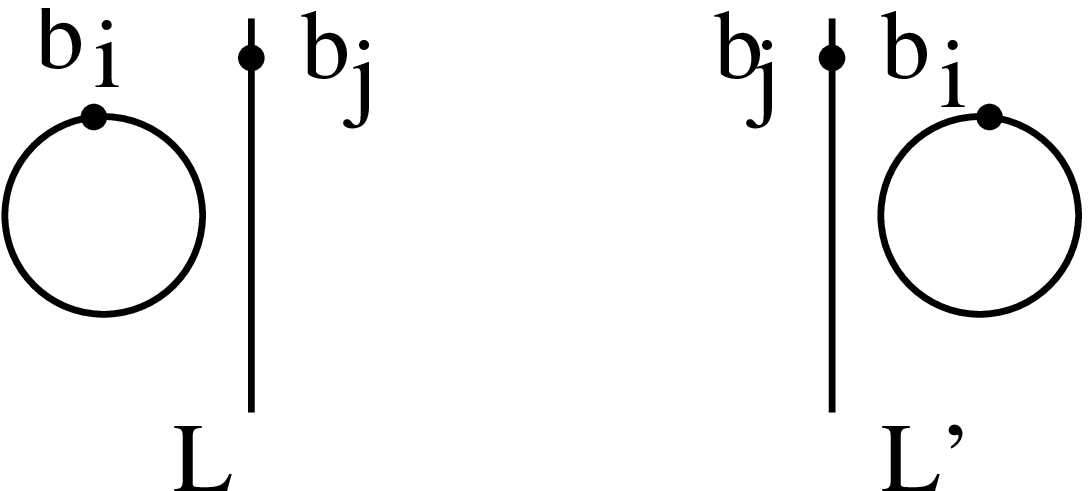}\\
\end{tabular}
\\
Fig.~2.8
\end{center}

To complete the Main Inductive Step it is enough to prove the
independence of $w^0$ of the order of components. Then we set
$w_{k+1} = w^0$. The required properties of $w_{k+1}$ have been
already checked.

\subsection{Independence of the Order of Components (IOC)}

It is enough to verify that for a given diagram $L$, $\mbox{cr}(L)\leq
k+1$, and for a fixed base points $b =
(b_1,\ldots,b_i,b_{i+1},\ldots,b_n)$ we have 
$$ w_b(L) = w_{b'}(L)$$
where $b'= (b_1,\ldots,b_{i+1},b_i,\ldots,b_n)$. By induction on $b(L)$
we can easily reduce it to the case of a descending diagram, $b(L)=0$. 
To deal with this case we will choose appropriate base points.

The proof will be concluded if we show, that our diagram can be
transformed into another one with a smaller number of crossings
by a series of Reidemeister moves not increasing the number of crossings.
To do it we can use IRM and MIH. This property of a descending diagram
is guaranteed by the following lemma.

\begin{lemma}\label{l:2.15}
Let $L$ be a diagram with $k$ crossings and a given ordering of components
$L_1,L_2,\ldots,L_n$. Then either $L$ has a trivial circle as a component
or there is a choice of base points $b = (b_1,\ldots,b_n)$, $b_i\in L_i$, 
such that a descending diagram $L^d$ associated with $L$ and $b$ (that is
all the bad crossings of $L$ are changed to good ones) can be changed
into a diagram with less than $k$ crossings by a sequence of Reidemeister
moves not increasing the number of crossings. 
\end{lemma}

Proof of lemma \ref{l:2.15}.
 
A closed region cut out of the plain by arcs of $L$
is called an $i$-gon, if it has $i$ vertices 
(only crossings can be vertices). See Fig.2.9.

\begin{center}
\begin{tabular}{c} 
\includegraphics[trim=0mm 0mm 0mm 0mm, width=.5\linewidth]
{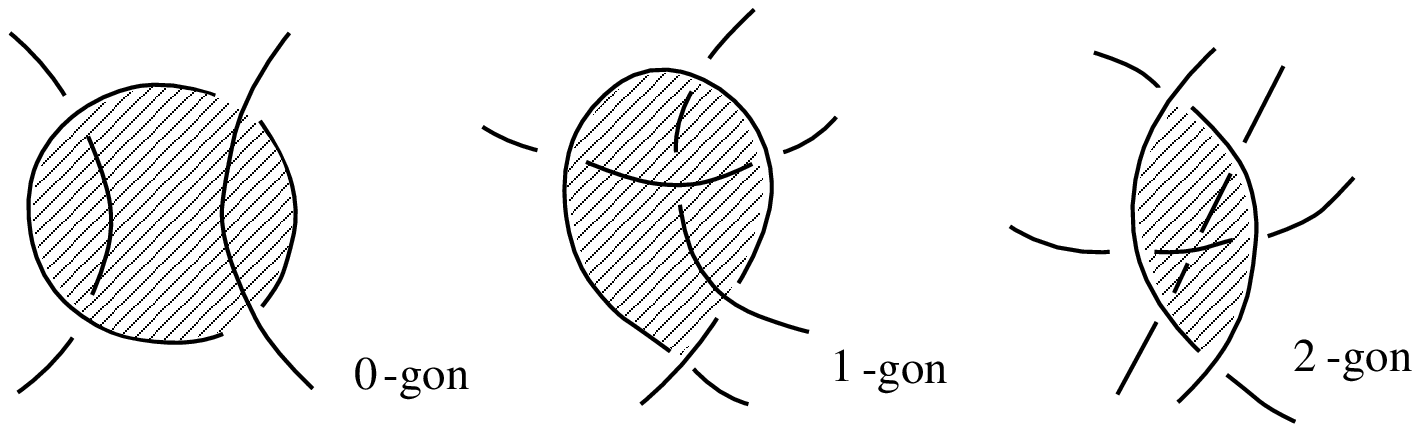}\\
\end{tabular}
\\
Fig. 2.9
\end{center}
 
Every $i$-gon with $i\leq 2$ is called an $f$-gon ($f$ stands for few).
Now let $X$ be an innermost $f$-gon, that is, an $f$-gon which
does not contain any other $f$-gon inside.

If $X$ is a $0$-gon then we are done because  $\partial X$
is a trivial circle. If $X$ is a $1$-gon then we are done too, because
$\mbox{int} X\cap L=\emptyset$ so on $L^d$ we can perform 
a first Reidemeister move
decreasing the number of crossings of $L^d$ (Fig.2.10).

\begin{center}
\begin{tabular}{c} 
\includegraphics[trim=0mm 0mm 0mm 0mm, width=.3\linewidth]
{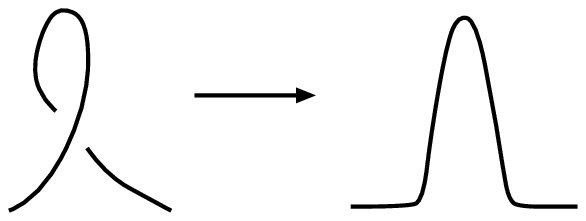}\\
\end{tabular}
\\
Fig. 2.10
\end{center}

Therefore we assume that $X$ is a $2$-gon. Each arc which cuts 
$\mbox{int} X$ goes from one edge to another.
Furthermore, since no component of $L$ lies fully in $X$ we can choose
base points $b=(b_1,\ldots,b_n)$ lying outside of $X$. This has important 
consequences. If $L^d$ is a descending diagram associated with $L$ 
and $b$ then each $3$-gon in $X$ 
allows for a Reidemeister move of the third type (i.e. the situation
of the Fig.2.11. is impossible).

\begin{center}
\begin{tabular}{c} 
\includegraphics[trim=0mm 0mm 0mm 0mm, width=.5\linewidth]
{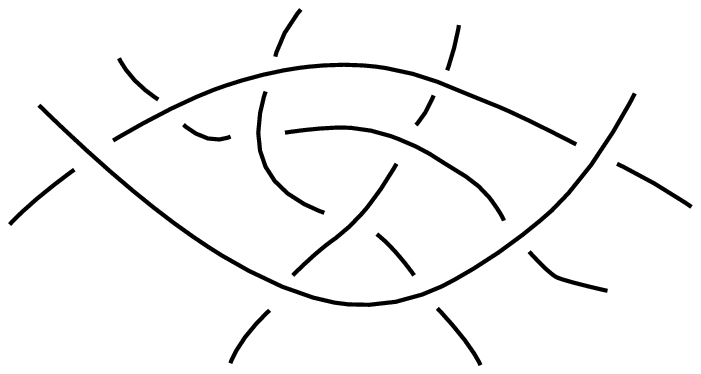}\\
\end{tabular}
\\
Fig. 2.11
\end{center}

Now we will prove Lemma \ref{l:2.15} by induction on the number of
crossings of $L$ contained in the $2$-gon $X$. We denote this number by $c$.

If $c=2$ then $\mbox{int} X\cap L=\emptyset$ and we are done 
thanks to the choice of base points. $2$-gon $X$ can be used to make 
the Reidemeister move of the second type on $L^d$ and to
reduce the number of crossings in $L^d$ in this way.

Assume that $L$ has $c>2$ crossings in $X$ and that Lemma 2.15 
has been proved for the number of crossings in $X$ smaller than $c$.

In order to make the inductive step we need the following lemma.

\begin{lemma}\label{l:2.16}
If $X$ is an innermost $2$-gon with $\mbox{int}X\cap L\neq\emptyset$,
then there is a $3$-gon $\bigtriangleup\subset X$ such that 
$\bigtriangleup\cap\partial
X\neq\emptyset$ and $\mbox{int}\bigtriangleup\cap L=\emptyset$.
\end{lemma}

Before proving Lemma~\ref{l:2.16} we will show how Lemma~\ref{l:2.15}
follows from it.

Using an $3$-gon $\bigtriangleup$ from Lemma~\ref{l:2.16}
we can make a Reidemeister move of the third type and reduce
the number of crossings $L^d$ in $X$ (compare Fig.2.12).

\begin{center}
\begin{tabular}{c} 
\includegraphics[trim=0mm 0mm 0mm 0mm, width=.5\linewidth]
{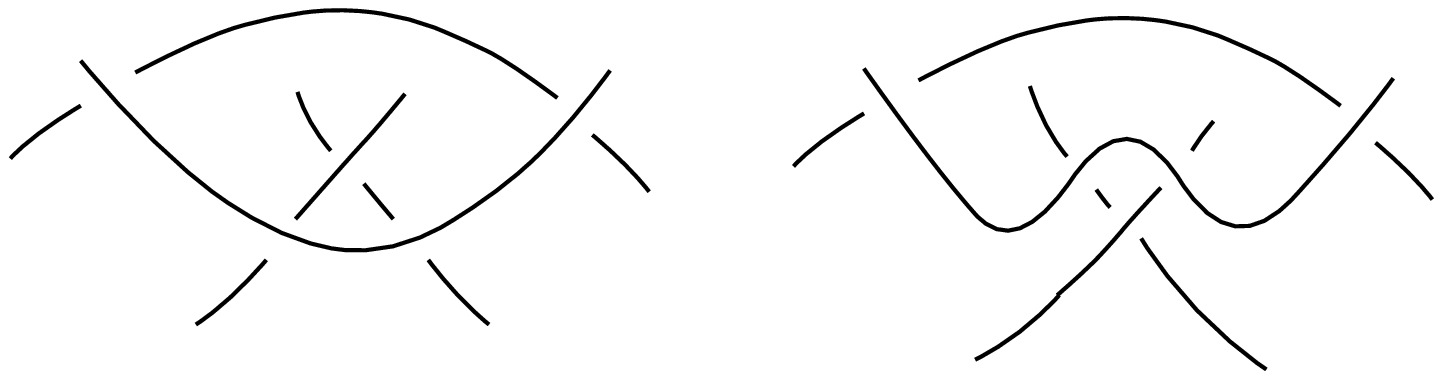}\\
\end{tabular}
\\
Fig.~2.12
\end{center}

Now either $X$ is an innermost $f$-gon with less than $c$ crossings
in $X$ or $X$ contains an innermost $f$-gon with less than $c$ 
crossings. In both cases we can use the inductive hypothesis to complete the
proof of the Lemma~\ref{l:2.15}.

Instead of proving Lemma~\ref{l:2.16} we will show a more general fact
of which Lemma~\ref{l:2.16} is a special case. 

\begin{lemma}\label{l:2.17}
Let us consider a $3$-gon $Y = (a,b,c)$ such that each arc which
cuts it goes from the $\overline{ab}$ edge to the $\overline{ac}$ edge
with no self-intersections. We allow that $Y$ is a $2$-gon considered
as a degenerated $3$-gon with an edge $\overline{bc}$ collapsed to a point
and moreover that there is no $f$-gon in $\mbox{int} X$. 
Furthermore let us assume that there is an arc which cuts $\mbox{int} X$. 
Then there is a $3$-gon $\bigtriangleup\subset Y$ such that 
$\bigtriangleup\cap\overline{ab}\neq\emptyset$ and $\mbox{int}\bigtriangleup$
is not cut by any arc.
\end{lemma}

Proof of Lemma~\ref{l:2.17}.

We proceed by induction on the number of arcs in $\mbox{int} Y\cap L$ 
(each such an arc cuts $\overline{ab}$ and $\overline{ac}$). For one arc
the Lemma is obvious (Fig.2.13).

\begin{center}
\begin{tabular}{c} 
\includegraphics[trim=0mm 0mm 0mm 0mm, width=.35\linewidth]
{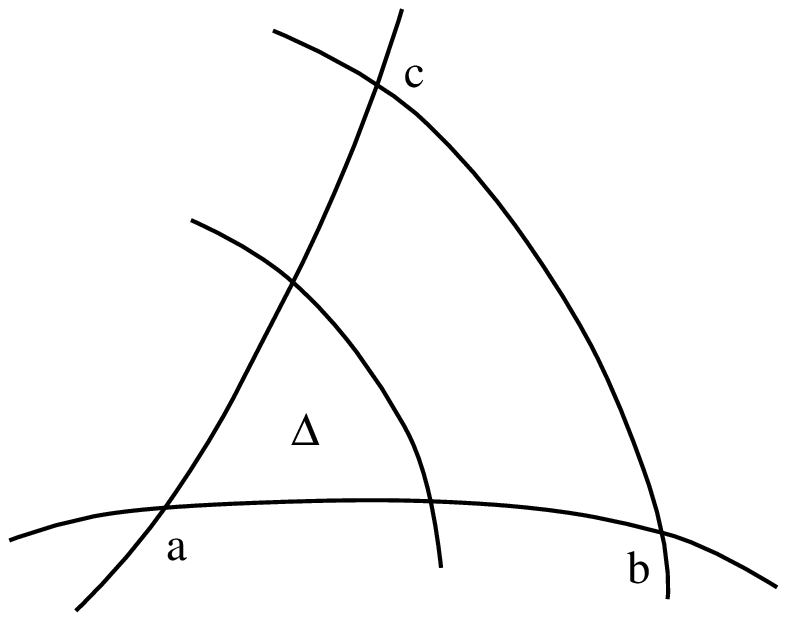}\\
\end{tabular}
\\
Fig.~2.13
\end{center}
        
Assume that the Lemma is true for $k$ arcs ($k\geq 1$) and let us consider
$(k+1)$-th arc $\gamma$. Let $\bigtriangleup_0=(a_1,b_1,c_1)$ be a $3$-gon
from the inductive hypothesis with a $\overline{a_1b_1}$ edge included
in $\overline{ab}$ (Fig.2.14).

\begin{center}
\begin{tabular}{c} 
\includegraphics[trim=0mm 0mm 0mm 0mm, width=.5\linewidth]
{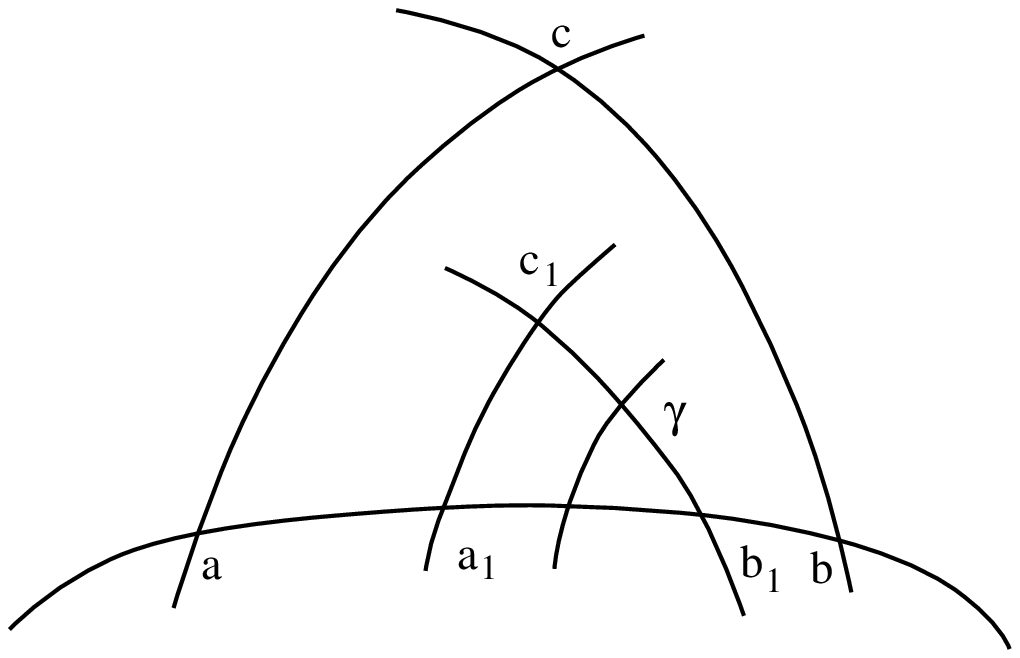}\\
\end{tabular}
\\
Fig.~2.14
\end{center}

If $\gamma$ does not cut $\bigtriangleup_o$ or cuts $\overline{a_1b_1}$,
then we are done (Fig.2.14). Therefore let us assume that $\gamma$ meets
$\overline{a_1c_1}$ (at $u_1$) and $\overline{b_1c_1}$ (at $w_1$). Let $\gamma$
cut $\overline{ab}$ at $u$ and $\overline{ac}$ at $w$ (Fig.2.15). 
Now we should consider two cases: 
\begin{enumerate}
\item $\overline{uu_1}\cap \mbox{int}\bigtriangleup_0=\emptyset$ (so
$\overline{ww_1}\cap \mbox{int}\bigtriangleup_0=\emptyset$);

\begin{center}
\begin{tabular}{c} 
\includegraphics[trim=0mm 0mm 0mm 0mm, width=.5\linewidth]
{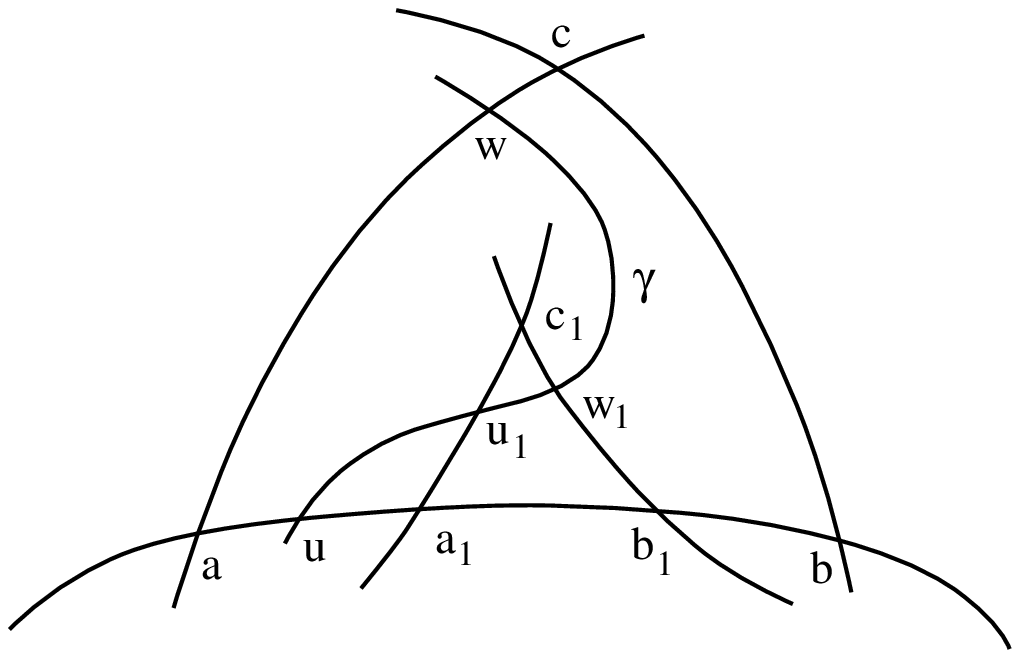}\\
\end{tabular}
\\
Fig. 2.15
\end{center}

Let us consider the $3$-gon $ua_1u_1$. No arc can cut the edge 
$\overline{a_1u_1}$ so each arc which cuts the $3$-gon
is cut by less than $k+1$ curves. Hence by the inductive hypothesis
there is a 3-gon $\bigtriangleup$ in $ua_1u_1$ with an edge on  
$\overline{ua_1}$ and the interior of  $\bigtriangleup$
is cut by no arc at all. In this way $\bigtriangleup$ is what we need
for Lemma~\ref{l:2.17}.

\item $\overline{uw_1}\cap \mbox{int}\bigtriangleup_0 = \emptyset$ 
( so $\overline{wu_1}\cap\mbox{int}\bigtriangleup_0=\emptyset$).\\
This case is analogous to Case 1.
\end{enumerate}

This completes the proof of Lemma~\ref{l:2.17} --- and hence the Lemma
~\ref{l:2.15} as well\footnote{A reader familiar with basic fact of the theory 
of braids may recognize that what we really analyze are braids from $\overline{ac}$ 
to $\overline{ab}$. In this language Lemma III:2.17 is intuitively obvious -- any braid has 
the last symbol. It however requires a proof that we actually deal with braids; here 
``descending" and lack of $f$-gons should be used as we did in the presented proof.}.

In this way we have completed the proof of IOC                                                   
and hence the Main Inductive Step and the proof of Theorem
~\ref{2:1.2}.




\section{Properties of Invariants of Conway Type}

We first observe that invariants coming from Conway algebras can be composed, as long as 
operations used in algebras form an entropic set (Proposition III:1.22). In particular,
having a Conway algebra $(A;|)$ we can form a new Conway algebra $(A;|^n)$. We look more carefully 
at the example giving a Homflypt polynomial (we consider more familiar skein relation, proposed 
by H. Morton $v^{-1}P_{L_+}-vP_{L_-}=zP_{L_0}$ with $P_{T_k}=(\frac{v^{-1}-v}{z})^{k-1}$.
Thus the Conway algebra yielding this invariant have $A=\Z[v^{\pm 1},z^{\pm 1}]$, $a_k=(\frac{v^{-1}-v}{z})^{k-1}$,
and $a|b= v^2a+vzb$. If we consider composite operation $|^n$ we get 
$$a|^nb= (a|^{n-1}b)|b= v^2(a|^{n-1}b)+vzb= v^{2n}a + (v+v^3+...+v^{2n-1})zb= $$ 
$$ v^{2n}a+\frac{v^{2n+1}-v}{v^2-1}zb=
v^{2n}a + v^n\frac{v^{n}-v^{-n}}{v-v^{-1}}zb.$$ 
The formula holds for any integer $n$ ($|^{-n}=*^n$).\\
The corresponding Conway skein relation has the form:
$$v^{-n}P_{L_+}-v^nP_{L_-}=\frac{v^{n}-v^{-n}}{v-v^{-1}}zP_{L_0}.$$
We can offer the following knot theoretical visualization of our composite relation: 
consider two positive antiparallel crossings as in Figure 3.0. 
The Conway skein relation yielded by $|$ and  
applied to $L_2$ gives $P_{L_2}= v^2P_{L_0}+vzP_{P_{\infty}}$. Now consider the picture with $2n$ 
antiparallel crossing as in Figure 3.0. Then in order to find $P_{L_{2n}}$ in terms of $P_{L_0}$ and $P_{L_{\infty}}$ 
we can use once $|^n$ in place of using  $n$ times $|$ to get 
$P_{L_{2n}}= v^{2n}P_{L_0}+v^n\frac{v^{n}-v^{-n}}{v-v^{-1}}zP_{L_{\infty}}$.
\\ \ \\
\begin{center}
\psfig{figure=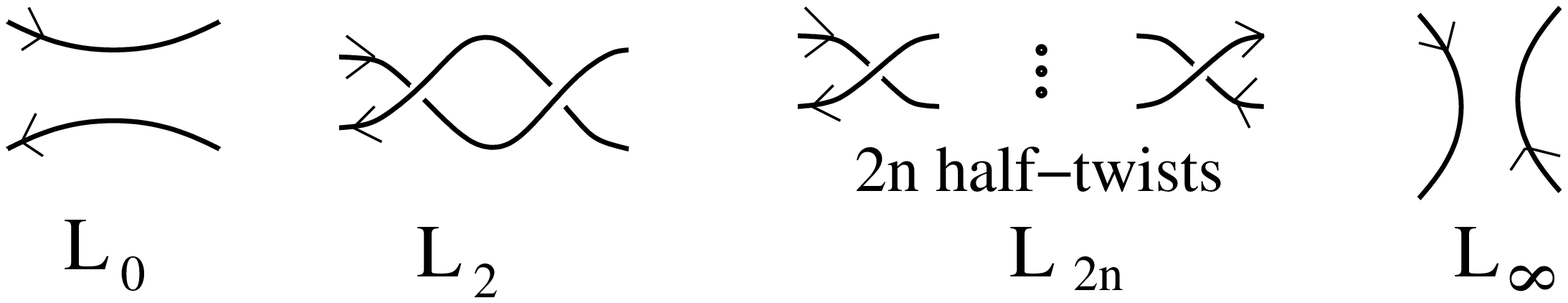,height=2.3cm}
\end{center}
\centerline{Fig. 3.0; Notation for $L_{0}$, $L_{2}$, $L_{2n}$, and $L_{\infty}$.}

Recall that an invariant of links is called an invariant of Conway time if the value of the invariant 
for $L_-$ and $L_0$ gives the value for $L_+$ and the value for $L_+$ and $L_0$ gives the value for $L_-$.  
We introduce now the relation on oriented links which
identifies links that can not be be distinguished by any invariant
of Conway type. This relation is called Conway skein equivalence
and it is denoted by $\sim_c$ \cite{Co-1}.

\begin{definition}
$\sim_c$ is the smallest equivalence relation on ambient 
isotopy classes of oriented 
links which satisfies the following condition:
Let $L_1'$ (respectively $L_2'$) be a diagram of a link $L_1$ (resp.
$L_2$) with a given crossing $p_1$ (resp. $p_2$) such that $\mbox{sgn} p_1 =
\mbox{sgn} p_2$ and  
$$(L_1')^{p_1}_{-\sgn p_1}\sim_c(L_2')^{p_2}_{-\sgn p_2}\ \ and
\ \ (L_1')^{p_1}_0\sim_c (L_2')^{p_2}_0$$
then $L_1\skein L_2$.
\end{definition} 

From the above definition it follows immediately:

\begin{lemma}\label{Lemma III.3.2}
Two oriented links are not (Conway) skein equivalent if there exists
a Conway type invariant which distinguishes them. In particular,
assigning the equivalence class of the Conway relation to a given link 
is an invariant of Conway type.

\end{lemma}

Conway relation can be also described as a ``limit" of the sequence
of relations. Namely:

\begin{enumerate}
\item[($\sim_0$)] $L_1\sim_0 L_2$ if $L_1$ is ambient isotopic to $L_2$.
$$\vdots$$
\item[$\sim_i$] is the smallest equivalence relation on ambient isotopy 
classes of  oriented links
which satisfies the condition:
 
let $L_1'$ (resp.~$L_2'$) be a diagram of a link $L_1$ 
(resp.~$L_2$) with a crossing $p_1$ (resp. $p_2$) such that
 $\sgn p_1=\sgn p_2$ and $(L_1')^{p_1}_{-\sgn
p_1}\sim_{i-1}(L_2')^{p_2}_{-\sgn p_2}$ and
$(L_1')^{p_1}_0\sim_{i-1}(L_2')^{p_2}_0$ then $L_1\sim_i L_2$.
\end{enumerate}

\begin{exercise}\label{Exercise III.3.3}
Show that the smallest relation which contains all the
$\sim_i$ relations is Conway equivalence relation. 

We can modify the $\sim_i$ relation without assuming each time that it is
an equivalence relation. Namely, by introducing relations 
$\approx_0,\approx_1,\ldots,\approx_i,...,\approx_{\infty}$ in the following way:

\begin{description}
\item[$\approx_0$] = $\sim_0$
$$\vdots$$
\item[($\approx_i$)] $L_1\approx_i L_2$
if there exist diagrams $L_1'$ for $L_1$
and $L_2'$ for $L_2$ with crossings $p_1$ and $p_2$ respectively such that
$\sgn p_1 = \sgn p_2$ and  $(L_1')^{p_1}_{-\sgn
p_1}\approx_{i-1}(L_2')^{p_2}_{-\sgn p_2}$ and
$(L_1')^{p_1}_0\approx_{i-1}(L_2')^{p_2}_0$
$$\vdots$$
\item[$\approx_{\infty}$] is defined to be the smallest equivalence relation 
on oriented links which contains all the $\approx_i$ relations.
\end{description}
\end{exercise}

\begin{problem}\label{Problem III.3.4}
\begin{enumerate}
\item[(1)] Are there links which are $\skein$ equivalent 
but not $\approx_{\infty}$ equivalent?
\item[(2)] Are there links which are $\sim_i$ equivalent but are not $\approx_i$
equivalent for any $i, i>0$?
\end{enumerate}
\end{problem}

\begin{exercise}\label{c:3.5}
One could try to obtain new invariants of links by resolving crossings
in pairs, for example, instead of doing it separately. Namely, one could 
deal with 5 diagrams connected with two crossings on a diagram of a link:

$$L^{p\ q}_{\varepsilon_1\ \varepsilon_2},L^{p\ q}_{-\varepsilon_1\
-\varepsilon_2},L^{p\ q}_{0\ -\varepsilon_2},L^{p\ q}_{-\varepsilon_1\
0},L^{p\ q}_{0\ 0}, (\varepsilon_i = +\mbox{ or }-)$$
and assume further that the value of an invariant for $L^{p\
q}_{\varepsilon_1\ \varepsilon_2}$ may be always found on the basis
of the value of an invariant for the remaining four diagrams.
Show that using this method we will not improve the invariants of the
Conway type.\footnote{When over 26 years ago I formulated 
Exercise~\ref{c:3.5} I was not aware of the fact that 
slight modification of the idea included there would have led to the
invariants of links called Vassiliev's invariants 
(or Vassiliev-Gusarov's invariants or the invariants 
of the finite type) \cite{Va-1,Bi-Li,Ba,Piu,P-9}.}

{\em Hint}. We can always add one crossing to the diagram without
changing the class of ambient isotopy of a link.
\end{exercise}

Now let us come back to invariants of the Conway type and to the $\skein$
equivalence. We start from examples of links which  are not ambient isotopic
but which are $\skein$ equivalent.

\begin{lemma}\label{l:3.6}
If $-L$ denotes the link obtained from the $L$ link by changing orientation
of each component of L then $$L\skein -L.$$ In particular 
$P_L(x,y) = P_{-L}(x,y)$ where
$P_L(x,y)$ is a Jones-Conway polynomial.
\end{lemma}

Proof is immediate if one notices that the sign of a crossing is not changed
when we change $L$ to $-L$. So by building the resolving tree (the same for 
$L$ and $-L$) we show even more than Lemma \ref{l:3.6} says: 
$L\approx_{\crs{L}-1}-L$ where
$\crs{L}$ is a minimal number of crossings of diagrams representing $L$.
We leave to a reader further improvement to Lemma \ref{l:3.6} using $cr(L)$ and the number of 
``bad" crossings in $L$ (as in the proof of the main theorem).

\begin{example}\label{p:3.7}
The $L_1$ and $L_2$ links from Fig.3.1 are $\skein$-equivalent
 (one should resolve them at the crossings shown on Fig.3.1 and get
$L_1\approx_1 L_2$). To show that $L_1$ and $L_2$ are not isotopic
one should consider global linking numbers of all its sublinks.
Compare Exercise \ref{Exercise III:3.46}.

\begin{center}
\begin{tabular}{c} 
\includegraphics[trim=0mm 0mm 0mm 0mm, width=.7\linewidth]
{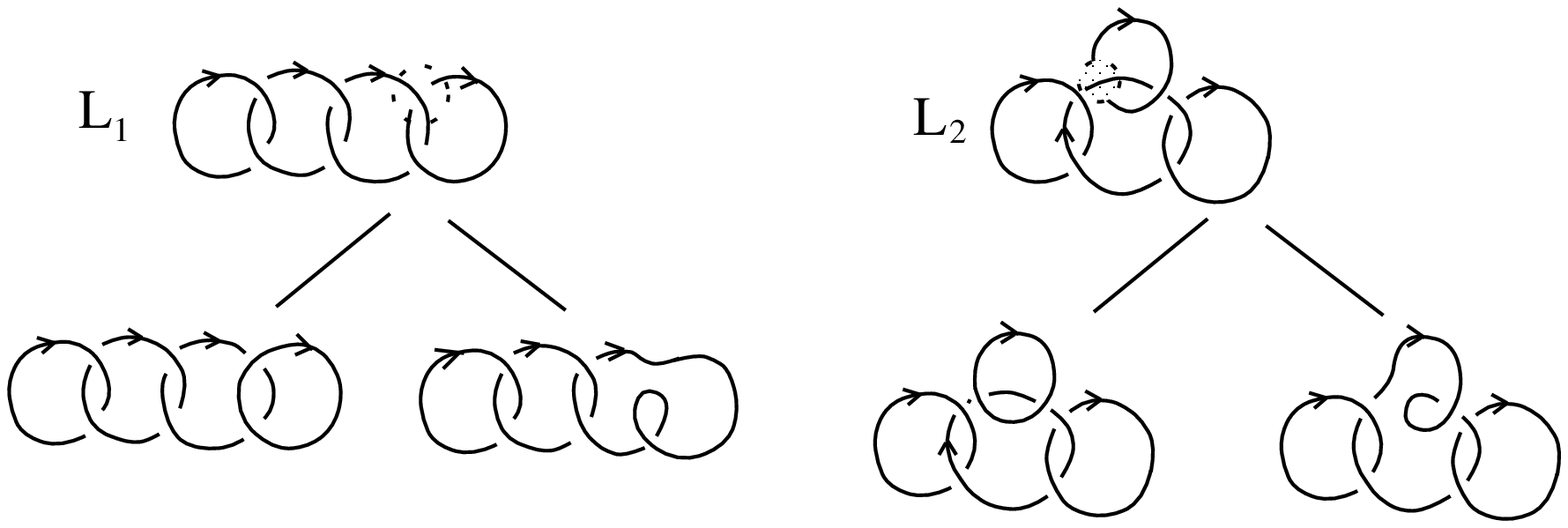}\\
\end{tabular}
\\
Fig. 3.1
\end{center}

\end{example}

For further examples we need the definition of a tangle and mutation \cite{Co-1}:

\begin{definition}\label{2:3.8}

\begin{enumerate}
\item[(1)] A tangle\footnote{We will consider later tangles with $n$ inputs 
and $n$ outputs, called $n$-tangles. Then our tangle will be called 
a 2-tangle.} is a part of a diagram of a link with two inputs and two
outputs (Fig.3.2(a)). It depends on an orientation of the diagram which
arcs are inputs and which ones are outputs. We distinguish tangles with
neighboring inputs and alternated tangles: Fig. 3.2 (b) i (c).

\begin{center}
\begin{tabular}{c} 
\includegraphics[trim=0mm 0mm 0mm 0mm, width=.7\linewidth]
{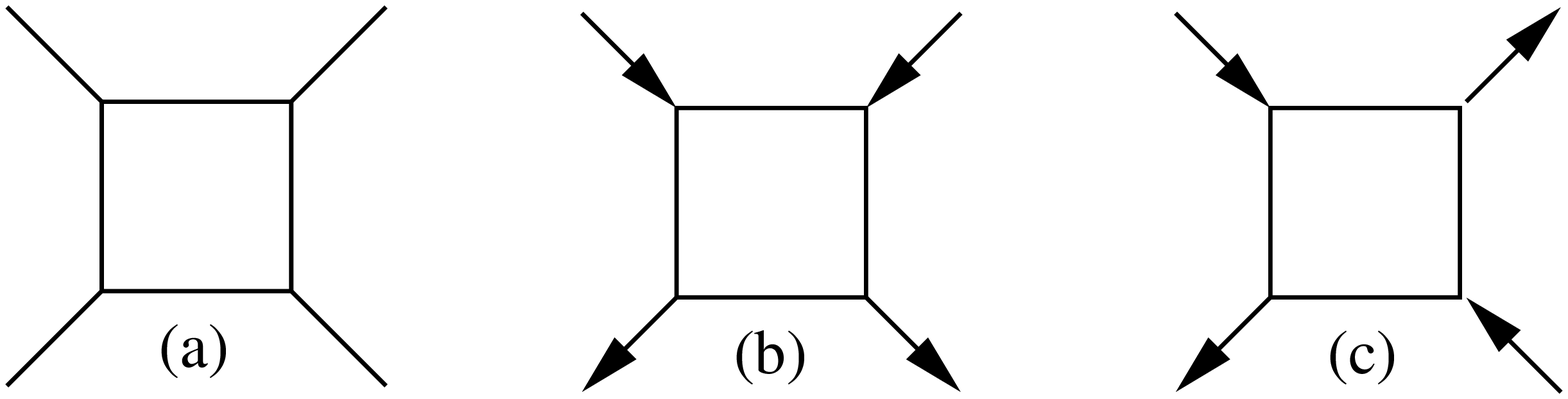}\\
\end{tabular}
\\
Fig. 3.2
\end{center}

\item[(2)] 
The change of an oriented diagram which concerns just one of its tangles
is called a mutation of an oriented diagram which can be of one of 
the following type (see Fig. 3.2(d)):
\begin{enumerate}
\item  $m_z$ is the $180^o$ rotation around the central axis, perpendicular
to the plane of the diagram (axis $z$),   
\item  $m_x$ is the $180^o$ rotation around the horizontal 
        symmetry axis of a square of
a tangle (axis $x$), or
\item  $m_y$ is the$180^o$ rotation around the vertical 
        symmetry axis of a square of a tangle (axis $y$). 
\end{enumerate}
Moreover, in all the cases, if the need be, we change the orientation
of all the components of a tangle into the opposite ones (so that at the
 change of a tangle the inputs and outputs are preserved). 
For example a $m_z$ rotation of a tangle (b) requires the 
change of the orientation of the components
of the tangle (compare Fig.3.2(b)). 
\end{enumerate}
\end{definition}

\begin{center}
\begin{tabular}{c}
\includegraphics[trim=0mm 0mm 0mm 0mm, width=.6\linewidth]
{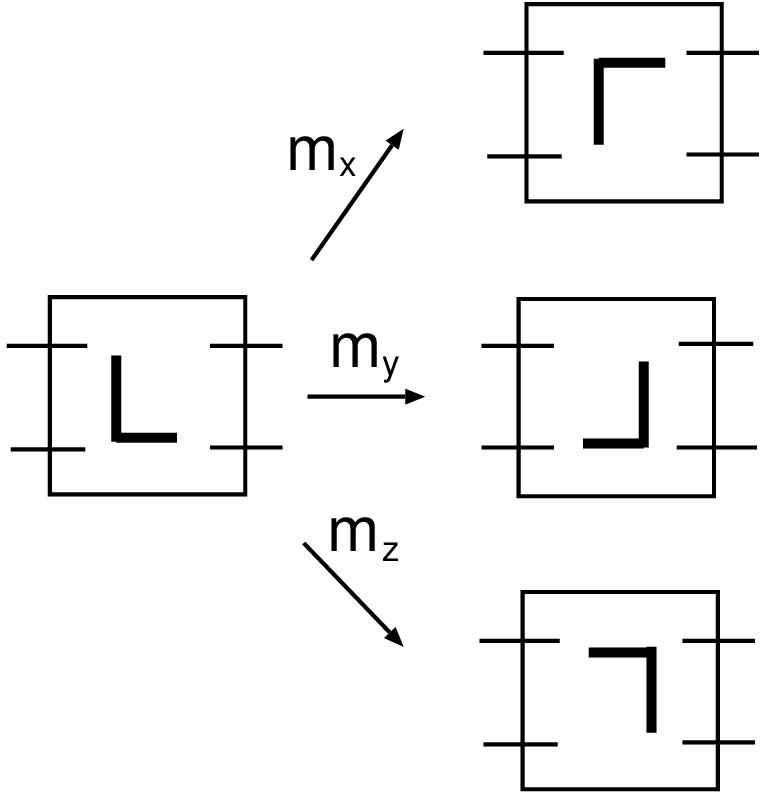}\\
\end{tabular}
\\
Fig. 3.2(d); three mutations, $m_x$, $m_y$, $m_z= m_xm_y$
\end{center}

\begin{lemma}[\cite{L-M-1,Hos-1,Gi}]\label{2:3.9}
Two links, $L_1$ and $L_2$, some diagrams of which differ only by
a mutation are $\skein$ equivalent. More precisely: if cr is the number
of crossings in a square of a mutated tangle, then 
$L_1\approx_{\mbox{cr}-1} L_2$. 
\end{lemma}

Proof.

If $\mbox{cr}\leq 1$ then the tangle is (ambient isotopic to) 
one of the four of Fig.3.3 
(in the figure of the tangle we omit possible trivial circles).

\begin{center}
\begin{tabular}{c} 
\includegraphics[trim=0mm 0mm 0mm 0mm, width=.8\linewidth]
{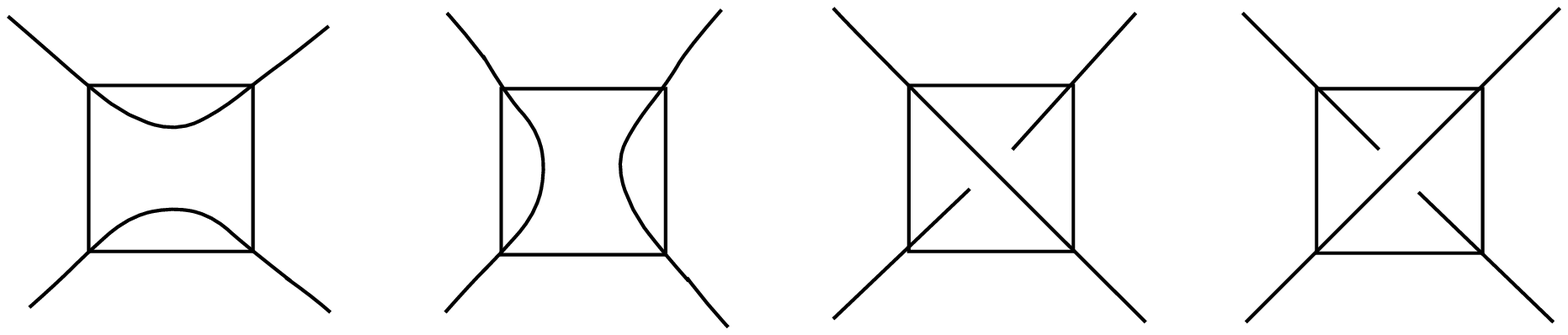}\\
\end{tabular}
\\
Fig.~3.3
\end{center}

Mutation of a link with respect to the tangles of Fig.3.3 does not change
the isotopy class of a link. Further on, to complete the proof, one should use
the standard induction with respect to cr 
and also on the number of bad crossings in the tangle,
just like in the proof of the main theorem ~\ref{2:1.2}).

The first nontrivial example of a mutation (i.e.~the mutation which changes
the isotopy class of a knot) was found for the diagrams of 11 crossings.

\begin{example}\label{2:3.10} The Conway knot (Fig.3.4) and the 
Kinoshita-Tere\-sa\-ka knot (Fig.3.5.) are mutants one of the other
(the mutated tangle is marked at Fig.3.4. and 3.5). So these knots are
\skein-equivalent, even more: they are ${\approx_1}$ 
equivalent. Let us begin resolving the knots from the marked crossings.
D. Gabai \cite{Ga-1} has shown that the above knots have different genera 
(see chapter 3) so they are not isotopic (R. Riley was the first to 
distinguish these knots \cite{Ri}).

\begin{center}
\begin{tabular}{c} 
\includegraphics[trim=0mm 0mm 0mm 0mm, width=.25\linewidth]
{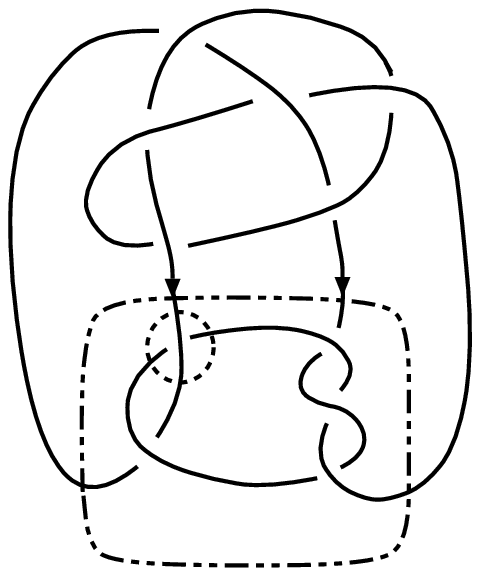}\\
\end{tabular}
\\
Fig. 3.4. The Conway knot.
\end{center}

\begin{center}
{\psfig{figure=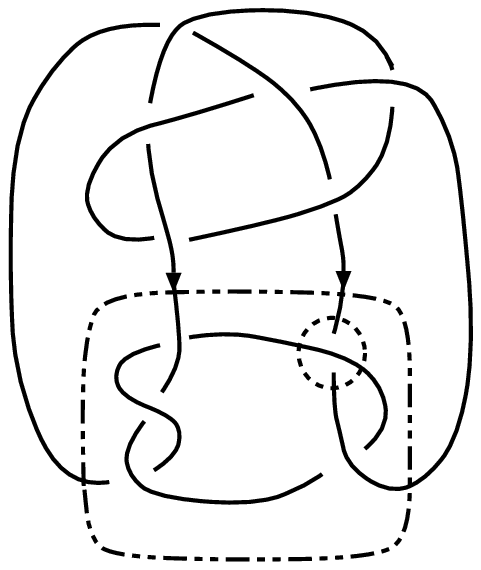,height=4.3cm}}
\end{center}
\begin{center}
Fig. 3.5. The Kinoshita-Terasaka knot.
\end{center}

\end{example}\label{2:3.11}

\begin{example}
Let $(p_1,p_2,\ldots,p_n)$ be a sequence of $n$ integers.
To this sequence we associate an unoriented
link $L(p_1,p_2,\ldots,p_n)$ which we call pretzel link
$L(p_1,p_2,\ldots,p_n)$ and which we present on Fig.~3.6.

\begin{center}
\begin{tabular}{c} 
\includegraphics[trim=0mm 0mm 0mm 0mm, width=.7\linewidth]
{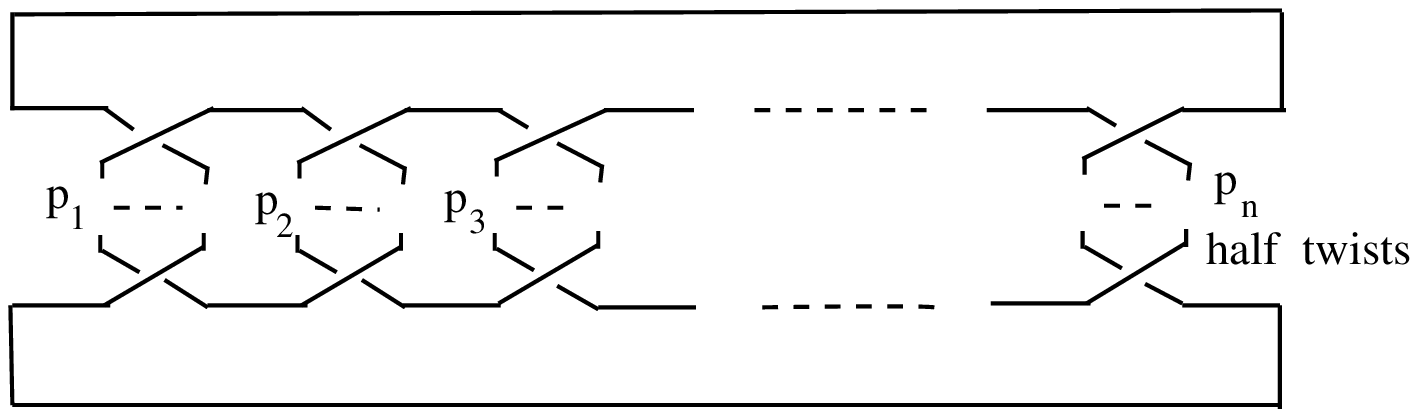}\\
\end{tabular}
\\
Fig.~3.6
\end{center}

Further, let $L(p^{\varepsilon(1)}_1,p^{\varepsilon(2)}_2,\ldots,p^{\varepsilon(n)}_n)$
denote the pretzel link $L(\row{p}{n})$ oriented in such a way that 
$\varepsilon(i)=1$ if all crossings in the $i$-th column 
are positive, and 
$\varepsilon(i)=-1$ if they are negative. To complete the definition
of the orientation of the link, we assume that the upper arc of the link
is oriented ``from the right- to the left-hand-side''
(c.f.~Fig.~3.7). Clearly, not all possible sequences
$p^{\varepsilon(1)}_1,p^{\varepsilon(2)}_2,\ldots,p^{\varepsilon(n)}_n$
can be realized by oriented pretzel links. For example, if 
all $p_i$ and $n$ are odd then $L(p_1,p_2,\ldots,p_n)$ is a knot and 
$\varepsilon(i) = -sgn(p_i)$.\ Because of Lemma~\ref{2:3.9}
it follows that for any permutation $\delta_n \in S_n$ we have
$$L(p^{\varepsilon(1)}_1,p^{\varepsilon(2)}_2,
\ldots,p^{\varepsilon(n)}_n)\skein
L(p^{\varepsilon(\delta(1))}_{\delta(1)},
p^{\varepsilon(\delta((2))}_{\delta(2)},\ldots,
p^{\varepsilon(\delta((n))}_{\delta(n)}).$$
namely, any permutation may be presented by a sequence of transpositions
of neighboring numbers and any such a transposition can be realized
as a mutation of the pretzel link.

In particular we can go from the pretzel link of two components
$L(3,5,-5^{-1},-3^{-1},-3^{-1})$ to its mirror image 
$L(-3^{-1},-5^{-1},-3^{-1},5,3,3)$ using finite number of mutations.
However these links are not ambient isotopic.

\begin{center}
\begin{tabular}{c} 
\includegraphics[trim=0mm 0mm 0mm 0mm, width=.7\linewidth]
{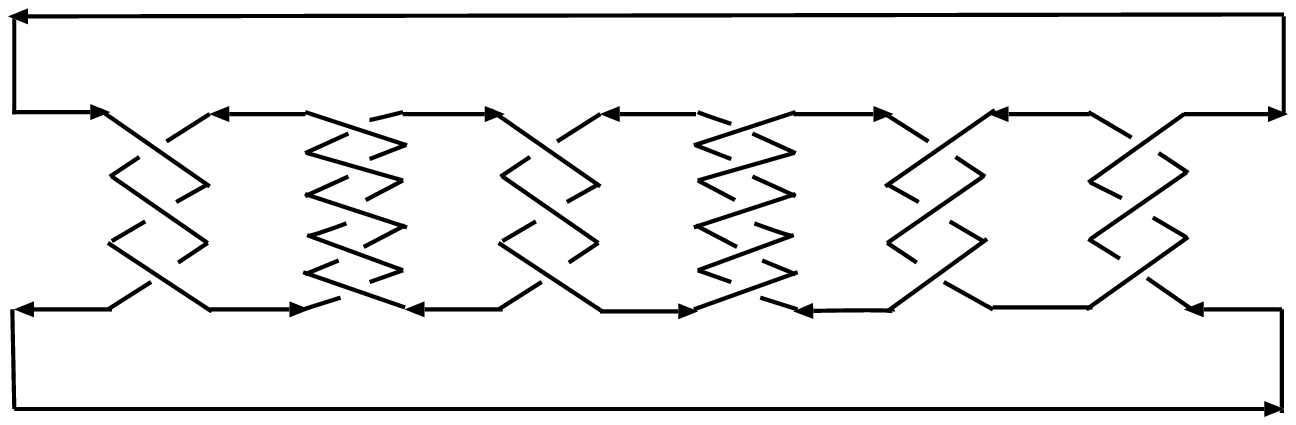}\\
\end{tabular}
\\
Fig. 3.7. $L(-3^{-1},-5^{-1},-3^{-1},5,3,3)$
\end{center}

\end{example}

\begin{example}\label{2:3.12}
Consider a diagram of a link $D$ with two alternating tangles 
inserted as in Fig. 3.9.
We use the following convention: $[2n]$ or $2n$ in a square 
denote tangle as in Fig.3.8.
\begin{center}
\begin{tabular}{c} 
\includegraphics[trim=0mm 0mm 0mm 0mm, width=.4\linewidth]
{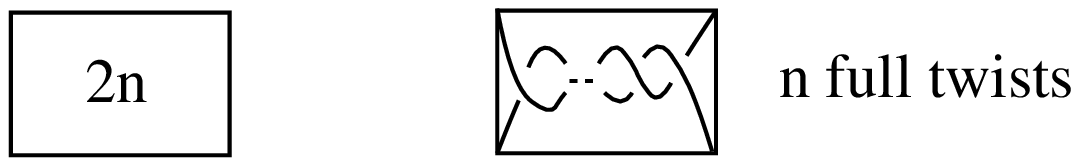}\\
\end{tabular}
\\
Fig. 3.8
\end{center}

Moreover, we  write 
 \parbox{0.9cm}{\psfig{figure=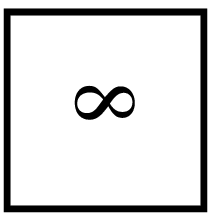,height=0.7cm}} 
for \parbox{1.1cm}{\psfig{figure=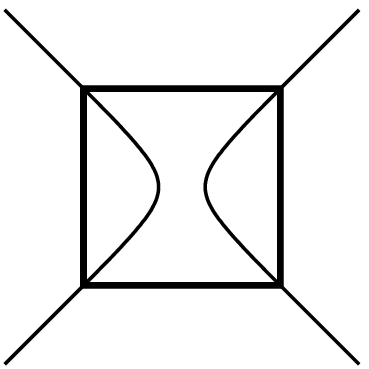,height=1.1cm}}.
Let $D_{p,q}$ denote the diagram obtained from $D$ by putting $2p$ 
into the first tangle and $2q$ into the second. Let us assume moreover that
the diagram $D_{\infty,q}$ is equivalent to $D_{p,\infty}$ for every 
$p$ i $q$. Then for $p+q=p'+q'$ we have $D_{p,q}\skein D_{p',q'}$. 
An example which satisfies the above condition was found by 
T.~Kanenobu \cite{Ka-1} (Fig.3.9).

\begin{center}
\begin{tabular}{c} 
\includegraphics[trim=0mm 0mm 0mm 0mm, width=.35\linewidth]
{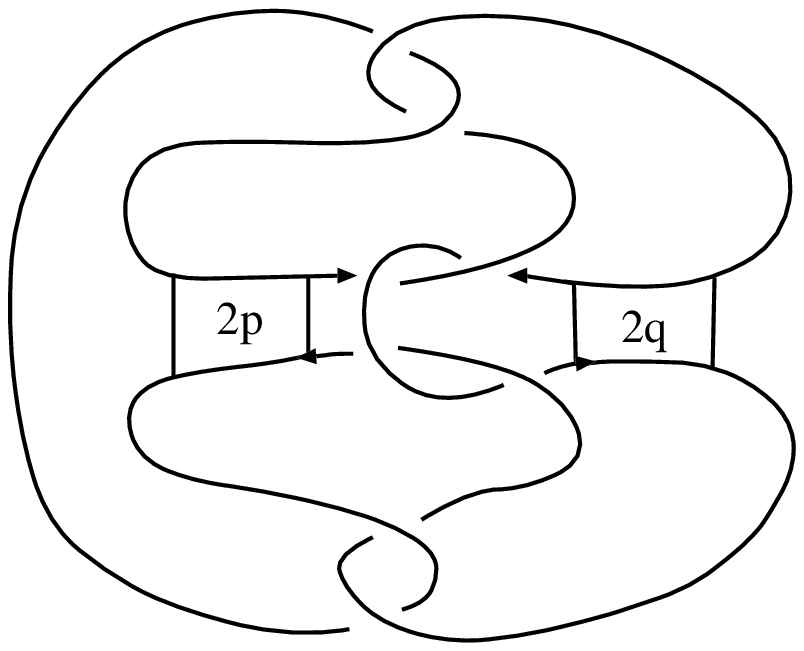}\\
\end{tabular}
\\
Fig.~3.9
\end{center}

It is easy to notice that in this example $D_{\infty,q}$ and
$D_{p,\infty}$ are trivial links of two components.
Kanenobu has shown that $D_{p,q}$ is isotopic to $D_{p',q'}$ if and only
if $(p,q)=(p',q')$ or $(q',p')$. That is how we obtain the next family
of examples of nonisotopic links, which are $\skein$-equivalent. 
Kanenobu proof uses the value of the Jones-Conway polynomial for $D_{p,q}$ 
(see example~\ref{c:3.13}) and the structure of the Alexander module for
$D_{p,q}$ (compare Chapter IV). 
\end{example}

Proof of the statement from Example 3.12 can be given by a standard
induction on $|p-p'|$ showing $D_{p,q}\approx_{|p-p'|} D_{p',q'}$.

\begin{example}\label{c:3.13}
 Compute the Jones-Conway polynomial for $D_{p,q}$ links of Fig.3.9
and show that $P_{D_{p,q}}(x,y) = P_{D_{p',q'}}(x,y)$ 
if and only if 
$p+q = p'+q'$.
\end{example}

\begin{exercise}\label{2:3.14}
Show that $D_{p,q}$ is isotopic to $D_{q,p}$ for links on Fig.3.9.

Hint. Show that Fig.3.10 pictures a link isotopic to $D_{p,q}$.

\begin{center}
\begin{tabular}{c} 
\includegraphics[trim=0mm 0mm 0mm 0mm, width=.35\linewidth]
{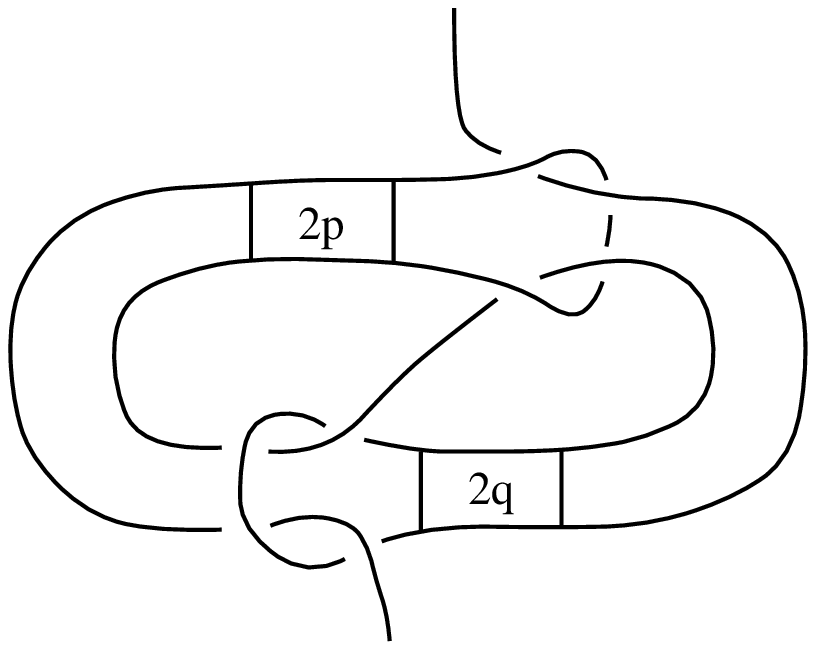}\\
\end{tabular}
\\
Fig. 3.10
\end{center}

\end{exercise}

\begin{exercise}\label{2.3.15}
Show that $D_{p,q}$ link (in Fig. 3.9) is amphicheiral 
(i.e.~isotopic with its mirror image) if and only if $p+q = 0$.
\end{exercise}

The Kanenobu example may be slightly extended if we allow odd numbers in 
tangles of the link in Fig.~3.9.
\cite{Ka-2}.

Let us denote a link the diagram of which is pictured at Fig.~3.9 
by $K(m,n)$, where numbers $m$ and $n$ appear in the tangles.
So we have $D_{p,q}=K(2p,2q)$. The orientation of
$K(m,n)$ is just  implied by the left tangle of Fig.3.9.

\begin{exercise}\label{2:3.16}
Show that $K(m,n)=-\overline{K}(-m,-n)$.
\end{exercise}

\begin{exercise}\label{2:3.17}
Show that $K(m,n)=-\overline{K}(-n,-m) = K(n,m)$.
\end{exercise}

\begin{exercise}\label{2:3.18}
Show that $K(m,n)\skein K(m',n')$ if and only if $m+n =
m'+n'$ and when $m+n$ is even then $m\equiv m' (\mbox{mod } 2)$.
\end{exercise}

With the help of a computer M.B. Thistlethwaite has shown that among the
12966 knots with at most 13 crossings there are 30 with the Conway
polynomial $1+2z^2+2z^4$. Examination of these failed to find a pair of knots
distinguished by the Homflypt  (Jones-Conway) polynomial but not by the Jones polynomial.
As a byproduct of these computations Lickorish and Millett \cite{L-M-1} 
have found the following example.

\begin{example}\label{Example III:3.19}
Let us consider three links pictured on Fig.~3.11  and denoted according
to \cite{Ro-1} by $8_8$ and $10_{129}$, and also $13_{6714}$ from \cite{This-1}.

\begin{center}
\begin{tabular}{c} 
\includegraphics[trim=0mm 0mm 0mm 0mm, width=.65\linewidth]
{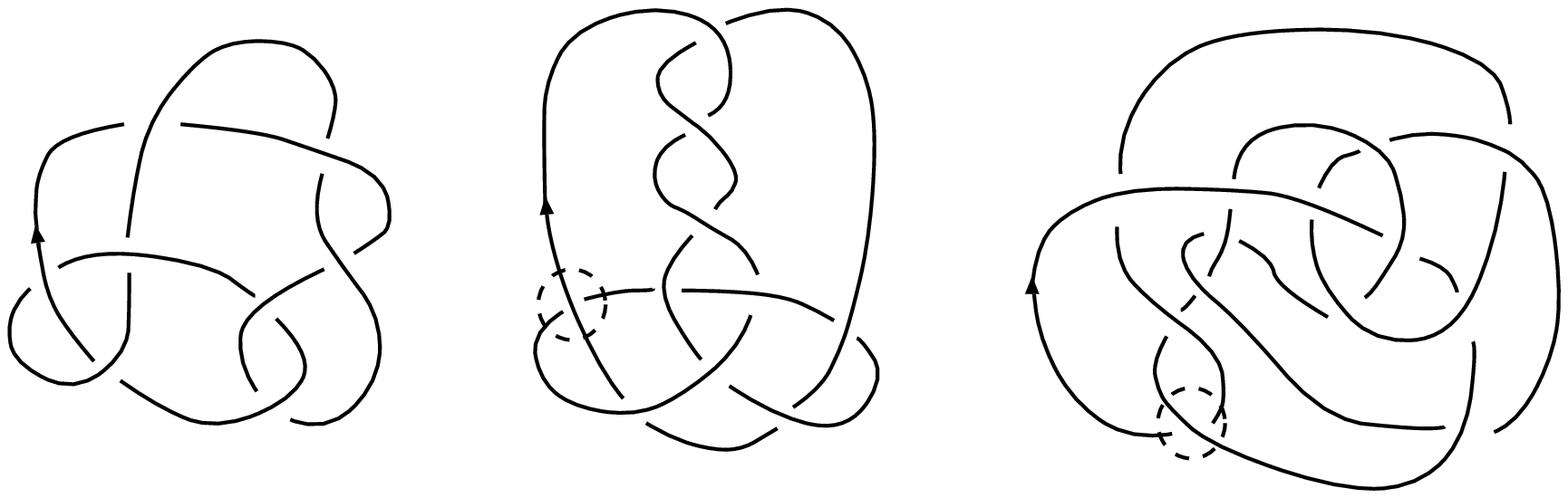}\\
\end{tabular}
\\
Fig. 3.11
\end{center}

Now changing the encircled crossing of $13_{6714}$ produces $10_{129}$ 
and smoothing that crossing produces $T_2$ the trivial
link of 2 components. Similarly, changing the encircled crossing in $10_{129}$
gives $8_8$ and smoothing it gives $T_2$. Hence we have triples 
$(13_{6714},10_{129},T_2)$ and $(8_8,10_{129},T_2)$, both of the form 
$L_+,L_-,L_0$. Therefore $8_8\skein 13_{6714}$. Lickorish and Millett 
found that $10_{129}$ and $\overline{8}_8$ (the mirror image of $8_8$) 
have the same Jones-Conway polynomial and they asked if these were $\skein$
equivalent (in \cite{L-M-1}). Kanenobu has given the positive answer to 
this question showing that knots $8_8,10_{129}$ and $13_{6714}$
are special cases of his $K(m,n)$ link \cite{Ka-2}.
\end{example}

\begin{exercise}\label{2:3.20}
Show that $8_8\approx K(0,-1)$, $10_{129}\approx K(2,-1)$ and
$13_{6714}\approx K(2,-3)$.
\end{exercise}

Examples which we have described so far have shown limitations
of invariants of Conway type. Still, it does not change the fact that
for example the Jones-Conway polynomial remains the best single invariant
of links. Only the new Kauffman polynomial (discovered in August 1985) 
may compete with it (compare \S 5).

For quite some time the question remained open whether the 
Jones-Conway polynomial is better then the classical Conway polynomial
(compare chapter 3 and 4) or the Jones polynomial. M.B. Thistlethwaite 
searched the tables of knots and found out \cite{L-M-1} that, for example,
a knot of 11 crossings ($11_{388}$ in tables of knots; compare 
\cite{Per} Fig.3.12), may be differentiated from its mirror image by
 the Jones-Conway polynomial but not by the Jones polynomial or by Conway
polynomial. Because we can compute that: 

\begin{center}
\begin{tabular}{c} 
\includegraphics[trim=0mm 0mm 0mm 0mm, width=.25\linewidth]
{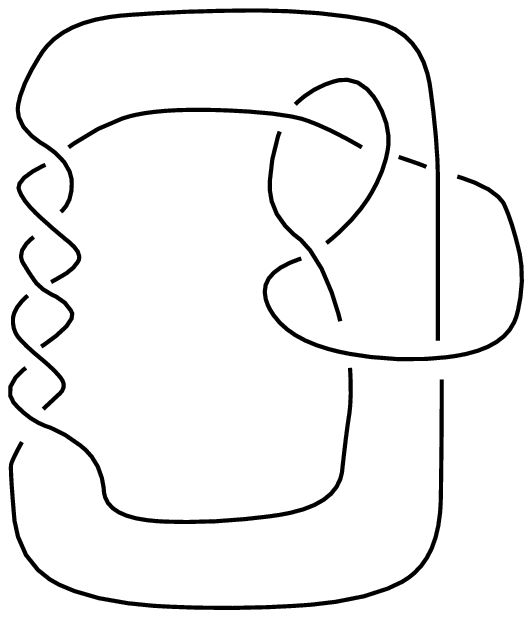}\\
\end{tabular}
\\
Fig. 3.12. Knot $11_{388}$.
\end{center}
\begin{eqnarray*}
&P_{\overline{11}_{388}}(x,y)& =\\
&5x^{-1}y-4x^{-1}y^{-1}+4x^{-2}y^{2}-10x^{_2}+x^{-2}y^{-2}&+\\
+&x^{-3}y^{3}-5x^{-3}y+6x^{-3}y^{-1}+x^{-4}-x^{-4}y^{-2}+3&\\
&V_{11_{388}}(t) &=\\
& t^{-2}-t^{-1}+1-t+t^2&=\\
&\nabla_{11_{388}}(z) &=\\
& 1-z^2-4z^4-z^6&\\
\end{eqnarray*}

and use the following lemma:

\begin{lemma}\label{III:3.21}
If $\overline{L}$ link is a mirror image of link $L$ then their
Jones-Conway polynomials satisfy equality:
$$P_{\overline{L}}(x,y) = P_L(y,x).$$
In particular, for Jones polynomials we have:
$$V_{\overline{L}}(t) = V_L(\frac{1}{t})$$
and for Conway polynomials:
$$\nabla_{\overline{L}}(z) = \nabla_L(-z).$$
\end{lemma}

Proof of the lemma is an easy consequence of the observation 
that the sign of each crossing changes on the way from a 
diagram to its mirror image. 

Based on the same idea, Lemma~\ref{III:3.21} can be partially generalized
to other invariants yielded by a Conway algebra.

\begin{lemma}\label{2:3.22}
Let ${\cal A} = \{A,a_1,a_2,\ldots,|,\star\}$ be a Conway algebra such
that there exists an involution on $A$ (i.e.~a mapping, the square of which is 
identity) $\tau :A\rightarrow A$ satisfying:
\begin{enumerate}
\item[(i)] $\tau (a_i)  = a_i$,
\item[(ii)] $\tau (a|b) = \tau (a)\star\tau (b)$.
\end{enumerate}

Then the invariant of links $A_L$, assigned to this algebra satisfies:
$$A_{\overline{L}} = \tau (A_L).$$
\end{lemma}

In Examples  1.5 and 1.6 the involution $\tau$ is the identity.
In Example \ref{p:1.8} (defining the  Jones-Conway polynomial) the involution
is the change of positions of $x$ and $y$. On the other hand in Example
1.11 $\tau (n,z)=(n,-z)$; the algebra of Example
1.7 has no involution.

It is worthwhile to note that in a Conway algebra build of 
terms (i.e.~``reasonable'' words build by the 
alphabet consisting of $a_1,a_2,\ldots$ and $|,\star$) 
the involution $\tau$ exists and it is 
uniquely determined by the conditions (i) and (ii). It is because
$\tau$ preserves axioms of the Conway algebra. 
The Conway algebra build of terms is universal for Conway algebras,
that is, there is exactly one homomorphism from it 
into any other Conway algebra.

\begin{remark}\label{2:3.23}
Suppose that ${\cal A}$ is a Conway algebra.
It may happen that for every pair
$u,v\in A$ there exists exactly one element $w\in A$ such that
$v|w=u$ and $u\star w=v$. Then we can introduce a new operation 
$\circ: A\times A \rightarrow A$ putting $u\circ v = w$ (it occurs in 
Examples  1.6, \ref{p:1.7} i \ref{p:1.8} but not in examples
 \ref{Example III:1.5} and \ref{p:1.10}). Then $a_n =
a_{n-1}\circ a_{n-1}$. If operation $\circ$ is well defined then we have
an easy formula for invariants of connected and disjoint sums of links.
In agreement with a general terminology we call a Conway algebra 
for which $\circ$ is well defined a Conway quasigroup. 
\end{remark}

\begin{definition}\label{2:3.24}
A link $L$ is called a splittable link ( splittable to $L_1$ and $L_2$) if $L$ is
a union of two non-empty sublinks $L_1$ and $L_2$, and moreover
there exist two disjoint 3-dimensional balls $B_1,B_2\subset S^3$ such that
$L_1\subset B_1$ and $L_2\subset B_2$. 
In such a case we say that $L$ is a 
disjoint sum of $L_1$ and $L_2$, and we write 
$L=L_1\sqcup L_2$. 
\end{definition}

\begin{theorem}\label{Theorem III:3.25}
If $L=L_1 \sqcup L_2$ then
$$P_{L_1\sqcup L_2}(x,y) = (x+y)P_{L_1}(x,y)P_{L_2}(x,y)$$ 
where $P_L(x,y)$ denotes the Jones-Conway polynomial of a link $L$.
\end{theorem}

Proof.

There is a diagram $L$ in which $L_1$ may be separated
from $L_2$ by an ordinary closed curve. We call it a splittable diagram.
Now we will prove Theorem \ref{Theorem III:3.25} for splittable diagrams.

We use the induction with respect to lexicografically ordered
pairs $(\crs{L}, b(L))$, where 
$\crs{L}$ denotes the number of crossings in a diagram
and $b(L)$ stands for 
the minimal number of bad crossings over all possible choices of base
points.

If $b(L) = 0$ then the theorem~\ref{Theorem III:3.25} holds because $L$
is a trivial link with $n(L)$ components and $L_1$ and $L_2$ are trivial links
of $n(L_1)$ and $n(L_2)$ components respectively. Thus by the definition

\begin{eqnarray*}
&P_L(x,y)&=\\
&(x+y)^{n(L)-1}&=\\
&(x+y)(x+y)^{n(L_1)-1}(x+y)^{n(L_2)-1}&=\\
&(x+y)P_{L_1}(x,y)P_{L_2}(x,y).&\\
\end{eqnarray*}

Let us assume that we have already proved Theorem~\ref{Theorem III:3.25} for 
splittable diagrams satisfying $(\crs{L},b(L))<(c,b)$.
Let $p$ be a bad crossing of diagram $L$ (for example, let $\sgn
p = +$). Then for $L^p_-$ and $L^p_0$ the theorem is true 
by the inductive hypothesis.
Let us assume, for example, that $p\in L_1$. Then

\begin{eqnarray*}
&P_L(x,y)&=\\
&P_{L^p_+}(x,y)&=\\
&\frac{1}{x}(P_{L^p_0}(x,y) - yP_{L^p_-}(x,y))&=\\
&\frac{1}{x}((x+y)P_{(L_1)^p_0}(x,y)\cdot P_{L_2}(x,y) -
y(x+y)P_{(L_1)^p_-}(x,y)\cdot P_{L_2}(x,y))&=\\
&(x+y)P_{(L_2)^p_0}(x,y)\cdot(\frac{1}{x}P_{(L_1)^p_0}(x,y) -
yP_{(L_1)^p_-}(x,y)))&=\\
&(x+y)P_{L_2}(x,y)\cdot P_{L_1}(x,y)&\\
\end{eqnarray*}
which completes the proof.

In the other cases we proceed similarly.

\begin{definition}\label{2:3.26}
An oriented link $L$ is a connected sum of two links
$L_1$ and $L_2$ (we denote it $L = L_1\# L_2$) if there exists
a sphere $S^2\subset S^3$ which divides $S^3$ into two 3-dimensional
balls $B_1$ and $B_2$ in such a way that $S^2$ meets $L$ transversally in 
two points, and if $\beta$ is an arc in $S^2$ joining these two points
then $(B_1\cap L)\cup\beta$ (resp.~$(B_2\cap L)\cup\beta$) 
is isotopic to $L_1$ (resp.~$L_2$).
\end{definition}

In Chapter IV we analyze a connected sum in detail and in particular 
we will show that it is uniquely defined for knots (with the sum knots form an abelian semigroup 
with cancellation property).

\begin{corollary}\label{2:3.27}
If $L = L_1\# L_2$ then
$$ P_L(x,y) = P_{L_1}(x,y)\cdot P_{L_2}(x,y).$$
\end{corollary}

Proof.

We can find a diagram of $L$ as presented on Fig.3.13. Let us rotate $L_2$ 
by $180^o$ twice: once clockwise, next counterclockwise to get two diagrams
$L_+$ and $L_-$ (Fig.3.14). Of course $L_+$ and $L_-$ are isotopic to
$L$ and the diagram $L_0$ (Fig. 3.14) is the disjoint sum of $L_1$ and $L_2$, 
so:

\begin{center}
\begin{tabular}{c} 
\includegraphics[trim=0mm 0mm 0mm 0mm, width=.3\linewidth]
{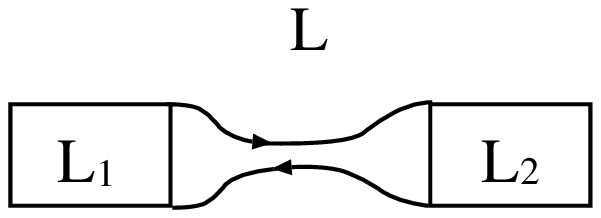}\\
\end{tabular}
\\
Fig.~3.13
\end{center}

\begin{center}
\begin{tabular}{c} 
\includegraphics[trim=0mm 0mm 0mm 0mm, width=.65\linewidth]
{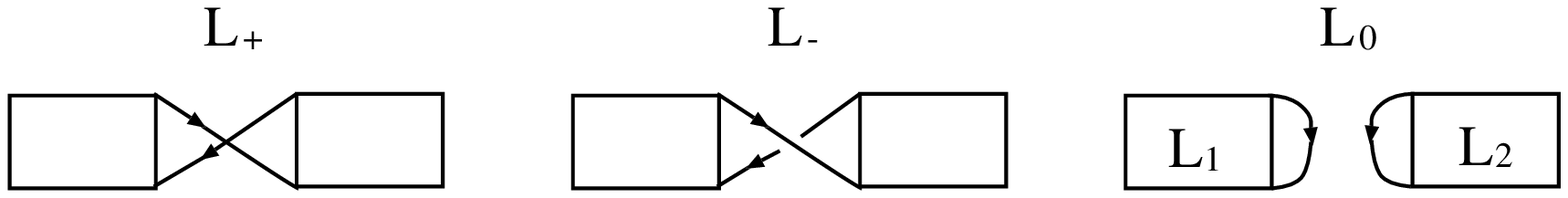}\\
\end{tabular}
\\
Fig.~3.14
\end{center}

$$xP_L(x,y)+yP_L(x,y) = P_{L_1\sqcup L_2}(x,y),$$
and therefore
$$ (x+y)P_{L_1\kwad L_2}(x,y) = P_{L_1\sqcup L_2}(x,y),$$
This formula and Theorem \ref{Theorem III:3.25} gives us Corollary~\ref{2:3.27}.

Theorem~\ref{Theorem III:3.25} and Corollary~\ref{2:3.27} may be generalized to cover
the case of $A_L$ invariants yielded by Conway algebras with the operation
$\circ$. First, let us observe that adding a trivial knot to the given link
$L$ changes $A_L$ for $A_L \circ A_L$ (in short $A^2_L$); Fig.3.15. 
In particular we obtain a known equality $a^2_i = a_{i+1}$. 
Considering Fig.3.14 we get more general formula:
$$A_{L_1\sqcup L_2} = A^2_{L_1\kwad L_2}.$$

\begin{center}
\begin{tabular}{c} 
\includegraphics[trim=0mm 0mm 0mm 0mm, width=.65\linewidth]
{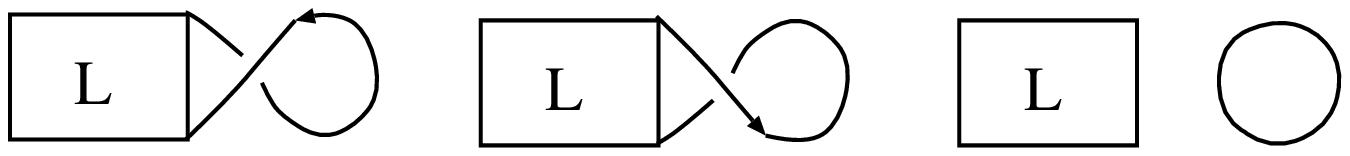}\\
\end{tabular}
\\
Fig. 3.15
\end{center}

Using a method similar to that of Theorem~\ref{Theorem III:3.25} and Corollary
~\ref{2:3.27} we can prove the following:

\begin{lemma}\label{III:3.28}
If a  Conway algebra ${\cal A}$ admits the operation $\circ$ and for each $w\in A$ 
there is a homomorphism $\phi_w:A\rightarrow A$ such that 
$$\phi_w(a_1) = w, \phi_w(a_2)=w^2,\phi_w(a_3) =
w^4,\ldots,\mbox{etc}$$
then
$$A_{L_1\kwad L_2} = \phi_{A_{L_1}}(A_{L_2}) = \phi_{A_{L_2}}(A_{L_1})$$
$$A_{L_1\sqcup L_2} = (\phi_{A_{L_1}}(A_{L_2}))^2 = (\phi_{A_{L_2}}(A_{L_1}))^2$$
\end{lemma}

\begin{exercise}\label{2:3.29}
Show that the algebras of Examples 1.6, \ref{p:1.7} and
\ref{p:1.8} satisfy the assumptions of Lemma~\ref{III:3.28}.
\end{exercise}

\begin{problem}\label{2:3.30}

\begin{enumerate}
\item [(i)] 
Let us consider the equation $a|x = b$ in the universal Conway algebra.
Is it possible for this equation to have more than one solution? 
(the equation $a_1|x = a_2$  may have no solution at all).

\item[(ii)] Let us assume that for certain diagrams of $L$ and $L'$ 
and for certain crossings we have $L_+ \skein L_+'$ 
and $L_- \skein L_0'$ are true.
Is then $L_0 \skein L_0'$ true as well?
\end{enumerate}
\end{problem}

Corollary~\ref{2:3.27} can be generalized in the following way:

\begin{definition}\label{2:3.31}

\begin{enumerate}
\item [(i)] Consider an alternating tangle $A$. There are two methods of
obtaining a link from the tangle $A$. They are marked $N(A)$ (numerator) 
and $D(A)$ (denominator) according to Fig.3.16.

Let $A^N$ denote $P_{N(A)}(x,y)$ and $A^D = P_{D(A)}(x,y)$ (i.e.~the respective
value of the  Jones-Conway polynomial).

\begin{center}
\begin{tabular}{c} 
\includegraphics[trim=0mm 0mm 0mm 0mm, width=.35\linewidth]
{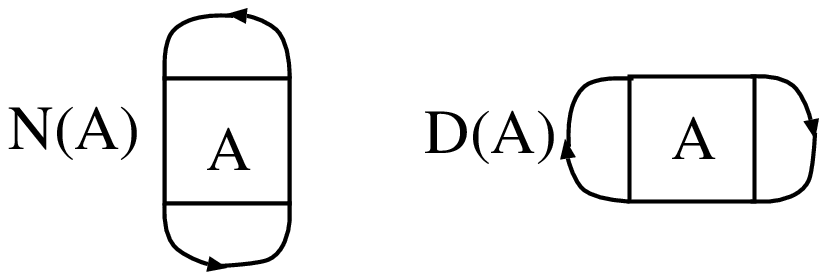}\\
\end{tabular}
\\
Fig. 3.16
\end{center}

\item [(ii)]
Having two alternating tangles $A$ and $B$ we can define their sum
just like on Fig.3.17. Let us notice that $D(A+B) = D(A)\kwad D(B)$.

\begin{center}
\begin{tabular}{c} 
\includegraphics[trim=0mm 0mm 0mm 0mm, width=.35\linewidth]
{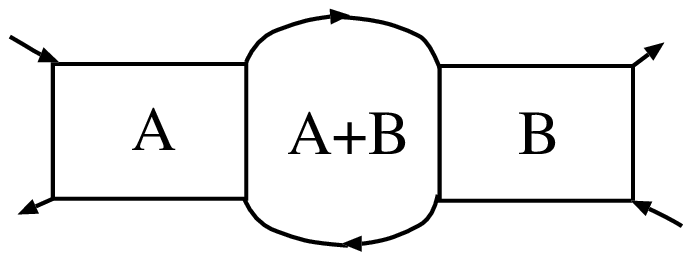}\\
\end{tabular}
\\
Fig. 3.17
\end{center}

\end{enumerate}
\end{definition}

Lickorish and Millett \cite{L-M-1} have generalized 
Conway's result \cite{Co-1}
concerning Conway polynomial by showing:

\begin{theorem}\label{2:3.32}

\begin{enumerate}
\item $(1-(x+y)^2)(A+B)^N = (A^NB^D+A^DB^N)-(x+y)(A^NB^N+A^DB^D)$
\item $(A+B)^D = A^D \cdot B^D$
\end{enumerate}
\end{theorem}

Proof.

Part (ii) has been already proved in Corollary ~\ref{2:3.27}.

Part (i) will be proved by induction with respect to ordered pairs
$(\crs{B},b(B))=$
(number of crossings in the tangle $B$, minimal number of bad crossings
in $B$), similarly as in the proof of Theorem~\ref{Theorem III:3.25}. 
For a tangle $B$ we can find
a resolving tree leaves of which are the tangles of Fig.3.18,
possibly with a certain number of trivial circles. 

\begin{center}
\begin{tabular}{c} 
\includegraphics[trim=0mm 0mm 0mm 0mm, width=.5\linewidth]
{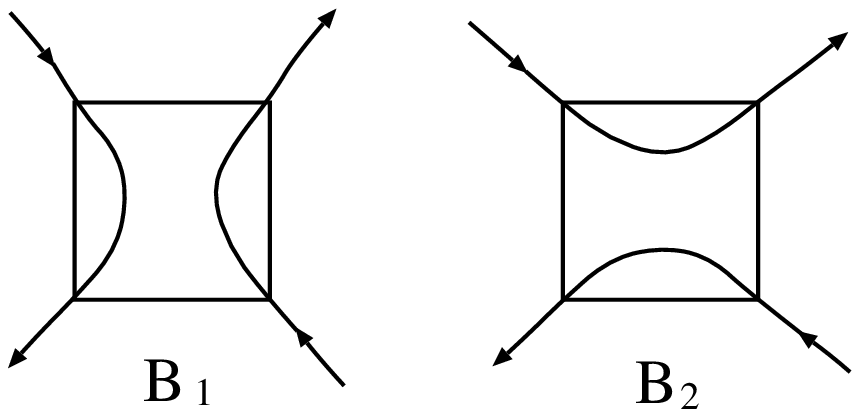}\\
\end{tabular}
\\
Fig. 3.18
\end{center}

Because the trivial circles appear in $A+B$ as well, they can be omitted
in further considerations. As $N(B_1)$ and $D(B_2)$ are trivial knots
and both $D(B_1)$ and $N(B_2)$ are trivial links of two components,
and moreover $N(A+B_1) =  D(A)$ and $N(A+B_2) =  N(A)$ so that $B^N_1 = B^D_2 = 1,
B^D_1 = B^N_2 = (x+y), (A+B_1)^N = A^D$ as well as $(A+B_2)^N = A^N$. 
It follows that $$(1-(x+y)^2)(A+B_1)^N = (1-(x+y)^2)A^D = (A^N+(x+y)+A^D) -
(x+y)(A^N+A^D(x+y))$$
similarly
$$(1-(x+y)^2)(A+B_2)^N = (1-(x+y)^2)A^N = (A^N+(x+y)+A^D) -
(x+y)((x+y)A^N+A^D).$$

Thus we have proved Theorem~\ref{2:3.32}(i) for $A+B_1$ and
$A + B_2$. 

Now the immediate verification shows that if the formula holds for
$B_-$ and $B_0$, then it is true for $B_+$ as well, and similarly if it holds
for $B_+$ and $B_0$, then it is true for $B_-$. This allows the inductive
 step and completes the proof of Theorem ~\ref{2:3.32}(i).

\begin{exercise}
Show that the Jones-Conway polynomial distinguishes square knot and granny
knot (Fig.3.19). 

\begin{center}
\begin{tabular}{c} 
\includegraphics[trim=0mm 0mm 0mm 0mm, width=.65\linewidth]
{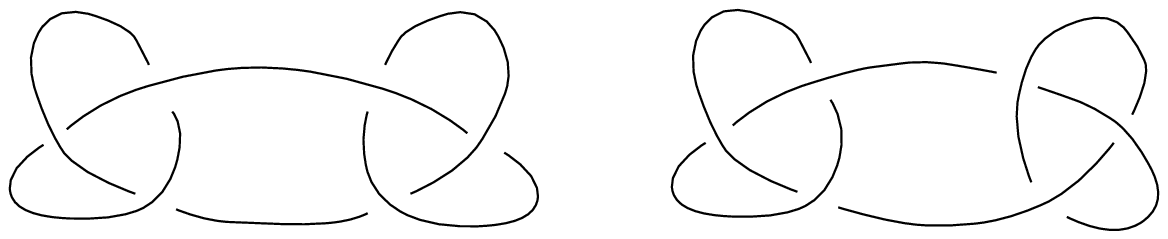}\\
\end{tabular}
\\
Fig.~3.19
\end{center}

\end{exercise}

\begin{corollary}[\cite{Co-1}.] 
Let us define the fraction of a tangle $A$ as follows:
$$F(A) = \frac{\nabla_{N(A)}(z)}{\nabla_{D(A)}(z)}$$
where $\nabla (z)$ is a Conway polynomial. We do not cancel common factors of
numerator and denominator. 

Then $F(A+B) = F(A) + F(B)$.
\end{corollary}

\begin{example}
Let $A$ be a tangle pictured at Fig.~3.20. Then $F(A) = \frac{z}{1}$.

\begin{center}
\begin{tabular}{c} 
\includegraphics[trim=0mm 0mm 0mm 0mm, width=.65\linewidth]
{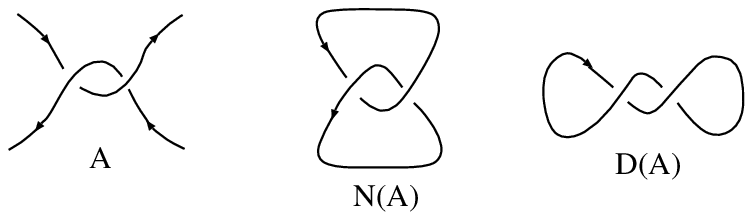}\\
\end{tabular}
\\
Fig. 3.20
\end{center}

\end{example}

\begin{problem}
Let ${\cal A}$ be a Conway algebra with the operation $\circ$ and the
homomorphism $\phi_w$. Find the formula for value of the invariant
yielded by the algebra for the numerator of a sum of two tangles.
\end{problem}

In the subsequent chapters we will show how to use theorem ~\ref{2:3.32} 
for computing the Jones-Conway polynomial of certain classes of links;
namely, links with two bridges, pretzel links and Montesinos links.

J.~Birman \cite{Bi-2} (and, independently, M. Lozano and H.Morton 
\cite{Lo-Mo}) has found examples of knots which are not isotopic
but have the same Jones-Conway polynomials. In \cite{L-M-1} it was noticed
that these links are not $\skein$ equivalent for they have different 
signatures (see \S 5 and chapter III).

In \cite{P-T-2} it was proved that these links are not algebraically
equivalent (i.e.~they can not be distinguished by any invariant yielded by 
a Conway algebra).
Fig.3.21. shows the simplest pair of Birman knots.

\begin{center}
\begin{tabular}{c} 
\includegraphics[trim=0mm 0mm 0mm 0mm, width=.5\linewidth]
{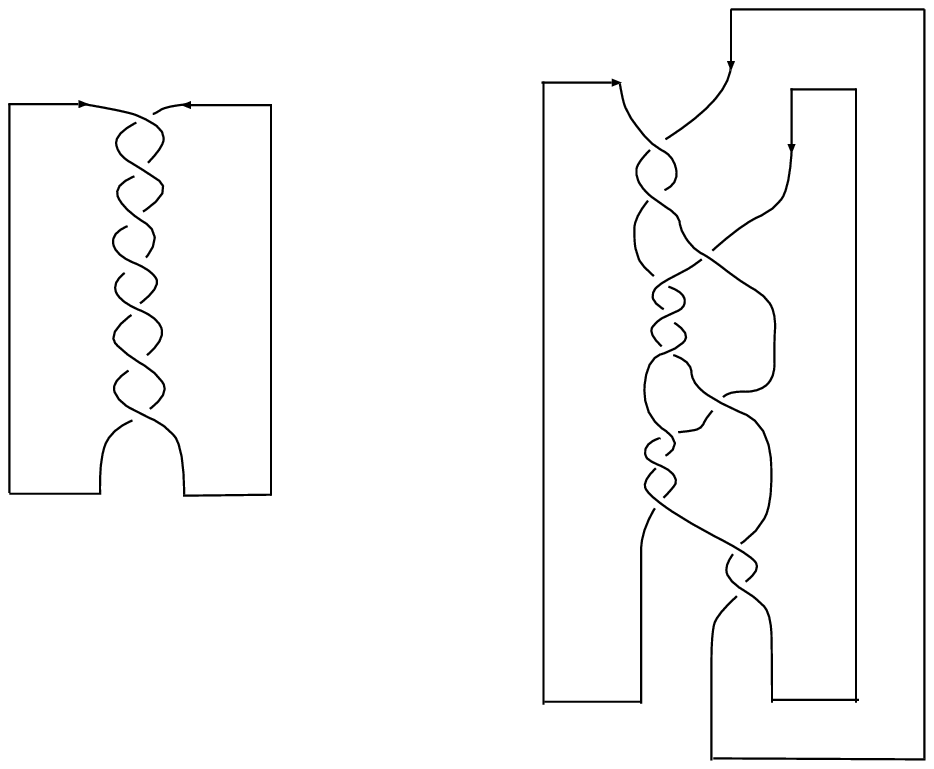}\\
\end{tabular}
\\
Fig. 3.21
\end{center}

\begin{problem}\label{Problem III.3.37}

\begin{enumerate}
\item [(i)] Do there exist links which have the same Jones-Conway polynomial
but are not algebraically equivalent?

\item [(ii)]
Do there exist links which are not algebraically equivalent, still they
have the same value of an invariant yielded by any finite Conway algebra?

\end{enumerate}
\end{problem}

Many algebraic properties of the Jones-Conway polynomial are known. 
Mostly these which relate its special substitutions with old 
invariants of links \cite{L-M-1,L-M-2,Mur-1,Mo-3} or \cite{F-W}. 
We will discuss these properties 
in respective chapters later on. Now we will present two quite elementary
properties of Jones-Conway polynomials.

\begin{lemma}\label{Lemma III:3.38}

\begin{enumerate}
\item [(i)] 
If $L$ is a link of odd number of components then all the monomials
of $P_L(x,y)$ polynomial are of even degree. If $L$ is a link of an even
number of components then these monomials are of an odd degree.

\item[(ii)]
 For every link $L$, $P_L(x,y)-1$ is divisible by
$x+y-1$, in particular, the Jones-Conway polynomial of a link cannot be
identically zero. 

\item [(iii)]
$P_L(x,y)+(-1)^{\mbox{com }(L)}$ is divisible by $x+y+1$;
$\mbox{com }(L)$ here means the number of components of link $L$.

\item [(iv)] $P_L(x,y) - (x+y)^{\mbox{com }(L)-1}$ is divisible by 
$(x+y)^2-1$.
\end{enumerate}
\end{lemma}
\begin{proof}
It is easy to check that the conditions (i)-(iv) are true for trivial
links. Then, it is enough to establish that if they are true
for $L_-$  $L_0$ (resp. $L_+$ and $L_0$) then they are true for $L_+$ 
(resp. $L_-$) as well.
\end{proof}
In fact one can improve a little Lemma \ref{Lemma III:3.38} using the same inductive proof.
For this one should consider the skein relation in slightly more general form:
$$xP_{L_+} + y P_{L_+} = zP_{L_0}$$ and do not assume that $z$ is invertible. Then the 
divisibility of $P_L - P_{T_{com(L)}}$ for knots can be formulated in more general manner 
than for any link (we did observe this already in the case of the Jones polynomial 
(see Chapter I). The value of the invariant for a trivial link of $n$ components is 
$P_{T_n}=(\frac{x+y}{z})^{n-1}$, thus if $z$ is not necessarily invertible we should work with invariants 
in the ring $A=\Z[x^{\pm 1},y^{\pm 1},z, \frac{x+y}{z}]$ or more formally\\
 $A=\Z[x^{\pm 1},y^{\pm 1},z,d]/(zd-(x+y))$.
The idea of working with this ring of invariants is used and developed when we work with periodic link, Chapter VIII.
With the above notation we have: 
\begin{proposition}\label{Proposition III:3.39}\ 
\begin{enumerate}
\item[(i)] For a link $L$ of $n$ components $P_L(x,y,z)-P_{T_n}(x,y,z)$ is divisible by $(x+y)\frac{x+y}{z}- z$.
\item[(ii)] For the Jones polynomial, that is $x=t^{-1},y=-t$, and $z=\sqrt{t}- \frac{1}{\sqrt{t}}$, we 
get: $V_L(t)-V_{T_n}(t)$ is divisible by $t^3-1$.
\item[(iii)] For a knot $K$, $P_K(x,y,z)-1$ is divisible by $(x+y)^2-z^2$. 
\item[(iv)] For the Jones polynomial of a knot $K$, $V_K(t)-1$ is divisible by $(t-1)(t^3-1)$.
\end{enumerate}
\end{proposition}
\begin{exercise}\
\begin{enumerate}
\item[(i)] Show that for a positive Hopf link, $H_+$ we have:
$$\frac{P_{H_+}-P_{T_2}}{(x+y)\frac{x+y}{z}- z}=\frac{1}{x}.$$
\item[(ii)] Show that for a negative Hopf link, $H_-$ we have:
$$\frac{P_{H_-}-P_{T_2}}{(x+y)\frac{x+y}{z}- z}=\frac{1}{y}.$$
\item[(iii)] Show that for a positive (right handed) trefoil knot, $\bar 3_1$ we have:
$$\frac{P_{\bar 3_1}-1}{((x+y)^2-z^2)}=-\frac{1}{x^2}.$$
\item[(iv)]  Show that for a negative (left handed) trefoil knot, $3_1$ we have:
$$\frac{P_{3_1}-1}{((x+y)^2-z^2)}=-\frac{1}{y^2}.$$
\item[(v)]  Show that for a figure eight knot, $4_1$ we have:
$$\frac{P_{4_1}-1}{((x+y)^2-z^2)}=-\frac{1}{xy}.$$
\item[(v)] Let $K_n$ be a twist knot of $n+2$ crossings (e.g. $K_2=4_1$ and $K_{-2}=\bar 3_1$).
 Show that for $n=2k$ we have:
$$\frac{P_{K_{2k}}-1}{((x+y)^2-z^2)}=-\frac{y^{k}-(-1)^{k}x^{k}}{xy^{k}(x+y)}.$$
Show that the formula holds also for negative $k$ and 
find the formula in the case of twist knots with $n$ odd.
\end{enumerate}
\end{exercise}

In the definition of the  Jones-Conway polynomial one could try to replace
the equation 1.9 by
\begin{formulla}\label{Formula III:3.41}
$$xw_1+yw_2 = zw_0 -v.$$
\end{formulla}
or, in the case $z$ is invertible, by
 
\begin{formulla}\label{e:3.39}
$$xw_1+yw_2 = w_0 -z.$$
\end{formulla}

Indeed, it leads to the invariant of links in 
$\Z\left[ x^{\pm 1},y^{\pm 1},z\right]$ \cite{P-T-1}, however this 
polynomial does not distinguish links better than the Jones-Conway 
polynomial (it was noticed in the Spring of 1995 
by the referee of \cite{P-T-1}, and later, independently
by O.Ya.Viro \cite{Vi}).\footnote{This is the case as long as we assume that a coefficient of $w_0$ 
is invertible.} Namely:

\begin{exercise}\label{Exercise III:3.43}
\begin{enumerate}
\item [(i)] 
Show that the following algebra is a Conway algebra.
${\cal A} = \{ A,a_1,a_2,\ldots,|,\star \}$ where $A = \Z\left[ x^{\mp
1},y^{\pm 1},z\right], a_1 =
1,a_2=x+y+z,\ldots a_i=(x+y)^{i-1}+z(x+y)^{i-2}+\cdots+z(x+y)+z,\cdots$.
We define the $|$ operation and $\star$ as follows:
$w_2|w_0 = w_1$ and $w_1\star w_0=w_2$ where
$$xw_1 + yw_2 = w_0-z; w_1, w_2,w_0\in A.$$

\item [(ii)] 
Show that the invariant of links $w_L(x,y,z)$, defined by a Conway algebra
of (i) satisfies  
$$w_L(x,y,z) = w_L(x,y,0) + z(\frac{w_L(x,y,0)-1}{x+y-1})$$
and that  
$$w_L(x,y,0) = P_L(x,y).$$
\end{enumerate}

Hint for (ii). Notice that $$a_i = (x+y)^{i-1} + 
z(\frac{(x+y)^{i-1}-1}{x+y-1}).$$
\end{exercise}

\begin{definition}\label{Definition III:3.44}
Every invariant of links can be used to build a better one. This will be called
a weighted simplex of the invariant. Namely, if $w$ is an invariant and
$L$ is a link of $n$ components $\row{L}{n}$, then we consider an 
$(n-1)$-dimensional simplex $\bigtriangleup^{n-1} = (\row{q}{n})$. 
To each face $(q_{i_1},\ldots,q_{i_k})$ of a simplex $\bigtriangleup^{n-1}$ 
we assign value $w_{L'}$, where $L' = L_{i_1}\cup\cdots\cup L_{i_k}$.

We say that two weighted simplexes are equivalent if there exists a permutation 
of their vertices which preserves weights of faces.
 
Of course, the weighted simplex of an invariant of ambient isotopy classes 
of oriented links is also an invariant of ambient 
isotopy classes of oriented links.

\end{definition}

\begin{example}[\cite{Bi-2}.]
The links of Fig.3.22 have the same Jones-Conway polynomial (as well as the 
signature --- see \S 4  and Chapter IV); still they are easy distinguishable
by the weighted simplex of the global linking numbers. 
(Example~\ref{p:1.10}).

\begin{center}
\begin{tabular}{c} 
\includegraphics[trim=0mm 0mm 0mm 0mm, width=.5\linewidth]
{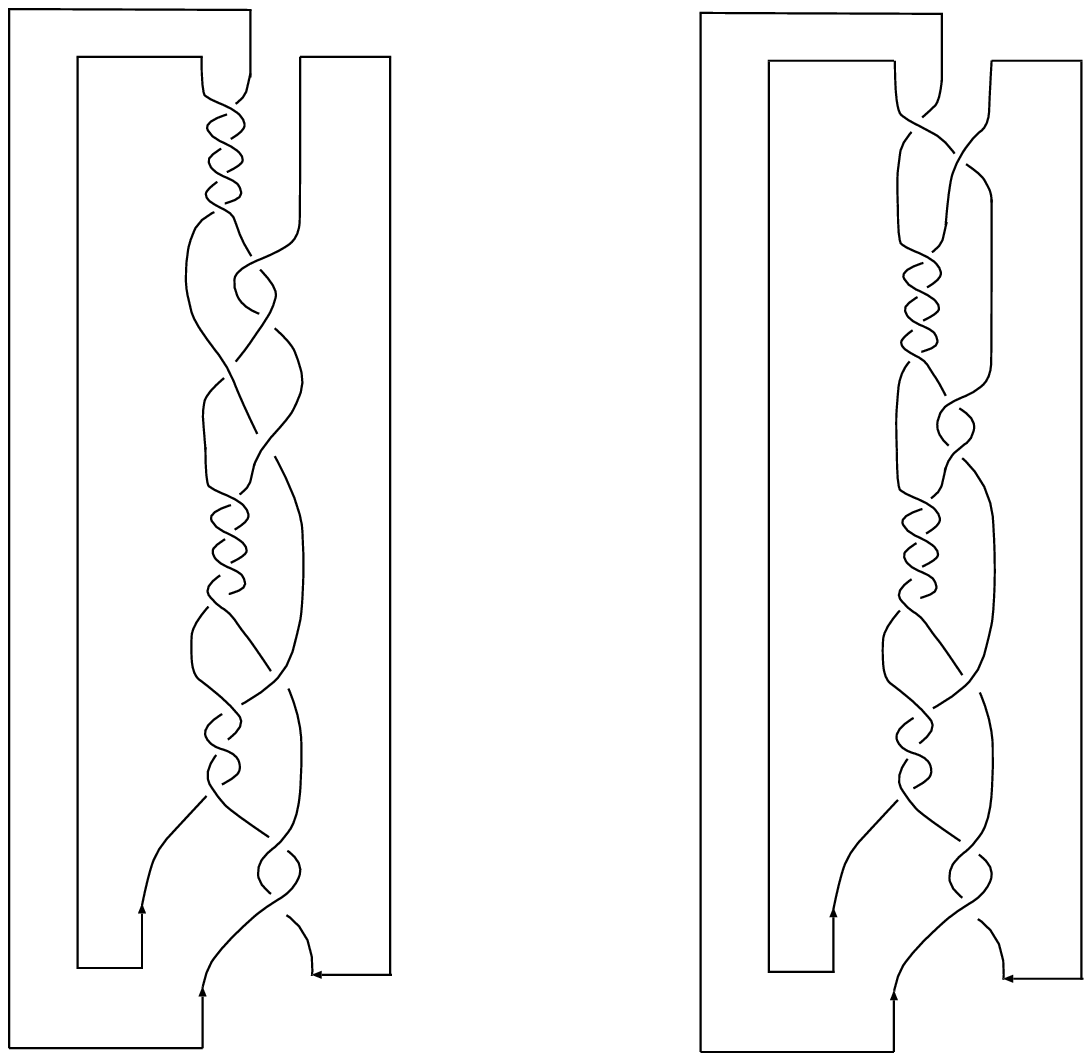}\\
\end{tabular}
\\
Fig. 3.22
\end{center}

\end{example}

\begin{exercise}\label{Exercise III:3.46}
Show that the links $L_1$ and $L_2$ of example~\ref{p:3.7}  
can be distinguished by the
weighted simplex of global linking numbers.
\end{exercise}


\section{Partial Conway Algebras.}

To get a Conway type link invariant with values in a set $A$ it is not 
necessary for the operations $|$ and $\star$ to be defined on the whole product
$A\times A$. Similarly there is no need for the relations $C3-C5$ to be
satisfied for all quadruples of $A\times A\times A\times A$. It is enough for
the operations to be defined and for the relations to be satisfied
merely in the case when 
the geometrical situation requires that. The above observation leads to
the definition of geometrically sufficient partial Conway algebras which define
the invariants of links of Conway type. These invariants can be more subtle 
than the ones obtained by  Conway algebras (e.g.~signature). The results 
we present here are based on  \cite{P-T-1,P-T-2,P-1}.

\begin{definition}\label{d:4.1}
A partial Conway algebra is a quadruple $(A,B_1,B_2,D)$ where $B_1$ and
$B_2$ are subsets of $A\times A$ and $D$ is a subset of
$A\times A\times A\times A$, together with 0-argument
operations $a_1,a_2,\ldots$ and two 2-argument operations $|$ and
$\star$ which are defined on $B_1$ and $B_2$, respectively. The operations
$|$ and $\star$ are assumed to satisfy equalities $C1-C7$ of Definition 1.1,
provided both sides of the respective equality are defined, and additionally 
$(a,b,c,d)\in D$ in the case of equalities $C3-C5$.
\end{definition}

\begin{definition}\label{d:4.2}
We say that a partial Conway algebra\\
 ${\cal A} =
(A,B_1,B_2,D,a_1,a_2,\ldots,|,\star)$ is geometrically sufficient if and
only if the following two conditions are satisfied.

\begin{enumerate}
\item[(i)] For every resolving tree of a link all the operations that are 
necessary to compute the root value are admissible, that is, all 
the intermediate values are in the sets where $|$ and $\star$ are defined.
\item[(ii)] Let $p_1, p_2$ be two different crossings of a diagram $L$. 
Let us consider the diagrams $L^{p_1\ p_2}_{\varepsilon_1\
\varepsilon_2},L^{p_1\ p_2}_{\varepsilon_1\ 0},L^{p_1\ p_2}_{0\
\varepsilon_2},L^{p_1\ p_2}_{0\ 0}$, where 
 $\varepsilon_i = -\sgn p_i$
and let us choose for them the resolving trees 
$T_{p_1,p_2},T_{p_1,0},T_{0,p_1},T_{0,0}$ respectively. 
Denote the root values of these trees by $w_{p_1,p_2},  w_{p_1,0},
w_{0,p_2}, w_{0,0}$ respectively. Now, we assume that in the above case
always \\
 $(w_{p_1,p_2},  w_{p_1,0},w_{0,p_2}, w_{0,0})\in D$.
\end{enumerate}
\end{definition}

Condition (ii) means that the resolving trees of $L$ of Fig.4.1 
give the same values at the roots of the trees.

\begin{center}
\begin{tabular}{c} 
\includegraphics[trim=0mm 0mm 0mm 0mm, width=.65\linewidth]
{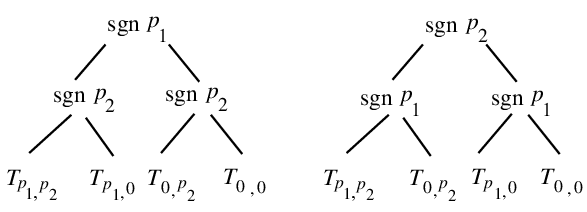}\\
\end{tabular}
\\
Fig. 4.1
\end{center}

The proof of Theorem~\ref{2:1.2}  can be used, without changes, for the
case of a geometrically sufficient partial Conway algebra.

\begin{theorem}\label{t:4.3}
Let ${\cal A}$ be a geometrically sufficient partial Conway algebra.
Then there exists a unique invariant $w$ which associates an element
from $A$ to any skein equivalence class of links and the invariant $w$
satisfies the following conditions:
\begin{enumerate}
\item $w_{T_n} = a_n$
\item $w_{L_+} = w_{L_-}|w_{L_0}$
\item $w_{L_-} = w_{L_+}\star w_{L_0}$.
\end{enumerate}
\end{theorem}

Similarly as in the case of Conway algebras, the conditions $C1-C7$ 
for partial Conway algebras are not independent.
Namely, we have the following result, the proof of which is left to the reader (but see the comment 
after the lemma).

\begin{lemma}\label{l:4.4}
Let $(A,B_1,B_2,a_1,a_2,\ldots,|,\star)$ 
be a partial algebra such 
that:
\begin{enumerate}
\item The property  (i) in Definition~\ref{d:4.2} is satisfied.

\item The property (ii) in Definition~\ref{d:4.2} is satisfied for each pair 
of crossings of positive sign, i.e.~the resolving trees of $L$ from Fig.~4.1
give the same values at their roots if $\sgn p_1=\sgn p_2 = +$.

\item The conditions $C1$, $C6$ and $C7$ are satisfied if both sides
of the equations are defined.
\end{enumerate}

If $D$ is a subset of $A\times A\times A\times A$ for which
the condition $C3$ is satisfied then ${\cal A} = 
(A,B_1,B_2,D,a_1,a_2,\ldots,|,\star)$ is a geometrically
sufficient partial Conway algebra.
\end{lemma}
In our proof, we follows that of Theorem III:1.3 but we should be sure that we 
are moving inside the domain of our partial algebra.
Thus assume (3), (6) and (7) and we will prove (4). Consider a link diagram $L$ 
and its two crossing of different signs, $p_1$ and $p_2$. We can assume that $L= L_{+\ \ -}^{p_1 p_2}$ 
Then when resolving first $p_1$ and then $p_2$ we have allowed expression:
$$(w_{ L_{-\ \ +}^{p_1 p_2}}* w_{L_{-\ \ 0}^{p_1 p_2}})|( w_{L_{0\ \ +}^{p_1 p_2}}*w_{ L_{0\ \ 0}^{p_1 p_2}})$$
If we resolve first $p_2$ and then $p_1$ we get another allowed expression 
$$(w_{ L_{-\ \ +}^{p_1 p_2}}| w_{L_{0\ \ +}^{p_1 p_2}})*( w_{L_{-\ \ 0}^{p_1 p_2}}|w_{ L_{0\ \ 0}^{p_1 p_2}})$$
and our goal is to show that they are equal.
Now $(w_{ L_{-\ \ +}^{p_1 p_2}}* w_{L_{-\ \ 0}^{p_1 p_2}})$ is from the definition equal to
 $w_{ L_{-\ \ -}^{p_1 p_2}}$ 
and $w_{L_{-\ \ 0}^{p_1 p_2}}|w_{ L_{0\ \ 0}^{p_1 p_2}})$ is equal to $w_{ L_{-\ \ +}^{p_1 p_2}}$

Now we will construct two examples of geometrically sufficient partial Conway
algebras 
and we will discuss the invariant of links which are defined by these 
two algebras. Let us begin with an example which leads to a direct 
generalization of the  Jones-Conway polynomial.
Instead of equations \ref{III:1.9} or \ref{e:3.39} we use other equations
depending on the number of components of $L_+$, $L_-$ and $L_0$.

\begin{example}\label{2:4.5}
The following partial algebra ${\cal A}$ is a geometrically sufficient
partial Conway algebra:

\begin{eqnarray*}
A&=&N\times
\Z\left[
{y_1^{\pm 1}, x'^{\pm 1}_2, z'^{\pm 1}_2,x_1^{\pm 1}}, z_1, x_2^{\pm 1},  z_2, 
x_3^{\pm 1}, z_3,\ldots\right]\\
B_1 = B_2&=&\left\{ ((n_1,w_1),(n_2,w_2))\in A\times A:
|n_1-n_2|=1\},D=A\times A\times A\times A\right\}\\
a_1&=&(1,1)\\
a_2&=&(2,x_1+y_1+z_1)\\
&\vdots&\\
a_n&=&(n,\prod^{n-1}_{i=1}(x_i+y_i) +
z_1\cdot\prod^{n-1}_{i=2}(x_i+y_i)+\cdots+z_{n-2}(x_{n-1}+y_{n-1})+z_{n-1})\\
&\vdots&\\
\end{eqnarray*}
where $y_i = x_i\cdot\frac{y_1}{x_1}$. To define operations $|$ and
$\star$ we consider the following system of equations:
\begin{eqnarray*}
(1)&x_1w_1 + y_1w_2 = w_0 - z_1&\\
(2)&x_2w_1 + y_2w_2 = w_0 - z_2&\\
(2')&x_2'w_1 + y_2'w_2 = w_0 - z_2'&\\
(3)&x_3w_1 + y_3w_2 = w_0 - z_3&\\
(3')&x_3'w_1 + y_3'w_2 = w_0 - z_3'&\\
&\vdots&\\
(i)&x_iw_1 + y_iw_2 = w_0 - z_i&\\
(i')&x_i'w_1 + y_i'w_2 = w_0 - z_i'&\\
\end{eqnarray*}
where $y_i' =\frac{x_i'y_1}{x_1}$, $x_i' = \frac{x_2'x_1}{x_{i-1}}$ and
$z_i'$ is defined inductively so that it satisfies the equality
$$\frac{z_{i+1}'-z_{i-1}}{x_1x_2'} = (1+\frac{y_1}{x_1})(\frac{z_i'}{x_i'}
- \frac{z_i}{x_i}).$$

Now we define $(n,w) = (n_1,w_1)|(n_2,w_2)$ and, respectively,
$(n,w)=(n_1,w_1)\star (n_2,w_2)$ in the following way: we set $n:= n_1$ and
further
\begin{enumerate}
\item if $n_1 = n_2-1$ then we use equation $(n)$ ($n_1=n$)
to determine
$w$, namely $x_nw + y_nw_1 = w_2 -z_n$ (respectively, $x_nw_1 +
y_nw = w_2 - z_n$), 

\item if $n_1 = n_2+1$ then we use equation $(n')$ to determine
$w$, namely $x'_nw + y_n'w_1 = w_2 -z_n'$ (respectively, $x_n'w_1 +
y_n'w = w_2 - z_n'$).
\end{enumerate}
\end{example}

We shall prove that such ${\cal A}$ is a geometrically sufficient partial
Conway algebra.

It is easy to check that first coordinates of elements from $A$
satisfy relations $C1-C7$ (they define the number of components in the link
c.f.~Example III.1.5. Also, it is not hard to check that $A$ satisfies
relations $C1$, $C2$, $C6$ and $C7$. Therefore we concentrate on relations
$C3-C5$.

It is convenient to use the following notation: if $w\in A$
then  $w = (|w|,F_w)$ and for 
$$w_1|w_2 = (|w_1|, F_{w_1})|(|w_2|,F_{w_2}) = (|w|,F_w) = w$$
we write 
$$
F_w =
\left\{
\begin{array}{lll}
F_{w_1}|_n F_{w_2}&\mbox{if}&n=|w_1| = |w_2|-1\\
F_{w_1}|_{n'} F_{w_2}&\mbox{if}&n=|w_1| = |w_2|+1.\\
\end{array}
\right.
$$

We also use a similar notation for the operation $\star$. 

In order to verify  $C3-C5$ we have to consider three main cases:
\begin{enumerate}

\item $|a| =  |c|-1 = |b|+1 = n;\ a,b,c,d\in A$.

Both sides of relations $C1-C5$ are defined if and only if 
$|d| = n$. The relation $C3$ is then reduced to the following equation:
$$ (F_a|_{n'}F_b)|_n (F_c|_{n'}F_d) =(F_a|_{n}F_b)|_{n'}
(F_c|_{n}F_d).$$
From this we get:

\begin{eqnarray*}
&\frac{1}{x_nx_{n+1}'}\cdot F_d - \frac{y_{n+1}'}{x_nx_{n+1}'}\cdot
F_c-\frac{y_{n}}{x_nx_{n}'}\cdot F_b+\frac{y_ny_{n}'}{x_nx_{n}'}\cdot
F_a&-\\
&\frac{z_{n+1}'}{x_nx_{n+1}'} -
\frac{z_n}{x_n}+\frac{y_{n}z_n'}{x_nx_{n}'}&=\\
&\frac{1}{x_n'x_{n-1}}\cdot F_d - \frac{y_{n-1}'}{x_n'x_{n-1}}\cdot
F_b-\frac{y_{n}'}{x_nx_{n}'}\cdot F_c+\frac{y_ny_{n}'}{x_nx_{n}'}\cdot
F_a&-\\
&\frac{z_{n-1}}{x_n'x_{n-1}} -
\frac{z_n'}{x_n'}+\frac{y_{n}'z_n}{x_nx_{n}'}&\\
\end{eqnarray*}
and subsequently
\begin{enumerate}
\item $x_{n-1}x_{n}' = x_nx_{n+1}'$

\item $\frac{y_{n+1}'}{x_{n+1}'} = \frac{y_n'}{x_n'}$

\item $\frac{y_n}{x_n} = \frac{y_{n-1}}{x_{n-1}}$

\item $\frac{z_{n+1}'}{x_nx_{n+1}'} +
\frac{z_n}{x_n}-\frac{y_{n}z_n'}{x_nx_{n}'} = \frac{z_{n-1}}{x_n'x_{n-1}} +
\frac{z_n'}{x_n'}-\frac{y_{n}'z_n}{x_nx_{n}'}$.
\end{enumerate}
Relations $C4$ and $C5$ give the same conditions (a)--(d).

\begin{enumerate}
\item $|a| =  |b|-1 = |c|-1 = n$.

\item $|d| = n$.

The relation $C3$ can be reduced to the following equation:
$$ (F_a|_{n}F_b)|_n (F_c|_{(n+1)'}F_d) =(F_a|_{n}F_c)|_{n}
(F_b|_{(n+1)'}F_d).$$

After simple calculations we find out that the above equality is equivalent
to:
\end{enumerate}
\setcounter{enumii}{4}
\item $\frac{y_n}{x_n} = \frac{y_{n+1}'}{x_{n+1}'}$

Relations $C4$ and $C5$ can be reduced to the condition (e) as well.

\begin{enumerate}\setcounter{enumiii}{1}
\item $|d| = n+2$.

In this case relations $C3-C5$ are reduced to the condition (c).
\end{enumerate}


\item  $|a| =  |b|+1 = |c|+1 = n$.
\begin{enumerate}
\item $|d| = n-2$.
\item $|d| = n$.
\end{enumerate}

After simple calculations we find out that in the cases 3(i) and 3(ii) 
the relations $C3-C5$ follow from the conditions (c) and (e). 
\end{enumerate}

The conditions (a)--(e) are equivalent to conditions satisfied by 
$x_i',y_i,y_i'$ and $z_i'$ in Example~\ref{2:4.5}. Therefore we have 
proved that ${\cal A}$ is a geometrically sufficient partial
Conway algebra.

The partial algebra discussed above in Example~\ref{2:4.5} defines
an invariant of links, the second component of which is a polynomial
of infinite number of variables. It is natural to ask how is this polynomial
related to other known invariants of links and whether it generalizes
the Jones-Conway polynomial.

\begin{problem}
\ \\
\begin{enumerate}
\item Do there exist two oriented links which have the same 
Jones-Conway polynomial but they can be distinguished by 
a polynomial of infinitely many variables?
\footnote{Adam Sikora (then a student of P.~Traczyk) proved that the answer
for the question 4.6.1 is negative,  \cite{Si-1}.
If the coefficient of $w_0$ in equations (1)--(i') 
was not equal to 1 but was non-invertible then the polynomial 
of infinite number of variables would distinguish some of Birman's links
which can not be distinguished by the  Jones-Conway polynomial; 
\cite{P-6}}

\item Do there exists two algebraically equivalent links (i.e. links which have always the same invariants 
coming from Conway algebras) which can be
distinguished by a polynomial of infinite variables?
\end{enumerate}
\end{problem}


It is quite possible that our partial algebra of polynomials of infinite number of variables
can be extended to a Conway
algebra (in that case Traczyk result would follows from Toyoda Theorem \cite{Toy-1,Toy-2,Toy-3,Toy-4})
At any rate, Birman's examples \cite{Bi-2}
cannot be distinguished one from the other
by a polynomial of infinite variables.
In particular we have:
\begin{exercise}
Prove that the links from Fig.~3.21 have the same polynomials of infinite
variables. The same for links pictured at Fig.~3.22.
\end{exercise}

The next example of a geometrically sufficient partial Conway algebra 
defines important classical invariants of links: classical signature, Tristram-Levine 
signature and its 
generalization --- supersignature (see \cite{P-T-3}; we were unable to prove there an 
existence of a supersignature, which would solve Milnor's unknotting conjecture. The dream 
of ``supersignature" was realized with invention of Khovanov homology (see Chapter X) and 
Rasmussen s-signature \cite{Kh-1,Ras}).
An advantage of signature (especially in view of the case of polynomial of 
infinite variables which we have just discussed) is the existence of examples
of algebraically equivalent links which can be distinguished by the signature
(e.g.~knot pictured at Fig.~3.21).
This implies that there exist algebraically equivalent links which 
are not skein equivalent. It is somewhat uncomfortable, though,
that until now we know no purely combinatorial proof that signature and
supersignature are invariants of ambient isotopy classes of links
(compare Chapter IV).

\begin{definition}\label{d:4.8}(Supersignature Algebra)\\
For a pair of real number $u$, $v$, such that $u\cdot v>0$,
we define a partial Conway algebra ${\cal A}_{u,v}$ which we 
call the supersignature algebra. We define it  as follows:
$A = (R\cup iR)\times (\Z\cup\infty)$, $B_1 = B_2 = \{
(r_1,z_1),(r_2,z_2)\in A\times A\ | \  \mbox{ if } r_1\in R  \mbox{ then 
$r_2\in iR$, if $ r_1\in iR$ then $r_2\in R$, if $z_1,z_2\neq\infty $ 
then $ |z_1-z_2| = 1$, 
and $r_i = 0$} \\ 
\mbox{ if and only if } z_i = \infty \}$.

The operations $|$ and $\star$ are defined as follows:

The first coordinates of elements
in $A=(R\cup iR)\times (\Z\cup\infty)$ are related by the equality
$$-uw_1 +vw_2 = iw_0, \mbox{ where } w_1,w_2,w_0 \in R\cup iR$$
which is similar to the case of the Jones-Conway polynomial (it is enough to substitute
$u = -xi$ and $v=yi$ to get the equation~\ref{III:1.9}). 
In particular, the first coordinate of the result of an operation
depends only on the first coordinates of arguments, so we will write simply
 $w_1 = w_2| w_0$ and $w_2 = w_1\star w_0$. 

In order to describe the second coordinate of the result of operations $|$ and
$\star$ we write $(r,z) = (r_1,z_1)|(r_0, z_0)$ or
$(r_1,z_1)\star (r_0, z_0)$ where $z$ is determined by the following conditions:
\begin{enumerate}
\item $i^z = \frac{r}{|r|}$ if $r\neq 0$,
\item
$|z-z_0| = 1$
if $r\neq 0$ and $r_0 \neq 0$,
\item $z=z_1$ if $r_0 = 0$,
\item $z = \infty$ if $r=0$.
\end{enumerate}
The $0$-argument operations (constants or unary operations) $a_i$ are defined as follows:
$$a_1 = (1,0),\ldots, a_k= \Biggr((\frac{v-u}{i})^{k-1},
\left\{
\begin{array}{lll}
-(k-1)&\mbox{if }& u<v\\
\infty&\mbox{if }&       u=v\\
k-1&\mbox{if }&u>v
\end{array}
\right. \Biggr) ,\ldots .$$

And finally, $D$ is the subset of $A\times A\times A\times A$ consisting of elements
for which the relations $C3--C5$ are satisfied.
\end{definition}

\begin{problem}[\cite{P-T-2}]\label{III:4.9}
For which values of $(u,v)$
is ${\cal A}_{u,v}$ a geometrically sufficient partial Conway algebra?
\end{problem}

For the pair $u,v$ for which the answer is positive, the algebra 
${\cal A}_{u,v}$ defines an invariant of links, the second component of which
we call the supersignature and we denote it by $\delta_{u,v}(L).$

In Chapter IV we  show that the answer for Problem~\ref{III:4.9} 
is affirmative for
$u =v\in  (-\infty,-\frac{1}{2}]\cup[\frac{1}{2},\infty)$ and then
for such $u$ and $v$ the supersignature coincides with 
Tristram-Levin signature \cite{Tr,Le,Gor} (unless it is equal $\infty$). 
In particular, for $u=v=\frac{1}{2}$ we get the classical signature.\footnote{T.Przytycka 
found (in the Spring of 1985) the values of $u,v$, for which the supersignature
is not well defined. In her examples $u\neq v,v^{-1}$. 
P.~Traczyk and M.~Wi\'sniewska \cite{Wis} have found a neighborhood of 
the origin
on the plane consisting of pairs $(u,v)$for which the supersignature does not
exist (Problem~\ref{III:4.9} has negative answer). There is a possibility that there exists a supersignature related
to the Jones polynomial, i.e.~defined for pairs $(u,v)$ such that 
$v=\frac{-t}{{\sqrt  t}-{\sqrt t}^{-1}}=t^2u$ where $t$ is a negative real number.} 
The proof of this fact goes beyond simple purely combinatorial
methods and therefore we postpone it until the next chapter.
Now, let us look what the obstructions for a direct solution
of Problem~\ref{III:4.9} are.

\begin{enumerate}
\item Suppose that the condition (i) in Definition~\ref{d:4.2} 
is satisfied. This implies that, given any resolving tree of a link,
all operations needed to compute the value of the invariant in the root
are admissible. This is because the first coordinate of elements of the supersignature
algebra (which is known to be an invariant of links as it is a
variant of the  Jones-Conway polynomial) assumes real values for links with
an odd number of components and purely imaginary values for links with 
an even number of components.

\item  We would like to prove the condition (ii) of Definition~\ref{d:4.2}.
The condition asserts that different resolving trees of a given link yield
the same value in their roots (equivalently, conditions $C3-C5$ are satisfied
when needed).
The first attempt of the proof is to check whether the condition $C3$ 
is true whenever both sides of the equality are well defined
(this is enough because of Lemma~\ref{l:4.4}) --- the idea is similar
to Example~\ref{2:4.5}. This time, however, this is not the case, as we see
from the following example.
\end{enumerate}

\begin{example}\label{III:4.10}
Let us consider the condition $C3$ for a supersignature algebra ${\cal A}_{u,v}$
$$((r_a,z_a)|(r_b,z_b))|((r_c,z_c)|(r_d,z_d))=((r_a,z_a)|(r_c,z_c))|((r_b,z_b)|(r_d,z_d))$$
Because of the definition of the operation $|$ we get:
\begin{enumerate}
\item $r_a|r_b = \frac{1}{u}(ir_b+vr_a)$
\item $r_a|r_c = \frac{1}{u}(ir_c+vr_a)$
\item $r_c|r_d = \frac{1}{u}(ir_d+vr_c)$
\item $r_b|r_d = \frac{1}{u}(ir_d+vr_b)$
\item
$(r_a|r_b)|(r_c|r_d)=(r_a|r_c)|(r_b|r_d)=\frac{1}{u^2}_d+ivr_c+ivr_b+v^2r_a)$

Now suppose that
\item $u,v>0;r_a,r_d\in R;r_b,r_c\in iR;z_a = -2, z_b=z_c = -1, z_d =
0$.

Because of these properties and in view of the condition (1) of Definition
~\ref{d:4.8}) we get:

\item $r_a<0, r_d>0,ir_b>0, ir_c>0$

Further, suppose that 

\item $ir_b+vr_a>0$, $ir_c+vr_a<0$, $m -r_d+ivr_c<0$, $-r_d+ivr_b>0$

and

\item $-r_d+ivr_c+ivr_b+v^2r_a < 0$.
\end{enumerate}

The conditions 6--10 may be satisfied (even for $u=v=\frac{1}{2}$),
yet the value of the second coordinate computed in left-hand-side of
the condition $C3$ is equal 2, while on the right-hand-side we get -2.
\end{example}

\begin{exercise}\label{2:4.11}
If, for some values of $u,v$, the answer for Problem ~\ref{III:4.9} 
is affirmative then, in view of the Example~\ref{III:4.10}, 
we can use the algebra to find bounds on Jones-Conway polynomials of links.
Find these bounds.
\end{exercise}

\begin{lemma}\label{III:4.12}
The supersignature $\delta_{u,v}$ (for these  $u,v$ for which it exists)
satisfies the following conditions:
\begin{enumerate}
\item[(a)] $\delta_{u,v}(L) = -\delta_{v,u}(\overline{L})$,
\item [(b)] $\delta_{u,v}(L_1\kwad L_2) = \delta_{u,v}(L_1)+\delta_{u,v}(L_2)$
\item[(c)] $\delta_{u,v}(L_1\sqcup L_2) =
\delta_{u,v}(L_1)+\delta_{u,v}(L_2)+ \varepsilon(u,v)$ where

$$\varepsilon(u,v) = \left\{
\begin{array}{lll}
1&\mbox{if}u>v\\
\infty&\mbox{if}&u=v\\
-1&\mbox{if}&u<v\\
\end{array}
\right.
$$

\item[(d)] $\delta_{u,v}(L_+)\leq \delta_{u,v}(L_-)$ if
$\delta_{u,v}(L_+)\neq\infty$ and $u,v>0$.
\end{enumerate}
\end{lemma}

\begin{proof}

In the proof of conditions (a), (b) and (d) we use a standard induction
with respect to the number of crossings for a choice of base points in the diagram.
Moreover in the proof of (a) we apply Lemma~\ref{III:3.21} which implies
equality
$$r_L(u,v) = r_{\overline{L}}(-v,-u) = 
\left\{
\begin{array}{l}
 r_{\overline{L}}(v,u)\mbox{ if $L$ has an odd number of components}\\
 -r_{\overline{L}}(v,u)\mbox{ if $L$ has an even number of components}
\end{array}
\right.
$$
and in the proof of (b) we apply Corollary~\ref{2:3.27} which implies 
$$r_{L_1\kwad L_2}(u,v) = r_{L_1}(u,v)\cdot_{L_2}(u,v).$$

The condition (c) follows from (b) once we note that $L_1\sqcup L_2$
can be obtained as a connected sum $(L_1\kwad T_2)\kwad L_2$ where $T_2$
is a trivial link with two components (Fig.~4.2) and
$\varepsilon(u,v) = \delta_{u,v}(T_2)$.

\end{proof}
\begin{center}
\begin{tabular}{c} 
\includegraphics[trim=0mm 0mm 0mm 0mm, width=.3\linewidth]
{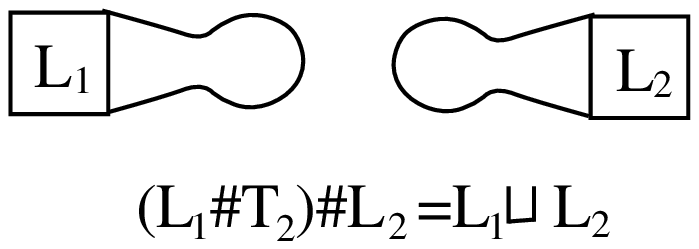}\\
\end{tabular}
\\
Fig.~4.2
\end{center}

\begin{exercise}\label{2:4.13}
Show that knots $8_8$ and $\overline{8}_8$ (Fig.~3.11) can be distinguished
by the supersignature $\delta_{u,v}$ for some values of $u$ and $v$; however 
if $u=v$ then the supersignature of both knots is equal to zero.

Sketch of the argument.

\begin{enumerate}

\item  $$r_{8_8}(u,v) =
-uv^{-1}+2+v^{-2}+u^{-1}v-2u^{-1}v^{-1}-u^{-2}v^2-2u^{-2}+u^{-2}v^{-2}+u^{-3}v+u^{-3}v^{-1}$$
Hence we can compute that:

\begin{enumerate}
\item $r_{8_8}(u,u)>0$
\item $r_{8_8}(u,v)<0$ if $u>>v\approx c$ or
$v>>u\approx c$
\item $r_{8_8}(u,v)>0$ if $u<<v\approx c$ or
$v<<u\approx c$,
\end{enumerate}
where $c$ is a constant number.

\item The removal (smoothing) of a crossing in some diagram of the knot 
$8_8$ yields a trivial knot of two components (c.f.Example ~ \ref{Example III:3.19}).
\end{enumerate}

In view of the definition of the supersignature, the above observations imply that for
$u,v>0$ we have:
$$\sigma_{u,v}(\overline{8}_8) = \left\{
\begin{array}{lll}
0&\mbox{if}&u<<v\\
0&\mbox{if}&u=v\\
-2&\mbox{if}&u>>v\\
\end{array}
\right.$$

Now applying Lemma~\ref{III:4.12}(a) we show easily the values of 
$u$ and $v$ for which $\sigma_{u,v}(8_8)\neq\sigma_{u,v}(\overline{8}_8)$; see Fig. 4.3.

\begin{center}
\begin{tabular}{c} 
\includegraphics[trim=0mm 0mm 0mm 0mm, width=.5\linewidth]
{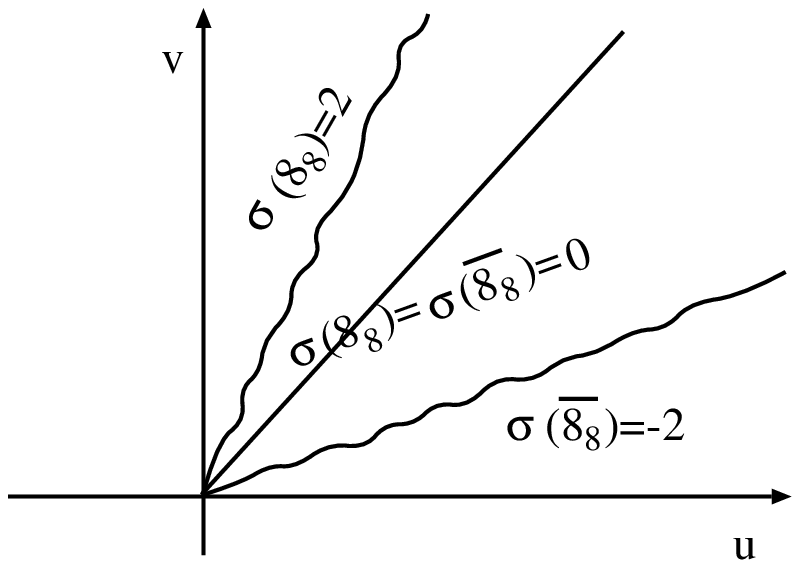}\\
\end{tabular}
\\
Fig.~4.3
\end{center}

\end{exercise}

\begin{remark}\label{2:4.14}
An important problem in knot theory concerns so-called unknotting (or Gordian)
number of a link which --- by definition --- is the minimal number of 
tunnel-to-bridges changes
which have to be done to modify a given link to the trivial
one (see Chapter IV). 
The supersignature may seem to be very useful to bound the unknotting 
number since its values for $L_+$ and $L_-$ differ by at most 2, unless
one of them is $\infty$. However, $\infty$ was assigned for our convenience
only, i.e.~to simplify the description. One may try to find another value 
of $\sigma_{u,v}(L)$ (different from $\infty$) for links with 
$r_L(u,v) = 0$. Useful information can be found in Lemma~\ref{Lemma III:3.38}(ii)
which reveals that the Jones-Conway polynomial, and thus also the polynomial 
$r_L (u,v)$, is not identically zero. Thus, even if for some $u_0$ and $v_0$ 
we have $r_L(u_0,v_0)=0$, then in some neighborhood of $(u_0, v_0)$
the polynomial $r_L(u_0,v_0)$ assumes non-zero values.
Subsequently, one may try to change the value of $\sigma_{u_0,v_0}(L)$
replacing $\infty$ with the mean value of $\sigma_{u,v}\neq 0$.
We leave this idea to the reader as a research problem.
\end{remark}

\begin{exercise}\label{2:4.15}
Prove that, if for a pair $(u,v)$ there exists the supersignature
$\sigma_{u,v}(L)$ and it is not equal $\infty$, then the minimal depth
of the resolving tree of  $L$ (i.e.~the distance from the root to the
farthest leaf) is not smaller than $$\frac{|
\sigma_{u,v}(L) |}{2} - \varepsilon (L),$$ where
$$\varepsilon = \left\{
\begin{array}{lll}
0&\mbox{if}&u=v\\
n(L)-1&\mbox{if}& u\neq v\\
\end{array}
\right.
;$$
and $n(L)$ denotes the number of components of $L$.
\end{exercise} 

\begin{example}\label{2:4.16}
The equivalence classes of the relation $\sim_c$ of oriented links
make geometrically sufficient partial Conway algebra.
Elements $a_i$ of this algebra are classes of trivial links with 
$i$ components, the operation $|$ (resp.~$\star$) is defined on a pair
of classes of links if they can be represented by diagrams of type
$L_-$ and $L_0$ (resp.~$L_+ $ and $L_0$) and the result is equal 
to the class of $L_+$ (resp.~$L_-$). The definition of $\sim_c$ equivalence
is chosen in such a way that axioms $C1-C7$ are satisfied when it is needed.
Let us note that this partial algebra is a universal geometrically sufficient
partial Conway algebra, that is there exits the unique homomorphism of this
algebra onto any geometrically sufficient partial Conway algebra.
\end{example}

As we have mentioned in  Remark \ref{2:3.23}, some Conway algebras admit an operation 
$\circ$ which allows to find the value of an invariant for $L_0$, provided
we know its value for $L_+$ and $L_-$. Geometrically
sufficient partial Conway algebras described in \ref{2:4.5} and \ref{d:4.8} 
allow for such an operation, it is an open question, however,  whether
the universal Conway algebra admits such an operation. More precisely,
it is an open question whether the equations $a|x=b$ and $a*x=b$  
can have at most one solution --- see Problem \ref{2:3.30}).

The involution $\tau$ from Lemma~\ref{2:3.32} is realized  
in the universal geometrically sufficient partial 
Conway algebra as the operation of replacing a given \skein-equivalence
class by the class of its mirror image.

\begin{exercise}\label{2:4.17}
\ \\
\begin{enumerate}
\item Suppose that $\tau$ is an involution of a geometrically sufficient partial 
Conway algebra.
What properties have to be satisfied $\tau$ in order to satisfy the condition
$$\tau(w_L) = w_{\overline{L}}?$$
\item Find such an involution $\tau$ for the algebra from Example~\ref{2:4.5}.

\item Prove that there exists no such involution for the algebra of the 
supersignature $\sigma_{u,v} (u\neq v)$ but if we modify the algebra 
so that new elements are quadruples 
$(r_L(u,v), r_L(v,u), \sigma_{u,v}(L), \sigma_{v,u}(L) )$, 
then the involution $\tau$ can be defined (c.f.~Lemma \ref{III:4.12}(a)).
\end{enumerate}
\end{exercise}

Also Lemma \ref{III:3.28}, which allows to find out the invariant of the
connected sum and disjoint union of links, can be partially extended
to the case of geometrically sufficient partial Conway algebras
(c.f.~Lemma \ref{III:4.12}(b) and (c)).



\section{Kauffman approach}

It is a natural question to ask whether the three
diagrams, $L_+$, $L_-$, and $L_0$, which have been used  
to build Conway type invariants can be replaced by other diagrams.
In fact, at the turn of December and January of 1984/85,
Krzysztof Nowi\'nski suggested considering another diagram
apart from $L_+$, $L_-$ and $L_0$, namely, the diagram obtained
by smoothening $L_+$ without preserving the orientation of $L_+$
(Fig.~5.1)
--- at that time, however, we did not make any effort to 
exploit this idea to get new invariants of links; likely we were discouraged by 
a lack of a natural orientation on $L_{\infty}$.

\begin{center}
\begin{tabular}{c} 
\includegraphics[trim=0mm 0mm 0mm 0mm, width=.3\linewidth]
{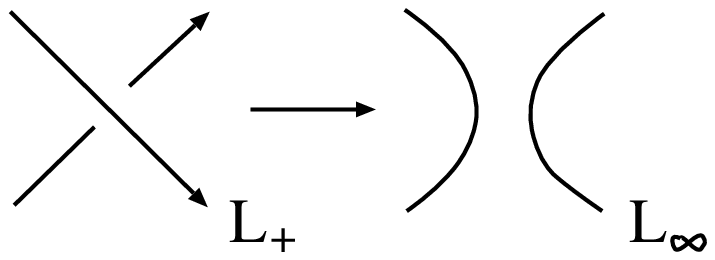}\\
\end{tabular}
\\
Fig.~5.1 
\end{center} 

In the early Spring of 1985 
R.Brandt, W.B.R.Lickorish i K.C.Millett \cite{B-L-M}
and, independently, C.F.Ho \cite{Ho} 
proved that four diagrams of unoriented 
links pictured in Fig.~5.2 
(the $+$ sign in $L_+$ does not denote the sign of the crossing,
even if the sign is defined) can be used to construct invariants of unoriented
links.
\ \\

\begin{center}
\begin{tabular}{c} 
\includegraphics[trim=0mm 0mm 0mm 0mm, width=.85\linewidth]
{L+L-L0Linf.eps}\\
\end{tabular}
\\
Fig.~5.2 
\end{center} 
 
\begin{theorem}\label{2:5.1}
There exists a uniquely determined invariant  $Q$ which to any
(ambient) isotopy class of invariant links associates an element of
$\Z[x^{\pm 1}]$. The invariant $Q$ satisfies the following conditions:
\begin{enumerate}
\item
[(1)] if $T_1$ is the trivial knot then $Q_{T_1}(x)=1$ 
\item
[(2)] $Q_{L_+}(x) + Q_{L_-}(x)= x(Q_{L_0}(x)+ Q_{L_{\infty}}(x))$, where unoriented 
 diagrams of links $L_+, L_-,L_0$ and $L_{\infty}$  are identical outside of the parts
pictured in Fig.~5.2.
\end{enumerate}
\end{theorem}
The proof of Theorem 5.1 is similar to that of Theorem \ref{2:1.2} (c.f.~\cite{B-L-M}). 
We will discuss it later in a more general context.

The polynomial $Q_L(x)$ has a number of features similar to these of
the Homflypt (Jones-Conway) polynomial $P_L(x,y)$. The proof of these properties is left for the 
reader as an exercise
\begin{exercise}\label{2:5.2}
Prove that:
\begin{enumerate}
\item
[(a)]
$Q_{T_n}(x)= (\frac{2-x}{x})^{n-1}$. where $T_n$ is the trivial link of $n$ components,
[(b)]
$Q_{L_1\ \#  L_2}(x) = Q_{L_1}(x)Q_{L_2}(x)$ \ ,
\item
[(c)]
$Q_{L_1\sqcup L_2}=\mu Q_{L_1}(x)Q_{L_2}(x)$ \ , 
where $\mu=Q_{T_2}(x)=2x^{-1}-1$ 
trivial link.
\item
[(d)]
$Q_L(x)=Q_{\bar L}(x)$, where $\bar L$ is the mirror image of the link $L$.
\item
[(e)]
$Q_L(x)=Q_{m(L)}(x)$, where $m(L)$ is the mutant of the link $L$.
\item
[(f)] $Q_L(x)(-2)= (-2)^{com(L)-1}$.
\end{enumerate}
\end{exercise}

The polynomial $Q_L(x)$ can sometime distinguish links which are
 ${\sim}_c$ equivalent
\begin{exercise}\label{2:5.3}
\begin{enumerate}
\item
[(a)] Show that $Q_L(x)$ distinguishes knots
$8_8, \bar {10}_{129}$, and $13_{6714}$ (c.f.~Example 3.19).
\item
[(b)] Conclude that the above knots can not be obtained by mutation 
of one another.
\end{enumerate}
\end{exercise}

Now we will discuss Kauffman approach \cite{K-4,K-5} which, in particular, 
allows to generalize the polynomial 
$Q_L(x)$ to an invariant distinguishing mirror images.
This approach is based on the Kauffman's idea that
instead of considering diagrams modulo equivalence (that is diagrams up
to all three Reidemeister moves) one can consider diagrams up to the 
equivalence relation which is based on second and third Reidemeister moves.
This way one does not get an invariant of links but 
often a simple correction/balancing allows to 
construct a true invariant\footnote{Kauffman approach has a good interpretation in terms
of framed links (the approach is used in Chapter X on skein modules of 3-dimensional manifolds).}.

\begin{definition}\label{2:5.4}
Two diagrams are called regularly isotopic if one can be obtained 
from the other via a finite sequence of Reidemeister moves of 
types two and three.
The definition of regular isotopy makes sense for 
orientable as well as for unorientable diagrams.
\end{definition}

While working with regular isotopy invariants
one is able to take into account some properties of diagrams which are 
not preserved by Reidemeister moves of the first type.

\begin{lemma}\label{2:5.5}
Let  the writhe number $w(L)$ of an oriented link diagram $L$ be defined by 
$w(L)={\sum}_{i} sgn(p_i)$, where the sum is taken
over all crossings of $L$.
Then $w(L)$ is regular isotopy invariant and
$w(-L)=w(L)=-w(\bar L)$. The number $w(L)$ is also called the Tait \footnote{Peter 
Guthrie Tait (1831-1901) was a Scottish physicist
who, influenced by vortex theory of atoms by  
W.~Thompson (Lord Kelvin), was tabulating diagrams of links.
The number of crossings and the number $w(L)$ 
were important ``invariants'' of this tabulation.}
number of the diagram of the link $L$ and denoted by $Tait(L)$.
Sometimes this number is also called a twist number of the diagram
 and denoted by $tw(L)$.
\end{lemma}

\begin{proof}
It is enough to show that the second Reidemeister move cancels or produces
two crossings with opposite signs and the third Reidemeister move 
preserves signs of all crossing. Moreover, we note 
that passing from the diagram $L$ to the diagram with the opposite orientation
$-L$ does not change the sign of a crossing and passing from $L$
to its mirror image $\bar L$ changes the sign of all crossings.
\end{proof}

The idea of Kauffman is based on the observation that the trivial 
knot can be represented (up to ambient isotopy) by different regular
isotopy classes of diagrams and 
to any such a class we can associate different values of some 
invariant. 
To any diagram $T_1$ representing the trivial knot, Kauffman 
associates the monomial $a^{w(T_1)}$. 
Subsequently, the Kauffman definition
of invariants is similar to the definition of Conway and Jones
polynomials, $P_L(x,y)$ and $Q_L(x)$. 
While passing from regular isotopy invariants to
isotopy invariants, Kauffman applies the following fact.
\begin{lemma}\label{2:5.6}
Let us consider the following elementary move on a diagram of a link, denoted 
by $(R^{\pm 1}_{0.5})$  and called the first weakened (or balanced) Reidemeister move.
That is the move which allows to create or 
to cancel the pair of curls of the opposite signs,
see Fig.~5.3. 
Let us observe that signs of crossings in
curls do not depend on the orientation of the diagram.
\\
\begin{center}
{\psfig{figure=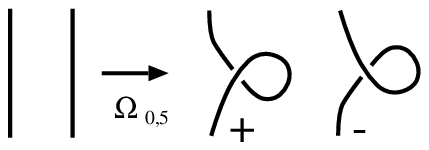,height=3.3cm}}
\end{center}
\vspace*{.2cm}
\begin{center}
Fig.~5.3
\end{center}

Then one can obtain a diagram $L_1$ from another diagram $L_2$ by regular isotopy 
and first weakened Reidemeister move if and only if 
$L_1$ is isotopic to $L_2$ and $w(L_1) = w(L_2)$. 
\end{lemma}

The proof of \ref{2:5.6} becomes clear once we note that the move
$R^{\pm 1}_{0.5})$ enables us to carry a twist to any place in the diagram.

Now we show how, using the Kauffman approach,  that 
the Conway polynomial can be generalized to the Homflypt polynomial and the 
polynomial $Q_L$ can be modified to another invariant which we will 
call Kauffman polynomial (or Kauffman polynomial of two variables).

\newpage
\begin{theorem}[\cite{K-4}]\label{2:5.7}
\ 
\begin{enumerate}
\item There exists a uniquely defined invariant of regular isotopy
of oriented diagrams, denoted by $(R_L(a,z))$, which is a polynomial in
$\Z[a^{\pm 1},z^{\pm 1}]$ and which satisfies the following conditions:
\begin{enumerate}
\item $R_{T_1}(a,z) = a^{w (T_a)}$, where $T_1$ is a diagram of a knot
isotopic to the trivial knot.

\item $R_{L_+} (a,z) - R_{L_-}(a,z) = zR_{L_0}(a,z)$.
\end{enumerate}
\item For any diagram $L$ we consider a polynomial 
$G_L(a,z) = a^{-w (L)} R_L(a,z)$.
Then $G_L(a,z)$ is an invariant of ambient isotopy of oriented
links and it is equivalent to the  Homflypt polynomial, that is
$$G_L(a,z) = P_L(x,y) \mbox{ where } x=\frac{a}{z}, y=\frac{-1}{za}.$$
\end{enumerate}
\end{theorem}

Proof. The method of the proof of Theorem \ref{2:1.2} can be used
to prove \ref{2:5.7} but we will present a proof based on the existence of 
the Homflypt polynomial, since its existence was already proved.
Namely, we consider a substitution in  $P_L(x,y)$ by setting 
$x=\frac{a}{z}$ and $y=\frac{-1}{az}$. As the result we obtain 
a polynomial invariant of ambient isotopy classes of oriented links
which we call $\tilde{G}_L(a,z)$ and which satisfies the following conditions:
\begin{enumerate}
\item[(a)] $\tilde{G}_{T_1} (a,z)= 1$
\item[(b)]  $a\tilde{G}_{L_+}(a,z) - \frac{1}{a}\tilde{G}_{L_-}(a,z) = z\tilde{G}_{L_0}(a,z)$.
\end{enumerate}

Now to any oriented diagram we associate a polynomial 
$\tilde{R}_L(a,z) = a^{w(L)}\tilde{G}_L(a,z)$. It is easy to see that 
the first Reidemeister move changes the value of 
$\tilde{R}_L(a,z)$ according to the sign of the curl:
$$\tilde{R}_{\parbox{0.2cm}{\psfig{figure=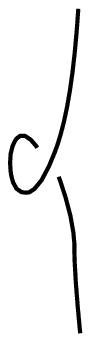,height=0.7cm}}}(a,z)=
a\tilde{R}_{\parbox{0.1cm}{\psfig{figure=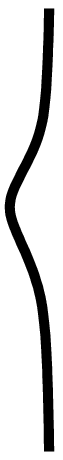,height=0.6cm}}}(a,z)$$
$$\tilde{R}_{\parbox{0.2cm}{\psfig{figure=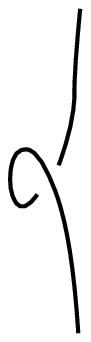,height=0.7cm}}}(a,z)=
a^{-1}\tilde{R}_{\parbox{0.1cm}{\psfig{figure=notwist.eps,height=0.6cm}}}(a,z)$$

Moreover the second and the third Reidemeister move do not change  
$\tilde{R}_L(a,z)$. Thus $\tilde{R}_L(a,z)$ is an invariant of regular
isotopy of oriented diagrams and it satisfies the condition
$\tilde{R}_{T_1}(a,z)= a^{w(T_1)}$. It is also easy to verify
that $\tilde{R}_L(a,z)$ satisfies the condition (b)) of Theorem
\ref{2:5.7}.  Moreover, since any diagram has a resolving tree therefore
the polynomial  $R_L(a,z)$, if it exists, 
is uniquely determined by the conditions (a) and (b). Thus,
setting $R_L(a,z) = \tilde{R}_L(a,z)$ and $G_L(a,z) =
\tilde{G}_L(a,z)$, we complete the proof of Theorem \ref{2:5.7}.

\begin{theorem}[\cite{K-5}]\label{2:5.8}
\ \\
\begin{enumerate}
\item
[(1)] There exists a uniquely defined invariant $\Lambda$ which to any
regular isotopy class attaches a polynomial in $\Z[a^{\pm 1},z^{\pm 1}]$ 
and which satisfies the following conditions:
\item
[(a)] ${\Lambda}_{T_1}(a,z)= a^{w(T_1)}$ ,
\item
[(b)] ${\Lambda}_{L_+} + {\Lambda}_{L_-}= z({\Lambda}_{L_0}+
{\Lambda}_{L_{\infty}})$.
\item 
[(2)]
For any oriented diagram $D$ we define $F_D(a,z)=a^{w(D)}
{\Lambda}_D(a,z)$. Then $F_D(a,z)$ is an invariant of (ambient) isotopy
classes of oriented links and it is a generalization of
the polynomial $Q$, that is $Q_L(x)=F_L(1,x)$.
\end{enumerate}
\end{theorem}

Proof. Part (1) of the theorem will be proved later in a more general context.

Part (2) follows from (1) if we notice that the first Reidemeister move 
changes the polynomial ${\Lambda}_L(a,z)$ in the following way:
$${\Lambda}_{\parbox{0.2cm}{\psfig{figure=plustwist.eps,height=0.7cm}}}(a,z)=
a{\Lambda}_{\parbox{0.1cm}{\psfig{figure=notwist.eps,height=0.6cm}}}(a,z) \mbox{\ \  and \ \ }
{\Lambda}_{\parbox{0.2cm}{\psfig{figure=minustwist.eps,height=0.7cm}}}(a,z)=
a^{-1}{\Lambda}_{\parbox{0.1cm}{\psfig{figure=notwist.eps,height=0.6cm}}}(a,z)$$

The polynomial $F_L(a,z)$ is called the Kauffman polynomial of the link $L$.

Now we will present some elementary properties of Kauffman polynomial.

\begin{theorem}\label{2:5.9}
\ \\
\begin{enumerate}
\item[(a)] $F_{T_n}=(\frac{a+a^{-1}-z}{z})^{n-1}$,
\item[(b)]
 $F_{L_1\kwad L_2} (a,z) = F_{L_1}(a,z)\cdot F_{L_2}(a,z).$
\item[(c)] $F_{L_1\sqcup L_2}(a,z) = \mu F_{L_1}(a,z)\cdot F_{L_2}(a,z)$
where $\mu = F_{T_2}=\frac{a^{-1}+a}{z}-1$ is the value of the invariant on the trivial 
link with two components.

\item[(d)] $F_L(a,z) = F_{-L}(a,z)$.
\item[(e)] $F_{\overline{L}}(a,z) = F_L(a^{-1},z)$.
\item[(f)] $F_L(a,z) = F_{m(L)}(a,z)$ where $m(L)$ is a mutant 
of the link $L$.
\end{enumerate}
\end{theorem}

The proof is rather straightforward, compare it with Lemma \ref{III:3.21},
Theorem \ref{Theorem III:3.25}, Corollary \ref{2:3.27}, Theorem \ref{2:3.9}, 
and Exercise \ref{2:5.2}).

The polynomial ${\Lambda}$ does not depend on the orientation of the diagram $D$. 
Therefore the dependence of the Kauffman polynomial $F$ on the orientation
of components of $D$ is easy to describe --- this is because
$F_D(a,z)$ differs from ${\Lambda}_D(a,z)$ by a power of $a$.

\begin{lemma}\label{2:5.10}
Let $D= \{ D_1, D_2, \ldots, D_i, \ldots, D_n\}$ be a diagram of an
oriented link with $n$ components and let $D' := \{ D_1,
D_2,\ldots, -D_i,\ldots,D_n\}$. We set $\lambda_i = \frac{1}{2}\sum \sgn
p_j$ where the summation is over all crossings
of $D_i$ with $D - D_i$. (The number $\lambda_i$ is called 
the linking number of $D_i$ and $D-D_i$ and it is often denoted by 
$\lk (D_i,D-D_i)$; it is an invariant of ambient isotopy of $D$ ($D_i$ should be kept as $i$ component);
this follows easily from considering Reidemeister moves.
Then
$$F_{D'}(a,z) =a^{-w (D')} a^{w (D)} F_D(a,z) =
a^{4\lambda_i}F_D(a,z).$$
\end{lemma}

Let us note that the above lemma implies that the Kauffman polynomial 
gives only as much information about orientations of $D$ as comes
from the linking numbers of its components $D_i$ with complements $D-D_i$
(we coded this information in weighted simplex of global linking numbers; Definition \ref{Definition III:3.44} and 
Exercise \ref{Exercise III:3.46}).

The Kaufman polynomial is much more useful for distinguishing a given link from
its mirror image (achirality of a link). However, we gave already an example of a (pretzel) link of
 two components, which is a mutant of its mirror image and they are not isotopic (Fig.~3.7),
although Kauffman expected initially that such a situation is impossible for knots.

\begin{conjecture}\label{III:5.11}[\cite{K-5}.] 
If a knot $K$ is not isotopic to its mirror image 
($\overline{K}$ or $-\overline{K}$) then
$F_K(a,z)\neq F_{\overline{K}}(a,z)$.
\end{conjecture}

The knots  $9_{42}$ and $10_{71}$ (according to the Rolfsen's notation
\cite{Ro-1}, see Fig.~5.4) contradict the conjecture, since they are 
isotopic to their mirror images still they have the same Kauffman polynomials
as their mirror images.

\begin{center}
\begin{tabular}{c} 
\includegraphics[trim=0mm 0mm 0mm 0mm, width=.65\linewidth]
{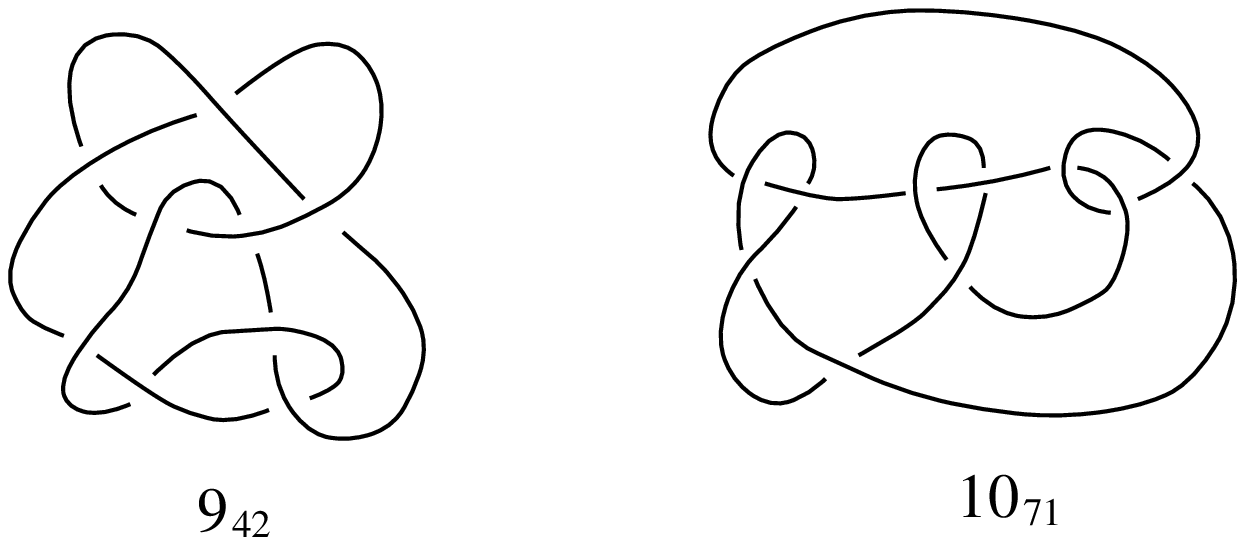}\\
\end{tabular}
\\
Fig.~5.4
\end{center}

\begin{remark}\label{2:5.12}
It was an open problem for a while whether the knot $9_{42}$ is algebraically equivalent to $\overline{9}_{42}$? 
(It is easy to find out that  $\sigma(9_{42})
= -2 = -\sigma(\overline{9_{42}})$ and therefore the two knots are not
\skein \ equivalent). However from Sikora result \cite{Si-2} it follows that no invariant coming from 
a Conway algebra can distinguish $9_{42}$ from $\overline{9}_{42}$.
\end{remark}

The Kauffman polynomial is also a generalization of the Jones polynomial.

\begin{theorem}[\cite{Li-2}]\label{2:5.13}
$$V_L(t)=F_L(t^{3/4},-(t^{-1/4}+t^{1/4}))$$
\end{theorem}

Proof. Let us begin by characterizing the Kauffman polynomial without
using regular isotopy.

\begin{lemma}\label{2:5.14}
The Kauffman polynomial is uniquely determined by the following conditions:
\begin{enumerate}
\item [(1)] $F_{T_1}(a,z) = 1$,
\item[(2)] Consider the Conway skein triple of oriented link diagrams, $L_+$, $L_-$, and $L_0$ of 
Figure 1.1. Let $cp(L)$ denote the number of components of the link $L$. We consider two cases 
dependent on whether the modified crossing in $L_+$ is a selfcrossing or the mixed one.
\begin{enumerate}
\item [(i)]
$\mbox{cp} (L_{+}) = \mbox{cp}(L_0)-1$ (a selfcrossing case):
 $$aF_{L_+}(a,z) +\frac{1}{a}F_{L_-}(a,z) = z(F_{L_0}(a,z) +
a^{-4\lambda}F_{L_\infty}(a,z)),$$
where $\mbox{cp} (L_{\infty})=\mbox{cp} (L_+)$ and $L_\infty$ is given any one of two possible orientations
and $\lambda = \lk (L_i, L_0-L_i)$ with $L_i$ being the new component of
$L_0$, the orientation of which does not agree with the orientation of the
corresponding component of $L_\infty$. 
\item [(ii)] 
$\mbox{cp} (L_{+}) = \mbox{cp}(L_0)+1$ (a mixed crossing case): 
$$aF_{L_+}(a,z) +\frac{1}{a}F_{L_-}(a,z) = z(F_{L_0}(a,z) +
a^{-4\lambda +2}F_{L_\infty}(a,z)),$$
where $\mbox{cp} (L_{\infty})= \mbox{cp} (L_0)= \mbox{cp} (L_+)-1$ and  
 $L_\infty$ is given any one of two possible orientations
and $\lambda = \lk (L_i, L_+-L_i)$ with $L_i$ being the component of 
$L_+$, the orientation of which does not agree with the orientation of the
corresponding component of $L_\infty$.
\end{enumerate}
\end{enumerate}
\end{lemma}

\begin{proof} In view of the definition of $F_L(a,z)$, Lemma \ref{2:5.14} 
follows easily from Lemma \ref{2:5.10}.  We show the calculation, as an example, 
in the case 2(ii). By the definition we have;
$${\Lambda}_{L_+} + {\Lambda}_{L_-}= z({\Lambda}_{L_0}+ {\Lambda}_{L_{\infty}}), \mbox { therefore}$$
$$a^{w(L_+)}{F}_{L_+} + a^{w(L_-)}{F}_{L_-}= z(a^{w(L_0)}{F}_{L_0}+ 
a^{w(L_{\infty})}{F}_{L_{\infty}}) \mbox { so:}$$
$$ a{F}_{L_+} + a^{-1}{F}_{L_-}= z({F}_{L_0}+ a^{w(L_{\infty})- w(L_0)}{F}_{L_{\infty}})$$
which reduces to the formula of 2(ii) as $w(L_{\infty})- w(L_0)= 2-2lk(L_i,L_+-L_i)$.

\end{proof}

Proceeding with the proof of Theorem \ref{2:5.13} we need an additional
characterization of the Jones polynomial. Let us recall that the  Jones polynomial
is uniquely defined by the following  conditions:
\begin{enumerate}
\item $V_{T_1}(t) = 1$, 
\item $\frac{1}{t}V_{L_+}(t) - tV_{L_-}(t) = (\sqrt{t} -
\frac{1}{\sqrt{t}})V_{L_0}(t)$.
\end{enumerate}

\begin{lemma}[Jones reversing formula]\label{2:5.15}
Suppose that $L_i$ is a component of an oriented link $L$
and $\lambda = \lk (L_i,L-L_i)$. If $L'$ is a link obtained from
 $L$ by reversing the orientation of the component $L_i$ then
$V_{L'}(t) = t^{-3\lambda}V_L(t)$.
\end{lemma}
Notice that the Jones reversing formula formula generalizes immediately to the case when 
we reverse the orientation of a sublink $U$ of a link $L$, that is having $\lambda = \lk (U,L-U)$ 
we obtain $V_{L'}(t) = t^{-3\lambda}V_L(t)$ where $L'$ is obtained from $L$ by reversing orientation of $U$.

Proof of Lemma \ref{2:5.15}. We present the original proof by Lickorish and Millett \cite{L-M-3}. 
Very simple proof was found by L.Kauffman (using Kauffman bracket polynomial).

The proof consists of five steps.

\begin{enumerate}
\item[(1)] Simple calculations show that Lemma \ref{2:5.15} is true for links
pictured in Fig.~5.5.

\begin{center}
\begin{tabular}{c} 
\includegraphics[trim=0mm 0mm 0mm 0mm, width=.65\linewidth]
{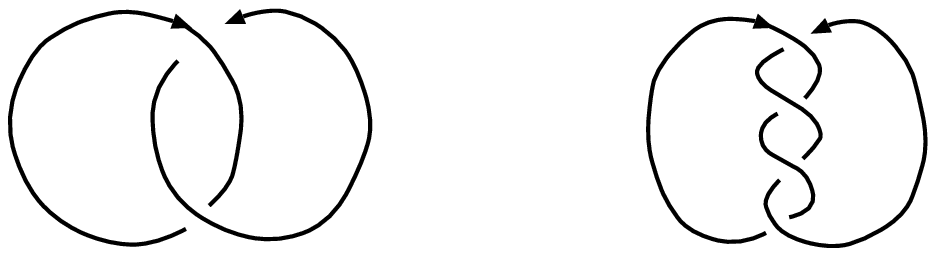}\\
\end{tabular}
\\
Fig.~5.5
\end{center}

\item[(2)] $V_L(t) = V_{-L}(t)$ and moreover, since $V_{L_1\kwad L_2}(t) =
V_{L_1}(t)\cdot V_{L_2}(t)$\  and 
\ $V_{\overline{L}}(t) = V_L(t^{-1})$, thus 
if Lemma \ref{2:5.15} is true for two links $L_1$ and $L_2$ then 
it is true also for $\overline{L}_1$, $L_1\kwad L_2$ and $L_1\kwad -L_2$.

\item[(3)] Fix a diagram of $L$. A standard induction with respect to the number 
of the (self)crossings of $L_i\subset L$  
and with respect to the number of bad (self)crossings of
$L_i$, implies that we can assume that $L_i$ is a descending diagram and
thus, in particular, $L_i$ is a diagram of the trivial knot. 

\item[(4)] Let us consider a 2-dimensional disc $D$ in $S^3$, the boundary
of which is the knot $L_i$.  We may assume that $L-L_i$ meets $D$
transversally in a finite number of points, say in $n$ points.
Now we proceed by using induction with respect to the number $n$.
The initial conditions of the induction will be discussed later in the
step (5). Now we assume that $n\geq 4$. 
Figure 5.6 presents Conway's skein triple of diagrams, $L_+$, $L_-$, and $L_0$,
where $L_0$ represents $L_i$. 
The disc bounded by $L_i$ is cut by $L-L_i$ in points which are marked
by crosses. The knot $L_i$ becomes in $L_-$ a trivial link with two components
$\gamma^-_1$ and $\gamma^-_2$ which bound discs meeting the remainder of  
$L_-$ in $n_1$ and $n_2$ points, respectively. Besides, the linking 
number of $\gamma^-_1$ and $\gamma^-_2$ with the remainder of the link $L_-$ 
is $\lambda_1$ and $\lambda_2$, respectively.
In the case of $L_+$ the situation is similar, only this time 
$\gamma^+_1$ and $\gamma^+_2$ are linked, i.e.~$\lk(\gamma^+_1,\gamma^+_2) = 1$),
see Fig.~5.6.

\begin{center}
\begin{tabular}{c} 
\includegraphics[trim=0mm 0mm 0mm 0mm, width=.65\linewidth]
{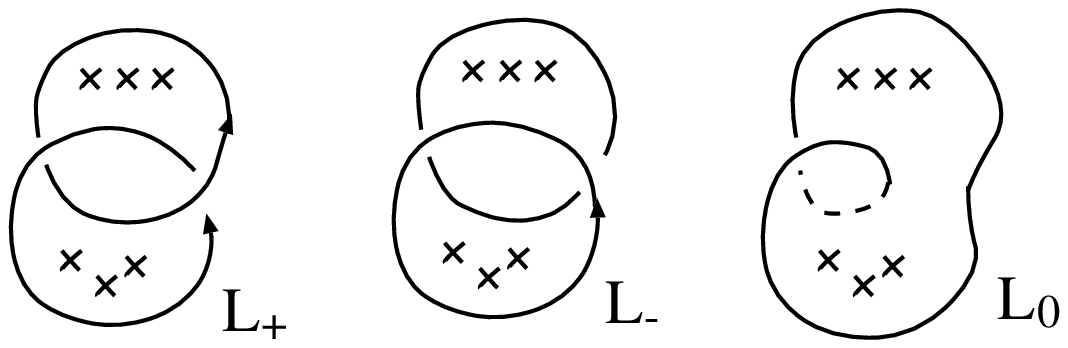}\\
\end{tabular}
\\
Fig.~5.6
\end{center}

Now, we have $n_1+n_2 = n$ and $\lambda_1 + \lambda_2 = \lambda$.
We may assume that both $n_1$ and $n_2$ are 2 at least. 
Now, let $L_+'$, $L_-'$, $L_0'$ be links obtained from links 
$L_+$, $L_-$, $L_0$, respectively, by changing the orientation
of components $(\gamma^+_1,\gamma^+_2)$, $(\gamma^-_1, \gamma^-_2)$ and 
$L_i$, respectively. Then 
$$\frac{1}{t}V_{L_+}(t) - tV_{L_-}(t) =
(\sqrt{t}-\frac{1}{\sqrt{t}})V_{L_0}(t)$$
and 
$$\frac{1}{t}V_{L_+'}(t) - tV_{L_-'}(t) =
(\sqrt{t}-\frac{1}{\sqrt{t}})V_{L_0'}(t).$$

The inductive assumption allows to apply Lemma \ref{2:5.15} to $L_+$
and $L_-$. Changing the orientation of $\gamma^+_1$ and $\gamma^+_2$ in $L_+$
we obtain:
$$V_{L_+'}(t) = t^{-3(\lambda_1 +\lambda_2)}V_{L_+}(t)$$
while changing the orientation of $\gamma^-_1$ and $\gamma^+_2$ in $L_-$ 
we similarly get:
$$V_{L_-'}(t) = t^{-3(\lambda_1 +\lambda_2)}V_{L_-}(t).$$
Therefore we get $t^{-3\lambda}V_L(t)=V_{L'}(t)$ as needed.

If $n<4$ then there are problems with the inductive assumption for
 $L_+$. However, we can extends the above argument for the case
$n=3$, $\lambda =\pm 3$, provided that the lemma is true for 
$n\leq 2$ and for $n=3$, $\lambda = \pm 1$. 

\item[(5)] Now let us consider the case of $n=3$, $\lambda = \pm 1$. 

We are to show that for the link pictured at Fig.~5.6 --- where the  
rectangle has to be replaced by a diagram (called a tangle) with three arcs coming into it
and three arcs  existing out of it,  we have $t^{-3}V_L(t) = V_{L'}(t)$.
We show, in Fig.~5.7 the case of $\lambda = 1$).

\begin{center}
\begin{tabular}{c} 
\includegraphics[trim=0mm 0mm 0mm 0mm, width=.5\linewidth]
{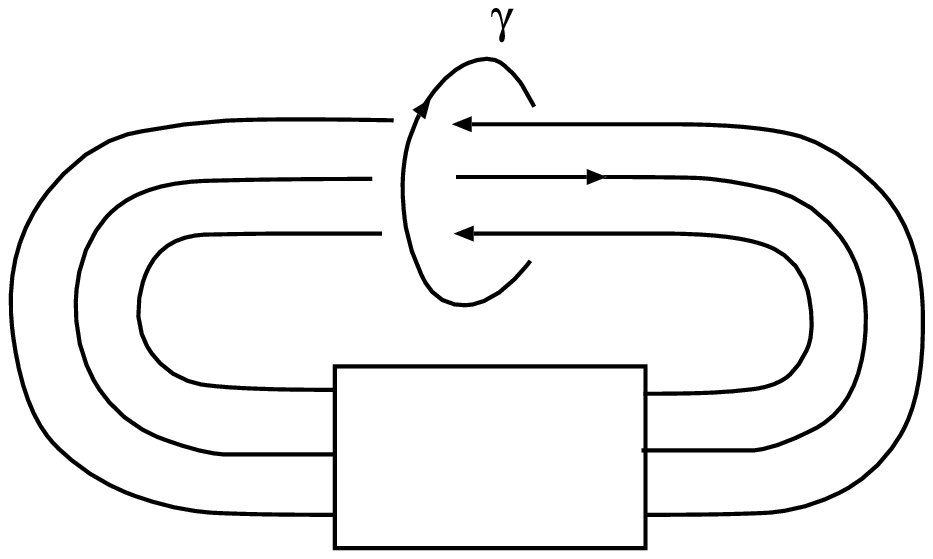}\\
\end{tabular}
\\
Fig.~5.7
\end{center}

Now one can use 
standard induction with respect to the number of crossings lying
in the above rectangle and with respect to the number of bad crossings
(for some choice of base points) in the rectangle 
to reduce the general situation in the rectangle to one of the six cases 
presented in Fig.~5.8.
Subsequently, it is enough to apply steps (1) and (2) of the proof. The case of 
the second rectangle is illustrated in Figure 5.9.

\begin{center}
\begin{tabular}{c} 
\includegraphics[trim=0mm 0mm 0mm 0mm, width=.7\linewidth]
{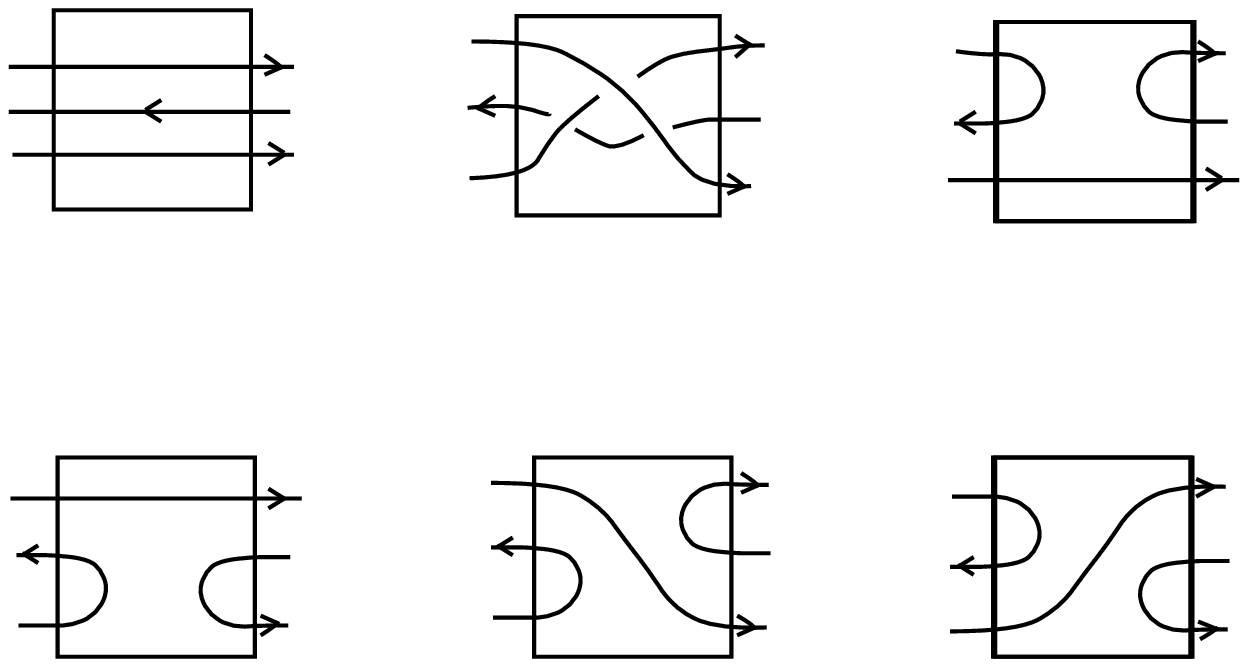}\\
\end{tabular}
\\
Fig.~5.8 six basic 3-tangles
\end{center}

\begin{center}
\psfig{figure=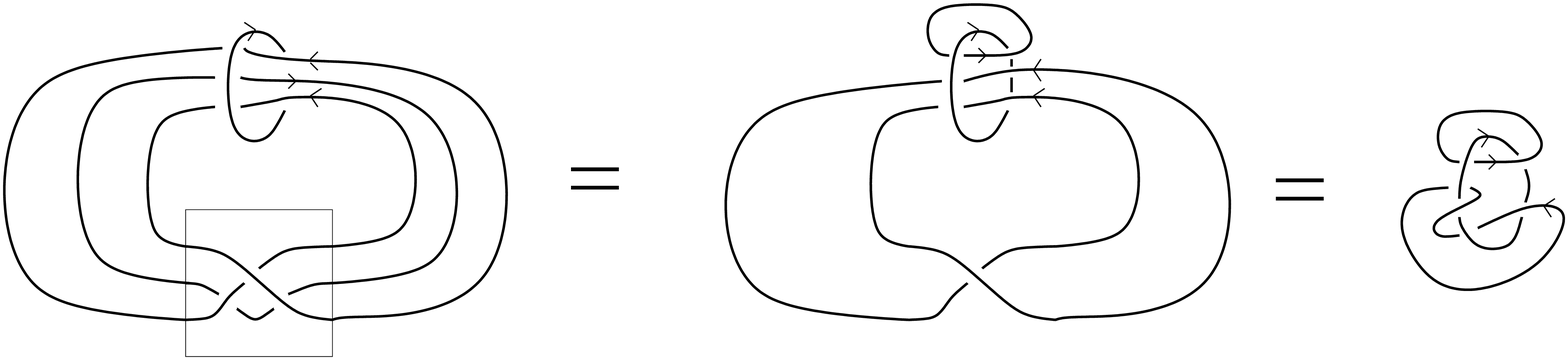,height=2.3cm}
\end{center}
\centerline{Fig. 5.9; Reduction in the case of the second rectangle}

Similarly we deal with the case $n\leq2$.
\end{enumerate}
This completes the proof of Lemma \ref{2:5.15}

The next step in the proof of Theorem \ref{2:5.13} is the so-called
$V_\infty$-formula \cite{Bi-2} in which we compare Jones polynomials of links 
$L_+,L_-$, and $L_{\infty}$ where $L_{\infty}$ has one of possible two orientation which agree 
with orientation on $L_+$ for component not taking part in the modified crossing.

\begin{lemma}[Birman]\label{2:5.16}
 \ 
\begin{enumerate}
\item[(1)] The case $cp(L_+)=cp(L_0)-1$ where $cp(L)$ denotes the number of components.
Let $L_\infty$ be given an orientation (as explained above)  and let 
$\lambda= lk(L_i,L_0-L_i)$ with $L_i$ being the new component of
$L_0$, the orientation of which does not agree with the orientation of the
corresponding component of $L_\infty$.
 Then
$$\sqrt{t}V_{L_+}(t) - \frac{1}{\sqrt{t}}V_{L_-}(t) = (\sqrt{t} -
\frac{1}{\sqrt{t}})t^{-3\lambda}V_{L_\infty}(t).$$

\item[(2)] The case $c(L_+)=c(L_0)+1$: \ \ Let $L_\infty$ be given any orientation (as explained above) 
and let $\lambda = \lk (L_i, L_+-L_i)$ with $L_i$ being the component of
$L_+$, the orientation of which does not agree with the orientation of the
corresponding component of $L_\infty$.
Then:
$$\sqrt{t}V_{L_+}(t) - \frac{1}{\sqrt{t}}V_{L_-}(t) = (\sqrt{t} -
\frac{1}{\sqrt{t}})t^{-3(\lambda-\frac{1}{2})}V_{L_\infty}(t).$$
\end{enumerate}
\end{lemma}

Proof.
\begin{enumerate}
\item $cp(L_+)=cp(L_0)-1$. 

Let us consider a diagram $X$ with two crossings, $p$ and $q$
(Fig.~5.10), such that $L_0 = X^{p\ q}_{-\ +}$, 
$L_+ = X^{p\ q}_{+\ 0}$ and $L_- = X^{p\ q}_{-\ 0}$. 
Considering the crossing $q$ we obtain:

\item $$ (a)\ \ -tV_{L_0}(t) +\frac{1}{t}V_X(t) = (\sqrt{t} -
\frac{1}{t})V_{L_-}(t).$$

Now let us change the orientation of the component of $L_0$  
which exists at the upper right-hand corner (Fig.~5.10) and call the
the resulting link $L_0'$. Similarly, we change $X$ to $X'$ (Fig.~5.10).

\begin{center}
\begin{tabular}{c} 
\includegraphics[trim=0mm 0mm 0mm 0mm, width=.5\linewidth]
{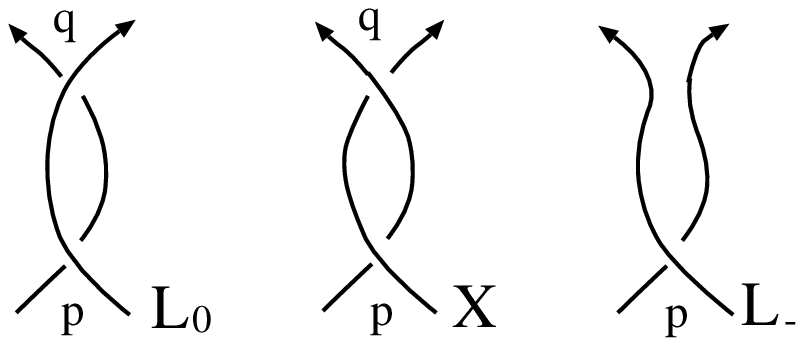}\\
\end{tabular}
\\
Fig.~ 5.10
\end{center}

Now let us choose the orientation of 
$L_\infty = X^{p\ q}_{-\ \infty} = X^{p\ q}_{\infty\ -} = X'^{p\ q}_{+\ 0}$ 
so that it agrees with the orientation of $L_0'$ (Fig.~5.11).

\begin{center}
\begin{tabular}{c} 
\includegraphics[trim=0mm 0mm 0mm 0mm, width=.5\linewidth]
{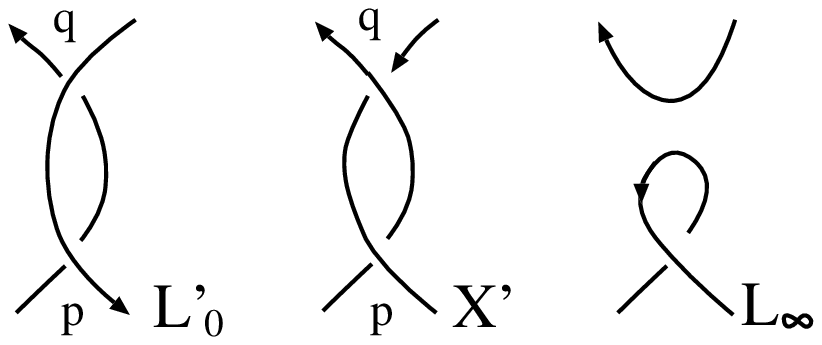}\\
\end{tabular}
\\
Fig.~5.11
\end{center}

Because of the diagrams at Fig.~5.11, considering the crossing $q$, we get 
$-tV_{X'}(t) + \frac{1}{t}V_{L_0'}(t) = (\sqrt{t} -\frac{1}{t})V_{L_\infty}(t)$. 
Moreover, because of Lemma \ref{2:5.15}, we have
$V_{L_0'}(t) = t^{3\lambda}V_{L_0}(t)$ and $V_{X'}(t) = t^{3(\lambda
-1)}V_X(t)$ and it follows that
$$ \mbox{(b)\ \ \ \ \ \ \ } \frac{1}{t}V_{L_0}(t) -\frac{1}{t^2}V_X(t) = (\sqrt{t} -
\frac{1}{\sqrt{t}})t^{-3\lambda}V_{L_\infty}(t).$$ 
The triple consisting of $L_+$, $L_-$ and $L_0$ gives us the equation
$$  \mbox{(c)\ \ \ \ \ \ \ \ } -tV_{L_+}(t) +\frac{1}{t}V_{L_-}(t) = (\sqrt{t} - \frac{1}{\sqrt{t}})V_{L_0}(t).$$

The equation $(b) +\frac{1}{t}(a) - \frac{1}{\sqrt{t}} (c)$ (i.e.~the sum of the
above equations with the respective coefficients) is the $V_\infty$-formula
we looked for.

\item $cp(L_+) =cp(L_0)+1$.

Let $L_+'$ be a link obtained from $L_+$ by changing of the orientation of $L_i$,
and $L_-'$ is obtained from $L_-$ in a similar manner. Now the smoothing of $L_+'$ is 
exactly the link $L_0'=L_{\infty}$.
Let us apply the defining equation of the Jones polynomial
for the triple  $L_-'$, $L_+'$ and $L_0'$. As the result we get:
$$-tV_{L_-'}(t) +\frac{1}{t}V_{L_+'}(t) = (\sqrt{t} -
\frac{1}{\sqrt{t}})V_{L_0'}(t),$$
further, by Lemma \ref{2:5.15}, we have 
$V_{L_+'}(t) = t^{3\lambda}V_{L_+}(t)$ and
$V_{L_-'}(t) = t^{3(\lambda -1)}V_{L_-}(t)$. Therefore
$$\mbox{(d)\ \ \ \ \ \ \ } -\frac{1}{t^2}V_{L_-}(t) +\frac{1}{t}V_{L_+}(t) = (\sqrt{t} -
\frac{1}{\sqrt{t}})t^{-3\lambda}V_{L_\infty}(t),$$
which completes the proof of the part (2) and of the whole
Lemma \ref{2:5.16}.
\end{enumerate}

Now we can conclude the proof of Theorem \ref{2:5.13}. As before,
we have to consider two cases:
\begin{enumerate}
\item[(1)] $pc(L_+) = c(L_0)-1$.

Let us consider the formula defining the Jones polynomial and also the  
$V_\infty$-formula from Lemma \ref{2:5.16}. We get:
$$-tV_{L_+}(t) +\frac{1}{t}V_{L_-}(t) = (\sqrt{t} -
\frac{1}{\sqrt{t}})V_{L_0}(t)$$ and
$$\sqrt{t}V_{L_+}(t) -\frac{1}{\sqrt{t}}V_{L_-}(t) = (\sqrt{t} -
\frac{1}{\sqrt{t}})t^{-3\lambda}V_{L_\infty}(t).$$

Now, it is enough to add the above equations to get the first formula 
of Lemma \ref{2:5.14}(2) for $a = t^\frac{4}{3}$ and $z = -(t^{-\frac{1}{4}}
+t^\frac{1}{4})$.

\item[(2)] $cp(L_+) = cp(L_0)+1$.

Exactly the same argument as in the first case yields the second formula
of Lemma \ref{2:5.14}(2).
\end{enumerate}

This completes the proof of Theorem \ref{2:5.13}.

The proof of Theorem  \ref{2:5.13} which we presented above reflects
the history of understanding the Jones polynomial as in \cite{Jo-1}, 
\cite{Bi-2}, \cite{L-M-3} and \cite{Li-2}. The proof can be actually
shortened if, instead of $V_L(t)$,  we consider the polynomial 
$t^{\frac{3}{4}w(L)}V_L(t)$ which is an invariant of regular isotopy.

Earlier, in Theorems 5.9, 5.13 and Lemmas 5.10, 5.14, we have already presented
some basic properties of the Kauffman polynomial 
the result below is another interesting and unexpected property.

While working on  $t_4$ moves I have suggested (and partially proved in April
of 1986) that $F_K(a,-a-a^{-1})=1$ if $K$ is a knot, \cite{Mo-4}. 
Lickorish and Millett proved this conjecture and at the same time they found
a formula for the Kauffman polynomial of an arbitrary link and for 
$z=-a-a^{-1}$ \cite{L-M-4} (c.f.~\cite{Lip-1}). 
The formula  for  $F_L(a,-a-a^{-1})$ was found independently, 
in a more general context, by Turaev \cite{Tu-1}.
Below we present a proof o the result, which needs very little of
 computations and no previous knowledge (guess) of the formula.
It is based on the idea of presenting an  unoriented link as a sum
of all oriented links which are obtained  from the given one by setting all
possible orientations \cite{P-7}.\footnote{This idea was 
presented for the first
time in \cite{Gol} and it was credited to Dennis Johnson
(compare also with \cite{Tu-2,H-P-2}).} 

\begin{theorem}\label{5.17a}
For any oriented link $L$ we have the following formula
  $$F_L(a,-a-a^{-1}) = ((-1)^{cp(L)-1}/2)\Sigma_{S \subset
L}a^{-4lk(S,L-S)}, $$ where 
the summation is over all sublinks $S$ of
the link $L$ (including $S=\emptyset$), 
and $cp(L)$, $lk(S,L-S)$ denote, respectively, 
the number of components of the link $L$ and 
global linking number of $S$ and $L-S$.

\end{theorem}

\begin{proof}
Let us consider the following very simple invariant of oriented diagrams
$g(D):=(-1)^{cp(D)}a^{w(D)}\in \Z[a^{\pm 1}]$. Let us note that 
$g(D_+)=-ag(D_0)=a^2g(D_-)$. Moreover, $g$ is an invariant of regular isotopy
and the first Reidemeister move changes $g$ by 
$a^{\pm 1}$. Now, let us extend the definition of $g$ for unoriented
diagrams. To avoid ambiguity the extended invariant is denoted by $G$.
If $D$ is an unoriented diagram then the invariant $G$ is given by the following
formula
$$G(D)=\Sigma_{D'\in OR(D)}g(D')$$ 
where $OR(D)$ is the set of all orientations of the diagram $D$.
We see immediately that $G$ is an invariant of regular isotopy
of unoriented links. Besides, $G$ has the following properties: 
\begin{enumerate}
\item
[(1)] $G({R}_{1+}(D))=aG(D)$ and $G({R}_{1-}(D))=a^{-1}G(D)$,
where ${R}_{1+}$ (respectively, ${R}_{1-}$) is the first Reidemeister move
which introduces a positive (respectively, negative) twist,

\item
[(2)] $G(D_+)=-a(G(D_0)+G(D_{\infty}))$ 
if the crossing of $D_+$ is a positive selfintersection.
\item
[(3)] $G(D_-)=-a^{-1}(G(D_0)+G(D_{\infty}))$ 
if the crossing of $D_-$ is a negative selfintersection.
\item
[(4)] $G(D_+)=-aG(D_0)-a^{-1}G(D_{\infty})$ 
in the case of a mixed crossing (the meaning of $D_+$ (resp. $D_-$) for a mixed crossing is chosen as follows:
We consider an orientation on $D_{\pm 1}$ in such a way that an oriented smoothing gives $D_0$; $D_+$ (resp. $D_-$) 
means that the smoothed crossing is positive (resp. negative).
\item 
[(5)] $G(D_-)=-a^{-1}G(D_0)-aG(D_{\infty})$ in the case of a mixed crossing with $D_-$ explained above.

\end{enumerate}

Formulas (2) and (3) follow from the observation that orientations of $D_0$ are in bijection with 
orientations of $D_{\pm}$ and $D_{\infty}$ (taken together). Similarly 
(4) and (5)  follow from the observation that orientations of $D_{\pm}$ are in bijections with 
orientations of $D_0$ and $D_{\infty}$.

Now, adding side-to-side equations  (2) and (3) as well as (4) and (5) 
we obtain the following formula  for an arbitrary crossing of an  unoriented
diagram $D$:
$$(6)\ \ \ \ \ G(D_+)+G(D_-)=(-a-a^{-1})(G(D_0)+G(D_{\infty})).$$
The properties (1) and (6) are also conditions defining the Kauffman polynomial of
unoriented links, ${\Lambda}_D(a,z)$ for $z=-a-a^{-1}$.
We have yet to compare the initial conditions. Namely, for a trivial circle
$T_1$ we have $G(T_1)=-2=-2{\Lambda}_{T_1}$. Therefore
$G(D)=-2{\Lambda}_D(a,-a-a^{-1})$ and further for a diagram of an oriented link 
we have: 
$$(7)\ \ \ \ \ F_{D'}(a,-a-a^{-1}) = a^{-w(D')}{\Lambda}_D(a,-a-a^{-1})= -\frac{1}{2}a^{-w(D')}G(D),$$ 
where $D$ is an unoriented diagram obtained from $D'$ by forgetting 
about its orientation.
By definition we have 
$$(8)\ \ \ \ \ G(D)=\Sigma_{D'\in OR(D)}g(D')=
(-1)^{cp(D)}\Sigma_{D'\in OR(D)}a^{w(D')}.$$
Now, if $L$ is an oriented diagram of a link and $S$ is its sublink,
and $L_S$ is obtained from $L$ by changing the orientation of $S$, then
$w(L_S)-w(L)=2lk(L_S)-2lk(L)=-4lk(S,L-S)$. This equality and the formula (8)
can be applied for any orientation $D''$ of the unoriented diagram $D$ to get:
$$(9)\ \ \ \ \ a^{-w(D'')}G(D)=(-1)^{cp(L)}
\Sigma_{S \subset D''}a^{-4lk(S,D''-S)}.$$
Finally, formulas (7) and (9) give 
$$F_{D'}(a,-a-a^{-1}) =  -\frac{1}{2}a^{-w(D')}G(D)= \frac{1}{2}(-1)^{w(D')-1}\Sigma_{S\subset D'}a^{-4lk(S,D'-S)},$$
which is the formula of Theorem \ref{5.17a}.

\end{proof}

\begin{exercise} Show that for unoriented diagrams $D_+,D_-,D_0$, and $D_{\infty}$ with conventions like in (2)-(5) 
of the proof of Theorem \ref{5.17a}.
\begin{enumerate}
\item
[(a)] $G(D_+)-G(D_-)=(a^{-1}-a)(G(D_0)+G(D_{\infty}))$\ 
in the case of selfcrossing,
\item
[(b)] $a^{-1}G(D_+)-aG(D_-)=0$\ in the case of selfcrossing,
\item
[(c)] $G(D_+)-G(D_-)=(a^{-1}-a)(G(D_0)-G(D_{\infty}))$\ 
in the case of mixed crossing,
\item
[(d)] $aG(D_+)-a^{-1}G(D_-)=(a^{-2}-a^2)G(D_0),$\
in the case of mixed crossing.
\end{enumerate}
\end{exercise}
\begin{exercise}
Let $F^*(a,z)$ be the Dubrovnik version of the Kauffman polynomial \footnote{Kauffman described 
the polynomial $F^*$ on a postcard to Lickorish sent from Dubrovnik in September
1985. He expected that this was a new polynomial invariant of links, independent
from $F$.}.
The polynomial $F^*$ satisfies a recursive condition 
$$a^{w(D_+)}F^*_{D_+}- a^{w(D_-)}F^*_{D_-} = z(a^{w(D_0)}F^*_{D_0}-a^{w(D_{\infty})}
F^*_{D_{\infty}}).$$
Find out the value of $F^*_L(ia,i(a-a^{-1}))$. 
Find a general relation between
$F^*(a,z)$ and $F(a,z)$ (\cite {Li-3}).
\end{exercise}

\begin{exercise}
The invariant $g$ of oriented diagrams can be extended or, more precisely,
quantized (using popular parlance).
The quantization leads to the following invariant of oriented links which 
generalizes the polynomial $F_L(a,-a-a^{-1})$), namely: 
$$\hat G(a,x)=\Sigma_{S \subset L}x^{cp(S)}a^{-4lk(S,L-S)} \ . $$
Show that $\hat G(a,x)$ is sometimes a better invariant than the original
Kauffman polynomial. For example, it distinguishes 4-component links which
are presented on Fig.~3.1.  Is it true that $\hat G(a,x)$ can distinguish 
3-component links which are not distinguished by the Kauffman polynomial?
\end{exercise}


\begin{problem}\label{2:5.17} By the determinant of an oriented link
we understand a numerical invariant obtained by an appropriate evaluation
of an invariant polynomial (i.e.~either Conway, or Jones, or Jones-Conway,
or Kauffman polynomial) of the link (see Chapter IV for detailed discussion of link determinant). 
Namely, $D_L= \triangle_L(-2i) =
P_L(\frac{1}{2}i,-\frac{1}{2}i) = V_L(-1)$ (more precisely $\sqrt{t} =
-i$) $= F((-i)^{\frac{3}{2}}, -((-i)^{-\frac{1}{2}}+(-i)^\frac{1}{2}))$.
Now, if we use the variant of the Homflypt polynomial which is used
in the definition of  the supersignature 
then $D_L = r(\frac{1}{2},\frac{1}{2})$.
The signature $\sigma (L)$ is a supersignature modeled on the determinant
of the link, $\sigma (L) = \sigma_{\frac{1}{2},\frac{1}{2}}(L)$.
In the Chapter IV we will show that the signature has properties similar to
Jones Lemma  \ref{2:5.15} (also compare with Lemma \ref{2:5.10}). 
Namely, if  $L_i$ is a component of an oriented link $L$ and
$\lambda = \lk (L_i, L-L_i)$ and, moreover, if $L'$ is obtained from $L$ 
by changing the orientation of $L$, then $\sigma(L') = \sigma (L) + 2\lambda$. 
This can not be an accident and a similar formula should be true 
for other supersignatures (especially these modeled on the  Jones polynomial).
We leave this question as a research problem for the reader.
\end{problem}

Below we introduce a new equivalence relation similar to \skein
\ (Conway or skein equivalence). The new relation is motivated by a four term relation of Kauffman 
polynomial, and bound the applicability
of Kauffman method like \skein \ bounds applicability of Conway type
invariants. 

\begin{definition}\label{2:5.18} Let $S$ be the set of partially oriented diagrams
(i.e.~diagrams which may contain some oriented and some unoriented components) modulo regular
isotopy equivalence. The Kauffman equivalence relation $\sim_K$ is the 
smallest equivalence relation on the set $S$ 
which satisfies the following condition:

If $L_1'$ (resp.~$L_2'$) is a diagram representing an element 
$L_1$ (resp.~$L_2$) of $S$ which contains a crossing $p_1$ (resp.~$p_2$) 
and moreover:
\begin{enumerate}
\item[(i)] $(L_1')^{p_1}_{-\sgn p_1} \sim_K (L_2')^{p_2}_{-\sgn p_2}$ where
$L^p_{-\sgn p}$ denotes the link obtained from $L$ by interchanging 
overcrossing and undercrossing in  $p$ (this does not depend on
whether and how $L$ is oriented),

\item[(ii)] $(L_1')^{p_1}_0 \sim_K (L_2')^{p_2}_0$ and $(L_1')^{p_1}_\infty
\sim_K (L_2')^{p_2}_\infty$ provided that $p_1$ is a crossing of oriented
components of $L_1'$ or $p_1$ is a selfcrossing of a component of  $L_1'$ 
(in the latter case $(L_1')^{p_1}_0$ and $(L_1')^{p_1}_\infty$ 
are well defined independently of the orientation of the component),

\item[(iii)] $\{ (L_1')^{p_1}_0, (L_1')^{p_1}_\infty\} = \{ (L_2')^{p_2}_0,
(L_2')^{p_2}_{\infty} \}$ (equality of unordered pairs of links up to Kauffman equivalence) 
if at least one of two components of  $L_1'$, which meet in $p_1$,
is unoriented,
\end{enumerate}

then $L_1 \sim_K L_2$.
\end{definition}

Let us note that the crossings $p_1$ and $p_2$ satisfy the conditions listed
in Definition \ref{2:5.18} only if they have similar properties. For example,
if two components of $L_1$ meeting in $p_1$ are oriented then
components of $L_2$ meeting in $p_2$ are oriented as well and moreover 
$\sgn p_1 = \sgn p_2$.

\begin{exercise}\label{2:5.19}
Formulate conditions for oriented diagrams $L_1$ and $L_2$,
such regular isotopy implies regular isotopy 
(i.e.~the situation when the diagrams which are not regularly isotopic
are also not isotopic).

Hint. Consider the writhe number  $w(\ )$ of any component of the diagram
and another invariant of regular isotopy of oriented diagrams, 
which we define as follows. Let us consider a diagram $L$. 
After smoothening of all crossings of $L$ we obtain a set of oriented 
circles on the plane. Some of them are oriented positively 
({\parbox{0.5cm}{\psfig{figure=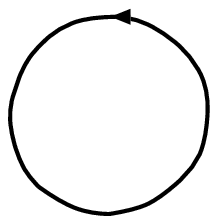,height=0.5cm}}})
some negatively 
({\parbox{0.5cm}{\psfig{figure=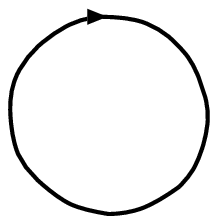,height=0.5cm}}});
the invariant in question is the result of the 
subtraction
of the number of ``negative'' circles from the number of ``positive'' circles. 
This invariant is often called the Whitney degree or rotation number (and defined for any curve or system of curves 
immersed in plane).
\end{exercise}

\begin{remark}\label{2:5.20}
In applications we do not need always such a subtle differentiation between
isotopy and regular isotopy. This is because of the following facts:
\end{remark}

\begin{exercise}\label{2:5.21}
\begin{enumerate}
\item Prove that if $L_2$ is obtained from $L_1$ by the first weakened
Reidemeister move $(R^{\pm 1}_{0.5})$ then $L_1\sim_K L_2$.

\item Prove that if $L_1$ is isotopic with $L_2$ and  $w(L_1) = w(L_2)$ then $L_1\sim_K L_2$ 
(compare Lemma \ref{2:5.6}).
\end{enumerate}
\end{exercise}

\begin{exercise}\label{2:5.22}
Prove that if an oriented diagram $L_1$ is a mutant  
of another oriented diagram $L_2$ then $L_1 \sim_K L_2$.
\end{exercise}

Now we show how invariants of links obtained by the Kauffman approach
can be described by one algebraic structure --- similarly as Conway type
invariants are described by the notion of Conway algebra.
We will also construct an example of a polynomial invariant of oriented
links which is a generalization of both, Jones-Conway and Kaufman
polynomials (however, this invariant  does not bring anything more
than these two polynomial together).
We will proceed similarly as in the case of Conway algebra, with the exception
that we will consider diagrams up to regular isotopy.
Since there is no need to distinguish the crossings of $+$ type from
the crossings of $-$ type, there will be only one operation recovering the value
of the invariant for $L_+$ (resp.~$L_-$) from its values for
$L_-$, $L_0$ and $L_\infty$ (resp.~$L_+$, $L_0$ and $L_\infty$). 
There is yet another another important remark: if a link $L_+$ is modified 
to $L_\infty$ then the new component of $L_\infty$ can not be assigned
any natural orientation, we could have, however, considered the partially 
oriented link $L$.
We do not do it for the practical reason, since in such a case
we would like to use
a method similar to the one used in the proof of Theorem
\ref{2:1.2} and for this we need the following equality:
$L^{p\ q}_{\varepsilon\ \infty} = L^{q\ p}_{\infty\ \varepsilon}$, where
$\varepsilon =+, -$ or $0$. 
This means that we would like 
the result of operations performed on two crossings to be independent
of the order of these operations. 
But this is not always the case, as is presented in Fig.~5.12.
%


\begin{center}
\begin{tabular}{c}
\includegraphics[trim=0mm 0mm 0mm 0mm, width=.65\linewidth]
{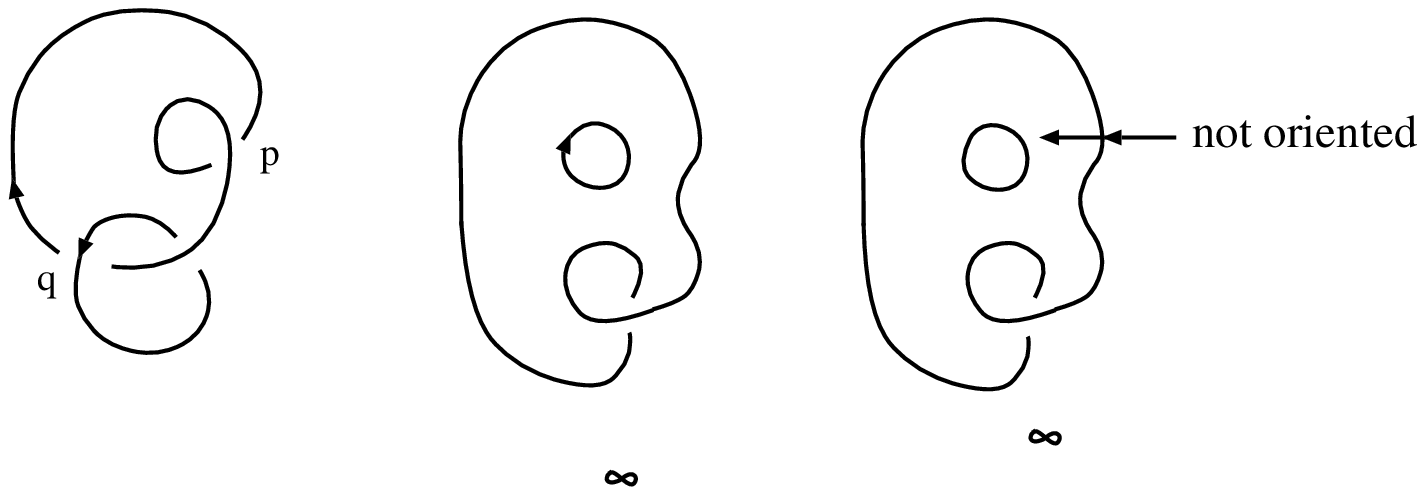}\\
\end{tabular}
\\
Fig.~5.12; diagrams $L$, $L^{pq}_{0,\infty}$ and $L^{qp}_{\infty,0}$ (last two diagrams are not the same as 
partially oriented diagrams)
\end{center}

Therefore we will restrict the consideration to the case of oriented
and unoriented link, that is we will either assume that all components
are oriented or none of them is oriented. In the latter case we will not 
distinguish $L_0$ from $L_\infty$. 

Let us consider the following general situation. Suppose that we are given
an abstract algebra  ${\cal A}$ with two universes, $A$ and $A'$,  and
with a countable (possibly finite) number 
of 0-argument operations in $A$ and $A'$, parameterized by two sequences:
$\{ a_{ij}\}_{i\in N, j\in \Z}$ and $\{ a_{ij}'\}_{i\in N, j\in \Z}$, 
respectively. Moreover, ${\cal A}$ is assumed to have 
two 3-argument operations, namely 
$* :A\times A\times A' \rightarrow A$ and
$*':A'\times A'\times A'\rightarrow A'$, and also 1-argument operation
$\phi:A\rightarrow A'$. 
We would like to construct invariants of classes of regular isotopy of
oriented and unoriented diagrams, which satisfy the following conditions:
\begin{enumerate}
\item If $L$ is an oriented diagram then the invariant $w$ is in $A$,
i.e.~$w_L\in A$ and if $L'$ is unoriented then $w_{L'}\in A'$. 

\item If $L'$ is an unoriented diagram obtained by forgetting the orientation
of an oriented diagram $L$ then $w_{L'} = \phi(w_L)$. 

\item $w_{T_{ij}} = a_{ij}$ where $T_{ij}$ denotes an oriented diagram
of a trivial link with $i$ components and the write number $w(T_{ij}) = j$.

\item $w_{T_{ij}'} = a_{ij}'$.

\item $w_{L^p} = w_{L^p_{-\sgn p}} * (w_{L^p_0},w_{L^p_\infty})$ where
$L^p_{-\sgn p}$ denotes a diagram obtained from $L$ by exchanging the tunnel
to the bridge at the crossing $p$.
\end{enumerate}

\begin{definition}\label{2:5.23}
We say that ${\cal A} =\{ A, A', \{ a_{ij}\}, \{a_{ij}'\}, *, *',
\phi\}$ is a Kauffman algebra if the following conditions are satisfied:
\begin{itemize}
\item[K1] $\phi(a_{ij}) = a_{ij}'$
\item[K2] $\phi (a*(b,c)) = \phi(a)*' (\phi(b),c)$ where the operation $*$
on $(a,b,c)$ is denoted by $a*(b,c)$ and similarly for the operation $*'$.
\item[K3] $a_{i,j-1}*(a_{i+1,j}, a_{i,j}') = a_{i,j+1}$.
\item[K4] $(a*(b,c))*(d*(e,f),g*'(h,i)) = (a*(d,g)*(b*(e,h)),c*'(f,i))$.
\item[K5] $(a*(b,c))*(b,c) = a$.
\item[K6] $a*' (b,c) = a*'(c,b)$
\end{itemize}
\end{definition}

\begin{theorem}\label{2:5.24}
For a given Kauffman algebra ${\cal A}$ there exists a uniquely
determined invariant of regular isotopy $w$ which associates an element
$w_L$ from $A$ to any oriented diagram $L$ and an element $w_{L'}$ from $A'$ 
to any unoriented diagram $L'$. Moreover, the invariant $w$ satisfies
the following conditions:
\begin{enumerate}
\item $w_{T_{i,j}} = a_{i,j}$

\item $w_L' = \phi (w_L)$, where $L$ is an oriented diagram and $L'$
is obtained from $L$ by forgetting its orientation.

\item $w_{L^p} = w_{L^p_{-\sgn p}} *(w_{L^p_0}, w_{L^p_\infty})$. 
\end{enumerate}
\end{theorem}

The proof of Theorem \ref{2:5.24} is similar to that of
Theorem \ref{2:1.2}, therefore we present only these parts of the argument
in which some differences occur.

For any diagram (oriented or not) we can build a resolving tree
such that three edges descend from any vertex which is not a leaf (Fig.~5.13)
and there are descending diagrams at the leaves (the diagrams are
descending for some choice
of base points and orientation --- in case of unoriented diagrams).

\begin{center}
\begin{tabular}{c} 
\includegraphics[trim=0mm 0mm 0mm 0mm, width=.25\linewidth]
{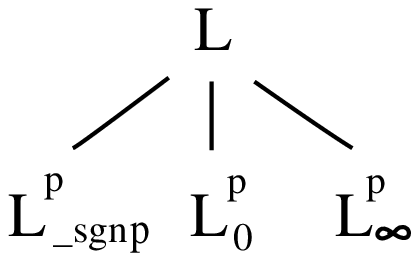}\\
\end{tabular}
\\
Fig.~5.13
\end{center}

Such a tree can be used to compute the value of the invariant at the
root diagram. We begin by constructing the invariant $w$ for
diagrams and subsequently we will show that it is not changed by Reidemeister
moves $R^{\pm 1}_{0.5}$, $R^{\pm 1}_2$ and $R^{\pm 1}_3$. 
We use induction with respect to the number of crossings in the diagram,
denoted now, as in the proof of Theorem  by $\crs{L}$.
For any $k\geq 0$ we define a function $w_k$ which associates an element from
$A$ (reps.~$A'$) to any oriented (resp.~unoriented) diagram with at most
$k$ crossings.
The invariant $w$ is then defined as  $w_L = w_k(L)$ where $k\geq \crs{L}$. 
Similarly as in the proof of Theorem \ref{2:1.2} we set
$w_0(L) = a_{n,0}$ if $L$ is a trivial oriented link with $n$ components
and we put $w_0(L') =a_{n,0}'$ if $L'$ is obtained from $L$ by forgetting
its orientation. Next we formulate the Main Inductive Hypothesis, or MIH.
We assume that we have already defined a function $w_k$ which assigns
an element from $A$ (resp.~$A'$) to any diagram $L$ with $\crs{L}\leq k$ 
and the function $w_k$ has the following properties:

\begin{property}\label{2:5.25}
Suppose that $U_{n,j}$ is an oriented diagram which is descending
for some choice of base points; moreover  $U_{n,j}$ has $n$
components, $\mbox{cr}(U_{n,j})\leq k$ and $Tait(U_{n,j}) = j$ (we use $Tait(D)$ in place of 
$w(D)$ so not to mix the writhe number of an oriented diagram with our invariant $w$).
Then $w_k(U_{n,j}) = a_{n,j}$.
Similarly, if $U_{n,j}'$ is obtained from $U_{n,j}$ 
by forgetting its orientation then $w_k(U_{n,j}') = a_{n,j}'$.   
\end{property}

\begin{property}\label{2:5.26}
$w_k(L) = w_k(L^p_{-\sgn p})*(w_k(L^p_0), w_k(L^p_\infty))$ if $L$
is an oriented diagram and $w_k(L) = w_k(L^p_{-\sgn
p})*'(w_k(L^p_0), w_k(L^p_\infty))$ if $L$ is unoriented.
\end{property}

\begin{property}\label{2:5.27}
$w_k(L) = w_k(R(L))$ where $R$ is a Reidemeister move of one of the following
types:
$R^{\pm 1}_{0.5}$, $R^{\pm 1}_2$, $R^{\pm 1}_3$ and $\crs{R(L)}\leq k$.
\end{property}

Next we want to  prove Main Inductive Step or MIS, that is, we 
want to define the function $w_{k+1}$ with appropriate properties,
which is defined for diagrams with at most $k+1$ crossings.
Similarly as in the case of Theorem \ref{2:1.2}, this will complete
the proof of Theorem \ref{2:5.24}. We begin the proof of MIS, similarly as for
\ref{2:1.2} by defining a function $w_b$ which for a diagram
$L$ with $\crs{L} = k+1$  depends on the choice of the base points 
$b=(b_1, b_2, \ldots, b_n)$ and on the choice of the orientation, if $L$ 
is unoriented. We define the function  $w_b$ by induction with respect
to the number of bad crossings, $b(L)$, and we apply the formula \ref{2:5.26} 
for the first bad crossing --- similarly as in the proof of Theorem \ref{2:1.2}). 
Subsequently, we prove that formula \ref{2:5.26} is satisfied for 
any crossing. The argument is similar to the appropriate one in the course 
of the proof of Theorem \ref{2:1.2}, only the conditions C3--C5 are replaced
by the condition K4.

The next step of the proof is to show that  $w_b$ does not depend on the choice
of base points (assuming the given order and orientation of components).
Again, we deal with this problem similarly as in \ref{2:1.2} 
and we choose base points $b$ and $b' = (b_1, b_2, \ldots, b_i', \ldots, b_n)$, 
where $b_i$ and $b_i'$ are on the opposite sides of a crossing of the
component $L_i$. Here we apply induction with respect to 
$B(L) = \mbox{max} (b(L),b'(L))$. If $B(L) = 0$ then $L$ is a descending
diagram with respect to both $b$ and $b'$ and therefore
$w_b(L) = w_{b'}(L) =a_{n,Tait(L)}$. 
If $B(L) = b(L) = b'(L) =1$ then $L$ has a common bad crossing for both
$b$ and $b'$, hence we can make  the inductive step using this crossing
(and also \ref{2:5.26}). 

Thus, we are left with the case $B(L) = 1$ and 
$b(L)\neq b'(L)$. The argument here is a little more complicated than the
argument in the respective part of the proof of \ref{2:1.2}. Namely:
let $p$ be the bad crossing of  $L$ with respect to either $b$ or $b'$,
then $p$ is a selfintersection of the component $L_i\subset L$.
For simplicity, let us assume that the diagram $L$ is oriented and 
$b(L) = 1$, $b'(L)= 0$ and $\sgn p = +$. Then $L$  is descending with respect
to $b'$, therefore $$w_{b'}(L) = a_{n,Tait(L)}.$$ Because of the property
\ref{2:5.26} we have 
$$w_b(L) = w_b(L^p_-)*(w_b(L^p_0), w_{b'}(L^p_\infty)).$$ 
Moreover $b(L^p_-) =0$ hence $w_b(L^p_-) =a_{n,Tait(L)-2}$. 
The diagram $L^p_0$  is descending for some choice of base points
it has $k$ crossings and $n+1$ 
components, hence $w_b(L^p_0) = w_k(L^p_0) = a_{n+1, Tait(L)-1}$.
Now, in order to apply the axiom $K3$ to get equality 
$w_b(L) = a_{n,Tait(L)} = w_{b'}(L)$ we
have to prove that  $w_b(L^p_\infty) = a_{n, Tait(L)-1}$. 
However, this is not immediate. 
Namely, $L^p_\infty$ does not have to be descending
with respect to any choice of base points or orientation.
We can use the fact that the diagram $L^p_\infty$ has only $k$ 
crossings and it consists of a descending part and an ascending part, and 
these parts may be put on different levels
(Fig.~5.14 illustrates the situation). Now to prove that the value of
the invariant $w$ for $L^p_\infty$ is equal $a_{n, Tait(L)-1}$ we have to
apply the following trick:

\begin{center}
\begin{tabular}{c} 
\includegraphics[trim=0mm 0mm 0mm 0mm, width=.35\linewidth]
{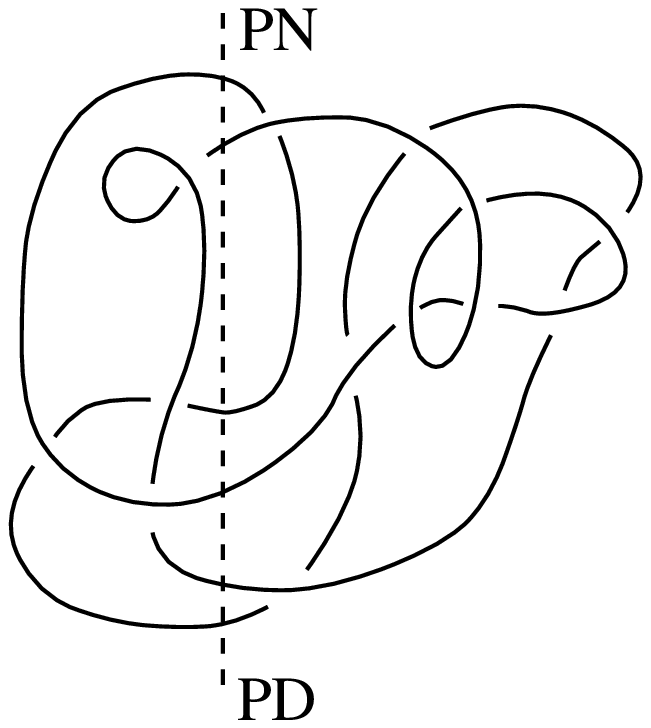}\\
\end{tabular}
\\
Fig.~5.14
\end{center}
Let us rotate the ascending part of the diagram of $L^p_\infty$ by
$180^o$ with respect to the vertical (N-S) axis i
and at the same time let us change 
the orientation of this part to the opposite (with respect to the given
orientation of $L^p_\infty$) --- this way we make some kind of a mutation.
The resulting diagram, call it $\tilde{L}$, is descending.

Therefore $w_k(\tilde{L}) = a_{n,Tait(L)-1}'$. On the other hand, we can build the
same resolving tree for $L^p_\infty$ and $\tilde{L}$ such that all
vertices of the tree correspond to diagrams with at most $k$ crossings
(there is a complete analogy with the operation of mutation).
Now, because of MIH we have 
$w_k(L^p_\infty) = w_k(\tilde{L}) = a_{n,Tait(L)-1}'$, 
which completes the  proof, c.f.~\cite{B-L-M}.

The rest of the proof of Theorem \ref{2:5.24} is almost the repetition
of the respective part of the argument in the proof of Theorem
\ref{2:1.2}. 
We change the Reidemeister move $R^{\pm 1}_1$ to
$R^{\pm 1}_{0.5}$. The proof of the independence of the order of components
in $w_b$ is the same as in \ref{2:1.2} because in Lemma
\ref{l:2.15} Reidemeister move $R^{\pm 1}_1$ can be replaced
by $R^{\pm 1}_{0.5}$. If $L$ is not oriented then applying Lemma
\ref{l:2.15} we show that the definition of $w_b$ does not depend
on the choice of the orientation that we have made.

This completes the proof of Theorem \ref{2:5.24}.

\begin{example}[Jones-Conway-Kauffman polynomial]\label{2:5.28}
\ \\
Let us consider the following algebra ${\cal A}$:
\\
$A = \Z[a^{\pm 1}, t^{\pm 1}, z]$,\ \ $A' = \Z[a^{\pm 1}, t^{\pm 1}]$, \\
$a_{i,j} = (\frac{a^{-1}+a}{t})^{i-1}(1-\frac{z}{t})a^j
+\frac{z}{t}(\frac{a^{-1}+a}{t}-1)^{i-1}a^j$,\\
$a_{i,j}'=(\frac{a^{-1}+a}{t}-1)^{i-1}a^j$, \\
moreover $b*(c,d)$ is defined by the equation $b*(c,d)+b = tc+zd$ 
and $b*'(c,d)$ is defined by the equation
$b*'(c,d)+b = tc +td$; the 1-argument operation
$\phi:A\rightarrow A'$ is defined on generators
$\phi (a) = a, \phi(t) = t, \phi(z) =t$. 

Now we will check that ${\cal A}$ is a Kauffman algebra.
The conditions K1, K2, K5 and K6 follow immediately by the definition
of ${\cal A}$. The condition K3 follows by the equality:
\begin{eqnarray*}
&(\frac{a^{-1}+a}{t})^{i-1}(1-\frac{z}{t})a^{j+1}&+\\
+&\frac{z}{t}(\frac{a^{-1}+a}{t}-1)^{i-1}a^{j+1}&+\\
+&(\frac{a^{-1}+a}{t})^{i-1}(1-\frac{z}{t})a^{j-1}&+\\
+&\frac{z}{t}(\frac{a^{-1}+a}{t}-1)^{i-1}a^{j-1}&=\\
&t((\frac{a^{-1}+a}{t})^{i}(1-\frac{z}{t})a^j
+\frac{z}{t}(\frac{a^{-1}+a}{t}-1)^{i}a^j)&+\\
+&z(\frac{a^{-1}+a}{t}-1)^{i-1}a^j&\\
\end{eqnarray*}

It remains to check the condition K4:
\begin{eqnarray*}
&(a*(b,c))*(d*(e,f),g*'(h,i))& =\\
=&-(a*(b,c)+t(d*(e,f))+z(g*'(h,i))&=\\
=&-(-a+tb+zc)+t(-d+te+zf)+z(-g+th+ti)&=\\
=&a-tb-zc-td+t^2e+tzf-zg+zth+zti&=\\
=&(a*(d,g))*(b*(e,h),c*'(f,i)).&\\
\end{eqnarray*}

The polynomial invariant of regular isotopy of oriented or unoriented diagrams
defined by the algebra ${\cal A}$ is called Jones-Conway-Kauffman polynomial
and it is denoted by  $J_L(a,t,z)$.  It can be modified
to an invariant of oriented links by setting
$\tilde{J}_L(a,t,z) = J_L(a,t,z)a^{-Tait(L)}$. 
\end{example}

\begin{example}\label{2:5.29}
Let us compute the Jones-Conway-Kauffman polynomial for a right-hand-side trefoil 
knot which is presented on Fig.~5.15. If we apply the resolving tree presented
on Fig.~5.15 then we will get the following relation in any Kauffman algebra:
$w_L =a_{1,1}*(a_{2,0}*(a_{1,1},a_{1,-1}'),a_{1,-2}')$.

\begin{center}
\begin{tabular}{c} 
\includegraphics[trim=0mm 0mm 0mm 0mm, width=.65\linewidth]
{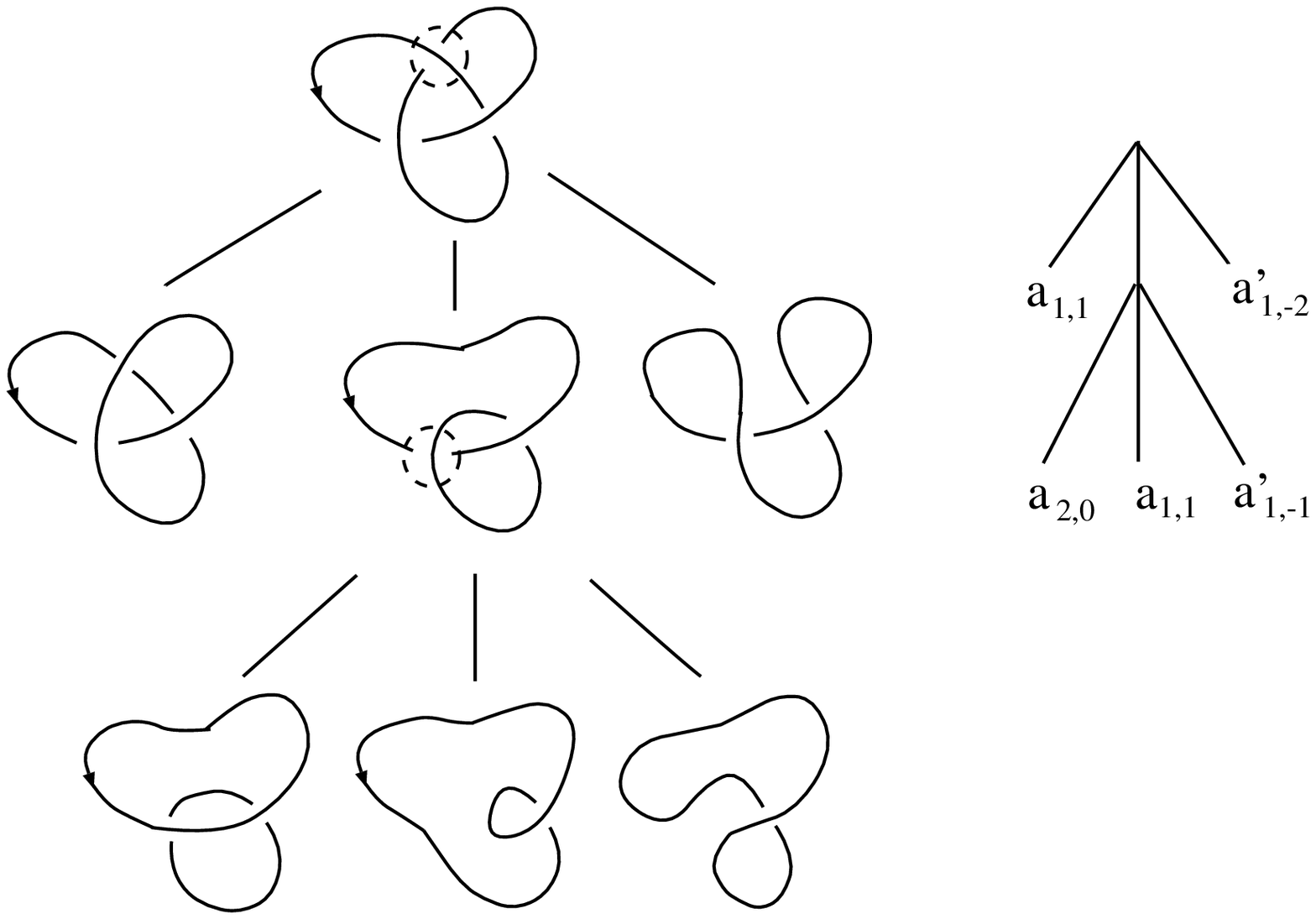}\\
\end{tabular}
\\
Fig.~5.15
\end{center}

Thus we get:
\begin{eqnarray*}
J_L(a,t,z) =&-a+t(\frac{a^{-1}+a-z}{t}+ta+za^{-1})+za^{-2}&=\\
&-a^{-1}-2a+t^2a+z(1+a^{-2}+ta^{-1}).&\\
\tilde{J}_L(a,t,z) = &-a^{-4}-2a^{-2}+t^2a^{-2}+z(a^{-3}+a^{-5}+ta^{-4}.&\\
\end{eqnarray*}
\end{example}

\begin{lemma}\label{2:5.30}
\begin{enumerate}
\item[(1)] $J_L(a,t,z) = J_L(a,t,0) +z(\frac{J_L(a,t,t)-J_L(a,t,0)}{t})$.
\item[(2)] $\tilde{J}(a,t,0) = P(\frac{a}{t},\frac{1}{at})$, 
thus $\tilde J$ is a version of the Homflypt polynomial.
\item[(3)] $J_L(a,t,t) = G_L(a,t)$, so $J$ is the Kauffman polynomial for regular
isotopy.
\end{enumerate}
\end{lemma}

Proof.
The property (1) is obviously true for diagrams representing trivial links.
Next we apply induction with respect to the length of the resolving tree
of the diagram.

To prove (2) and (3) it is enough to check the initial conditions and to
compare the operations $*$ which are used in the respective definitions
of polynomial invariants.

Therefore the Jones-Conway-Kauffman polynomial is equivalent
to Jones-Conway (Homflypt) and Kauffman polynomials taken together. In this context, there is 
a remarkable similarity with the situation of Exercise \ref{Exercise III:3.43}.

\begin{lemma}\label{2:5.31}
The invariants defined by Kauffman algebras are invariants of 
$\sim_K$ equivalence of oriented or unoriented diagrams.
\end{lemma}

The proof is immediate.

\begin{remark}\label{2:5.32}
The theory of invariants defined by Kauffman algebras can be developed
in parallel  to the theory of invariants given by Conway algebras.
In particular:
\begin{enumerate}
\item One may look for a pair of involutions 
$\tau$ on $A$ and $\tau '$ on $A'$ such that
$\tau(a_{i,j}) = a_{i,-j}$, 
$\phi(\tau(w)) = \tau '(\phi(w))$ and
$\tau(a*(b,c)) =\tau(a)*(\tau(b),\tau '(c))$. Then $A_{\overline{L}} =
\tau (A_L)$ where  $A_L$ is an invariant of an oriented diagram $L$
and $\overline{L}$ is the mirror image of  $L$ (compare with Lemma
\ref{2:3.22}). 
\item It is possible to build a universal Kauffman algebra (the elements 
of which are terms) and to show that for such a universal Kauffman algebra
the involutions $\tau$ and $\tau '$ exist.
\item It is reasonable to look for an operation 
$ o:A\times A\times A'\rightarrow A$
which, for given oriented diagrams, would recover the value of the invariant
for $L_0$ from its values for $L_+$, $L_-$ and $L_\infty$. 
\item One may look for conditions for a Kauffman algebra which should
provide a simple formula for the value of the associated invariant
for connected sum and disjoint sum of diagrams.
\item One may look for conditions for a Kauffman algebra which would allow
a simple modification of the associated invariant of regular isotopy to
an invariant of ambient isotopy of diagrams. 
For example, if there exist bijections $\beta: A\rightarrow A$ and $\beta
':A'\rightarrow A'$ satisfying $\beta (a_{i,j}) = a_{i,j-1}$,
$\phi (\beta(a)) = \beta '(\phi (a))$ and $\beta (a*(b,c)) =
\beta(a)*(\beta(b),\beta '(c))$ then $\underbrace{\beta( \beta(
\ldots\beta(A_L)\ldots))}_{Tait(L) \mbox{ times }}$ is an invariant
of ambient isotopy.
\item One may consider geometrically sufficient partial Kauffman algebras
(similarly to the case of Conway algebras,  see Definition \ref{d:4.2})
which would define invariants of regular isotopies of diagrams.

\item Similarly as in Example \ref{2:4.5}, one may consider 
polynomials of infinite number of variables which generalize 
the Jones-Conway-Kauffman polynomial.

\item It is possible to show that the invariants defined by 
geometrically sufficient partial Kauffman algebras are preserved by 
mutations and, further, that mutations preserve classes of
$\sim_K$ equivalence of diagrams.
\end{enumerate}

 We leave a possible extension of ideas presented above for the reader.
 
\end{remark}

Many problems which we have
formulated for invariants of Conway type can be
extended for invariants obtained via the Kauffman method.

\begin{problem}\label{2:5.33}
\ \\
\begin{enumerate}
\item Do there exist two oriented diagrams which have the same 
Jones-Conway-Kauffman polynomial and which can be distinguished
by another invariant obtained from some Kauffman algebra?

\item Do there exist two oriented diagrams which can not be distinguished
by any invariant defined by  a 
Kauffman algebra but  they can be distinguished
by invariants coming from a geometrically sufficient partial Kauffman algebra?

\item Do there exist two oriented diagrams which are not $\sim_K$ equivalent
and which can not be distinguished by any invariant defined by a geometrically
sufficient partial Kauffman algebra?

\item If an oriented diagram of a knot, $L$, satisfies $L\sim_K\overline{L}$
is it true then that $L$ is isotopic to either $\overline{}L$ or to
$-\overline{L}$?
\end{enumerate}

It seems that a positive answer for (2) can be obtained by applying signature.
The question (4) is a weakened version of Kauffman's conjecture
\ref{III:5.11}.
\end{problem}

\begin{exercise}\label{2:5.34}
Let us consider moves of diagrams: $R_{0.1}$ and $R_{0.2}$, as shown on 
Fig.~5.16. Let $L$ be a diagram with $k$ crossings. Prove that there exists
a choice of base points $b$ of $L$ such that the descending diagram $L^d$,
associated  to $L$ and $b$, can be modified to a diagram with fewer than
$k$ crossings via a sequence of moves consisting of
$R_{0.1}$, $R_{0.2}$, $R^{\pm 1}_{2}$ and $R^{\pm 1}_3$, 
which do not increase the number of crossings.

\begin{center}
{\psfig{figure=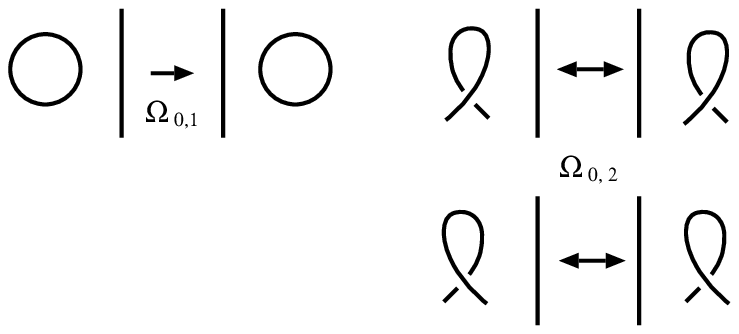,height=6.3cm}}
\end{center}
\begin{center}
Fig.~5.16
\end{center}

Hint. See Lemma \ref{l:2.15}.
\end{exercise}

\begin{problem}\label{III:5.39}
When we were defining invariants of diagrams via Kauffman algebras, or
when we were defining the relation $\sim_K$, we had a problem with 
a natural orientation for the whole diagram  $L^p_\infty$. 
This was because the new component of $L^p_\infty$ was formed
from the pieces of $L$ which had opposite orientations.
(Fig.~5.17).
\begin{center}
{\psfig{figure=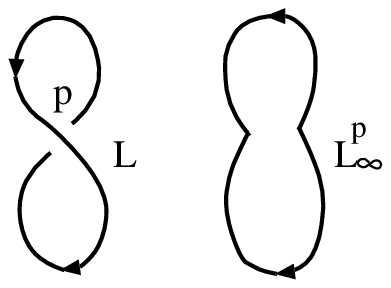,height=4.3cm}}
\end{center}
\begin{center}
Fig.~5.17
\end{center}

Therefore, it seems to be reasonable to consider diagrams, the components 
of which can have different orientations, i.e.~any component is divided 
into arcs and each arc has its orientation. Even a simple diagram 
(Fig.~5.18) presents new difficulties which we have to consider 
(resolve the diagram first starting from $p$, and then from $q$). 
The author tried to compute 
the new invariant and his computations show that the problem is difficult but
not hopeless. We leave it as a research problem for the reader.

\begin{center}
{\psfig{figure=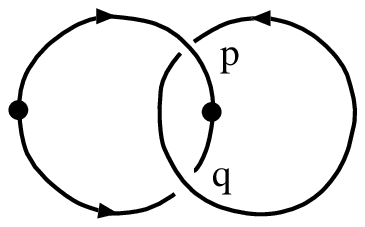,height=4.3cm}} 
\end{center}
\begin{center}
Fig.~5.18
\end{center}

\end{problem}


\end{document}